\documentclass[11pt,a4paper,twoside]{article}
\usepackage[hmargin=1.25in,vmargin=1.25in,centering]{geometry}
\usepackage{hyphenat}
\usepackage[nottoc]{tocbibind}

\usepackage[biblabel,nosort,nocompress]{cite}

\usepackage{graphicx}
\usepackage{amsmath,amsthm}
\usepackage{amssymb}

\usepackage{esint}
\usepackage{mathrsfs}

\usepackage{euler}
\newcommand*{\laplace}{\mathop{}\!\vb*{\Delta}}

\usepackage{biolinum}
\usepackage{garamondlibre}

\usepackage{graphicx}
\usepackage{mathtools}
\usepackage{cases}
\usepackage{xfrac}

\usepackage{physics}

\usepackage{enumitem}

\usepackage[dvipsnames]{xcolor}
\usepackage[colorlinks=true, pdfstartview=FitV, linkcolor=RoyalBlue,citecolor=ForestGreen, urlcolor=BurntOrange]{hyperref}
\usepackage{fancyhdr}
\renewcommand\footnotemark{}
\usepackage[sf,sl,outermarks]{titlesec}
\titleformat{\section}
{\Large\bfseries\sffamily}{\filcenter\Large{\thesection.}}{1em}{\filcenter}
\titleformat{\subsection}
{\large\bfseries\sffamily}{\thesubsection.}{1em}{}
\titleformat{\subsubsection}[runin]
{\normalsize\bfseries\itshape\sffamily}{{\normalfont\bfseries\itshape\S}\thesubsubsection.}{0.5em}{}[.\hspace*{0.5ex}]
\titleformat{\paragraph}[runin]
{\normalsize\bfseries\itshape\sffamily}{{\normalfont\bfseries\itshape\S}\theparagraph.}{0.5em}{}[.\hspace*{0.5ex}]

\usepackage{abstract}

\theoremstyle{definition}
\newtheorem{defi}{\sffamily Definition}[section]

\theoremstyle{plain}
\newtheorem{theorem}[defi]{\sffamily Theorem}
\newtheorem{prop}[defi]{\sffamily Proposition}
\newtheorem{lemma}[defi]{\sffamily Lemma}
\newtheorem{cor}[defi]{\sffamily Corollary}

\theoremstyle{remark}
\newtheorem*{remark}{\sffamily Remark}

\makeatletter
\providecommand{\proofnamestyle}{\itshape\sffamily\bfseries}
\renewenvironment{proof}[1][\proofname]{\par
	\pushQED{\qed}%
	\normalfont \topsep6\p@\@plus6\p@\relax
	\trivlist
	\item\relax
	{\proofnamestyle
		#1\@addpunct{.}}\hspace\labelsep\ignorespaces
}{%
	\par \popQED\endtrivlist\@endpefalse
}
\makeatother

\newcommand*{\wt}[1]{\widetilde{#1}}
\newcommand*{\R}{\mathbb{R}}

\providecommand\given{} 
\newcommand\SetSymbol[1][]{\nonscript\:#1\vert \allowbreak \nonscript\: \mathopen{}}
\DeclarePairedDelimiterX\Set[1]\{\}{ \renewcommand\given{\SetSymbol[\delimsize]} #1 }

\DeclareMathOperator*{\spt}{\operatorname{spt}}
\DeclareMathOperator*{\dist}{\operatorname{dist}}
\DeclareMathOperator*{\Div}{\operatorname{div}}

\DeclareMathOperator*{\Curl}{\operatorname{curl}}
\newcommand*{\const}{\operatorname{const.}}
\newcommand*{\Dt}{\mathop{}\!\vb{D}_t}

\newcommand*{\pd}{\partial}

\newcommand*{\vv}{{\vb{v}}}
\newcommand*{\vw}{{\vb{w}}}
\newcommand*{\Omt}{\Omega_t}
\newcommand*{\Gmt}{\Gamma_t}
\newcommand*{\cs}{c_{\text{s}}}

\newcommand*{\Blin}{{B_{\textup{lin}}}}
\newcommand*{\Elin}{{E_{\textup{lin}}}}
\newcommand*{\wtC}{{\widetilde{C}}}
\newcommand*{\Lip}{\textup{Lip}}
\newcommand*{\err}{\textup{err.}}
\newcommand*{\Ombd}{{\Omega_{t}^{\textup{bdry}}}}
\newcommand*{\calC}{\mathcal{C}}
\newcommand*{\scrJ}{\mathscr{J}}
\newcommand*{\calM}{\mathcal{M}}
\newcommand*{\bbH}{\mathbb{H}}
\newcommand*{\fkE}{\mathfrak{E}}
\newcommand*{\vom}{\vb*{\omega}}
\newcommand*{\cL}{\operatorname{\mathscr{L}}}

\newcommand*{\e}{\vb*{e}}
\newcommand*{\cE}{\mathcal{E}}
\newcommand*{\fka}{\mathfrak{a}}
\newcommand*{\scD}{{\mathscr{D}}}
\newcommand*{\cD}{\mathcal{D}}
\newcommand*{\fD}{{\mathfrak{D}}}
\newcommand*{\cH}{{\mathcal{H}}}

\newcommand*{\vn}{{\vb{N}}}
\newcommand*{\vbu}{{\vb{u}}}
\newcommand*{\Order}[2]{\mathcal{O}_{#1}({#2})}

\newcommand*{\fks}{\mathfrak{s}}
\newcommand*{\fkw}{\vb*{\mathfrak{w}}}
\newcommand*{\fkf}{\mathfrak{f}}
\newcommand*{\fkg}{\vb*{\mathfrak{g}}}
\newcommand*{\fkr}{\mathfrak{r}}
\newcommand*{\fkR}{\vb*{\mathfrak{R}}}
\newcommand*{\vps}{\vb*{\Psi}}
\newcommand*{\fkH}{\mathfrak{H}}
\newcommand*{\fkF}{\mathfrak{F}}
\newcommand*{\fkG}{\vb*{\mathfrak{G}}}
\newcommand*{\fkS}{\mathfrak{S}}
\newcommand*{\fkW}{\vb*{\mathfrak{W}}}

\newcommand*{\ala}{\frac{\alpha-1}{2}}
\newcommand*{\alb}{\frac{\alpha}{2}}
\newcommand*{\Hala}{H^{0, \ala}}

\newcommand*{\Omin}{\Omega_{0}^{\text{in}}}
\newcommand*{\kk}{k+1-\varkappa_0}
\numberwithin{equation}{section}

\begin{document}
	\title{%
		\bfseries\sffamily
		Physical Vacuum Problems for the Full Compressible {E}uler Equations:  Low-regularity {H}adamard-style Local Well-posedness %
		\footnote{{\bf\sffamily Date: }\today.}
	}
	\author{%
		\textsc{Sicheng LIU} 
		\thanks{%
			\textbf{\sffamily Sicheng LIU (\textit{corresponding author}): } Department of Mathematics, University of Macau, Taipa, Macao, China.
			Email: \href{mailto:scliu@link.cuhk.edu.hk}{\nolinkurl{scliu@link.cuhk.edu.hk}}. 
		}
		\and
		\textsc{Tao LUO}
		\thanks{%
			\textbf{\sffamily Tao LUO: } Department of Mathematics, City University of Hong Kong, Kowloon, Hong Kong, China.
			Email: \href{mailto:taoluo@cityu.edu.hk}{\nolinkurl{taoluo@cityu.edu.hk}}.
		}
	}
	\date{
		} 
	\maketitle
	\begin{abstract}  
		This manuscript concerns the dynamics of non-isentropic compressible Euler equations in a physical vacuum. We establish the Hadamard-style local well-posedness in low-regularity weighted Sobolev spaces, where the gas-vacuum interface is allowed to have unbounded curvature, demonstrating existence, uniqueness, and continuous dependence on initial data.  Additionally, we prove sharp a priori energy estimates and continuation criteria.
		
		The approach is based on the framework of Eulerian coordinates, avoiding the regularity issues of the flow map and the high nonlinearity induced by the Lagrangian transformation.
	\end{abstract}
	{\sffamily\tableofcontents}
	
	\addtocontents{toc}{\protect\setcounter{tocdepth}{2}}	
	\section{Introduction}
	The non-isentropic compressible Euler equations in gas dynamics can be written as:
	\begin{equation}
		\begin{cases*}
			\pdv*{\rho}{t} + \Div\qty(\rho\vv) = 0,  \\
			\pdv*{(\rho \vv)}{t} + \Div\qty(\rho\vv \otimes \vv) + \grad p = \vb*{0},  \\
			\pdv*{(\rho S)}{t} + \Div\qty(\rho S \vv) = 0.
		\end{cases*}
	\label{euler}	
	\end{equation}
	Here, $\rho$ represents the density, $\vv$ is the velocity field, $p$ denotes the pressure, and $S$ is the specific entropy of the gas. We assume that the gas obeys a polytropic process, meaning that the pressure is given by the state equation:
	\begin{equation}\label{state eqn}
		p = \rho^{1+\beta} \e^{S} \qc \const \beta > 0,
	\end{equation}
	where $(1+\beta)$ is the adiabatic index.	We shall study the vacuum free boundary problems in this manuscript. Specifically, the gases are assumed to be contained within a time-dependent set $\Omt \subset \R^d \, (d \ge 1)$ having a free boundary denoted by $\Gmt \coloneqq \pd\Omt$. The boundary condition for gas dynamics, as opposed to hydrodynamics, is that the density vanishes at the free boundary, i.e.,
	\begin{equation*}
		\Gmt = {\Set*{x \given \rho (t, x) = 0}} \cap \overline{\Set*{x \given \rho(t, x) > 0}}.
	\end{equation*}
	In this manuscript, we will consider the \emph{physical vacuum} regime for the one-phase free boundary problems, which is characterized by the normal acceleration of the free boundary $\Gmt$ being bounded both from below and from above. Such scenarios naturally arise in the study of Euler's equations with damping \cite{liu1996comp}, gaseous stars \cite{luoxinzeng2014well}, and shallow water waves \cite{lannes2020model}. More precisely, let $\Dt$ denote the \emph{material derivative} along the particle paths:
	\begin{equation*}
		\Dt \coloneqq \pdv{t} + (\vv\vdot\grad{}),
	\end{equation*}
	where $\vv$ is the velocity field of the gas. This derivative represents the rate of change of a quantity as it moves with the flow, which is crucial for analyzing the dynamics. In this way, the compressible Euler system \eqref{euler} can be rewritten as
	\begin{equation*}
		\begin{cases*}
			\Dt \rho + \rho \divergence\vv = 0, \\
			\Dt \vv + \rho^{-1}\grad p = \vb*{0}, \\
			\Dt S = 0.
		\end{cases*}
	\end{equation*}
	The state equation \eqref{state eqn} implies that
	\begin{equation*}
		\rho^{-1}\grad{p} = \e^{\frac{S}{1+\beta}}\grad(\frac{1+\beta}{\beta}p^{\frac{\beta}{1+\beta}}).
	\end{equation*}
	In particular, as long as the specific entropy $S$ is uniformly bounded, the physical vacuum boundary condition reads
	\begin{equation}\label{bdry con phy vac}
		\rho = p = 0 \qand 0 < \abs{\grad(p^{\frac{\beta}{1+\beta}})} < \infty \qq{on} \Gmt.
	\end{equation}
	The \emph{sound speed} $\cs$ of the gas is defined through:
	\begin{equation*}
		\cs^2 \coloneqq \pdv{p}{\rho} = (1+\beta)\rho^{\beta}\e^{S} = (1+\beta)\e^{\frac{S}{1+\beta}}p^{\frac{\beta}{1+\beta}}.
	\end{equation*}
	In particular, the uniform boundedness of the specific entropy $S$ implies that the sound speed $\cs$ satisfies the following decay rate when approaching the vacuum boundary $\Gmt$:
	\begin{equation}\label{sound speed decay rate}
		\cs^2(x) \simeq \dist\qty(x, \Gmt) \qfor x\in\Omt \text{ with } \dist\qty(x, \Gmt) \ll 1.
	\end{equation}
	
	\subsection{Backgrounds and Related Works}
	The mathematical study of the compressible Euler equations has a rather long history, dating back to the foundational work of Leonhard Euler in the 18\textsuperscript{th} century. These equations are pivotal in describing the motion of a compressible fluid, such as a gas, and form a cornerstone of fluid dynamics. They encapsulate the principles of conservation of mass, momentum, and energy, providing a comprehensive framework for understanding fluid behaviors under various conditions. Over the years, significant progress has been made in analyzing these equations, including the understanding of the existence, uniqueness, and stability of their solutions.
	
	\subsubsection{Well-posedness of compressible Euler equations}
	In the absence of vacuum regions, the compressible Euler equations are routinely treated as a symmetric hyperbolic system, whose local well-posedness is well-established (see, for instance, \cite{kato1975cauchy,majda1984book,Ifrim-Tataru2023}). This local well-posedness can be demonstrated in the entire space $\R^d$ using energy methods, provided that the initial data belong to the Sobolev class $H^s(\R^d)$ with $s > 1 + \frac{d}{2}$. Recent findings suggest that these Sobolev indices can be further refined. For the isentropic irrotational flows, the compressible Euler system can be reformulated into nonlinear wave equations, whose sharp local well-posedness results are available (cf. \cite{smithtataru2005sharp,wang2017geo}). There have also been recent advances in understanding low-regularity solutions to rotational flows (see \cite{wang2022rough,disconzietal2022rough,andersson2022well}). Additionally, very recent research \cite{wang2024global} has addressed global solutions for irrotational flows.
	
	\subsubsection{Vacuum problems for ideal compressible flows}
	A vacuum state refers to regions where the fluid density is zero, which can occur naturally, such as in study of astrophysics, or in engineered systems where gases are evacuated. Understanding vacuum problems has practical applications in various fields, including astrophysics (e.g., star formation, supernova explosions, and the behavior of interstellar gases) and engineering (e.g., high-speed aerodynamics, vacuum technologies, and aerospace engineering). Studying such problems involves understanding how the fluid behaves near these regions, which can be challenging due to the degeneracy and singularity of the equations near vacuum states. Another difficulty is that the vacuum boundary (the interface separating the gas from the vacuum) evolves over time. The insights gained from the analysis of these problems can lead to advancements in various engineering fields and improve our understanding of natural phenomena.
	
	The mathematical investigation of vacuum states began with {\sc Liu and Smoller} \cite{LIU1980345}. The notion of physical vacuums is particularly motivated by the work of {\sc Liu} \cite{liu1996comp} in the study of isentropic compressible flows with damping. A critical aspect of the mathematical analysis is the behavior of the sound speed near the vacuum boundary. For clarity and brevity, we presently focus on the barotropic flows as an example, where the compressible Euler system and the state equation can be expressed as follows:
	\begin{equation*}
		\begin{cases*}
			\pdv*{\rho}{t} + \Div\qty(\rho\vv) = 0, \\
			\pdv*{(\rho\vv)}{t} + \Div(\rho\vv\otimes\vv) + \grad{p} = \vb*{0}, \\
			p = \rho^{1+\beta}.
		\end{cases*}
	\end{equation*}
	Recall that the sound speed $\cs$ is given by
	\begin{equation*}
		\cs^2 \coloneqq \pdv{p}{\rho} = (1+\beta)\rho^{\beta}.
	\end{equation*}
	Then, the Euler equations can be rewritten in the form of a symmetric hyperbolic system:
	\begin{equation*}
		\begin{cases*}\displaystyle
			\pdv{\overline{c}}{t} + (\vv\vdot\grad)\overline{c} + \frac{\beta}{2}\overline{c}\div{\vv} = 0, \\[0.5em]
			\displaystyle
			\pdv{\vv}{t} + (\vv\vdot\grad)\vv + \frac{\beta}{2}\overline{c}\grad{\overline{c}} = \vb*{0},
		\end{cases*}
	\end{equation*}
	where $\overline{c} \coloneqq (2\cs)/\beta$ is the renormalized sound speed. In particular, there are two major physical scenarios of the vacuum problems, depending on the behavior of the sound speed $\cs$ near the boundary.
	\begin{itemize}
		\item \textbf{Compressible liquids:}\quad The pressure $p$ is a non-vanishing constant on the vacuum boundary, indicating the balance of momentum. In this scenario, the sound speed $\cs$ converges to a non-zero constant as it approaches the boundary.
		\item \textbf{Gases:}\quad The density and, consequently, the sound speed $\cs$ converge to zero nearing the vacuum boundary.
	\end{itemize}
	The local well-posedness theories for the liquid-vacuum free interface problems in the context of the compressible Euler equations have been extensively studied. For detailed discussions, see \cite{lindblad2005,Trak09CPAM,lindbladluo2018apriori,luo2018motion,ginsberglindbladluo2020local,luozhang2022local,wangzhangzhao2022well} and the references therein.
	
	Gas flows can also be categorized into two cases based on whether the particle acceleration $\cs\grad{\cs}$ vanishes at the vacuum boundary. Specifically, the decay rate of the sound speed $\cs$ relative to the distance $d_{\Gmt}$ significantly influences the mathematical analysis. Assuming the decay of $\cs$ follows a power law with respect to $d_{\Gmt}$, i.e.,
	\begin{equation*}
		\cs \simeq (d_{\Gmt})^{\lambda} \qq{when} d_{\Gmt}(x) \ll 1,
	\end{equation*}
	it is natural to distinguish the following scenarios:
	\begin{itemize}
		\item \textbf{Rapid decay:}
		\begin{equation*}
			\lambda \ge 1.
		\end{equation*}
		In this case, the flow can heuristically be regarded as smooth up to the boundary. With the vanishing acceleration $\cs\grad{\cs} = \vb*{0}$ on $\Gmt$, particles at the front move freely, and internal waves cannot reach the boundary arbitrarily fast. This geometric structure will persist for at least a short duration, allowing the solutions to be constructed using classical methods \cite{makinoukaikawashima1986sur,chemin1990dyna}. For the one-dimensional analysis, see \cite{LiuYang1997comp,LiuYang2000comp, Xu-Yang2005, yang2006sing}. Moreover, the results by {\sc Liu and Yang} \cite{LiuYang1997comp} indicate that $\cs^2$ loses smoothness  after a finite time, which disrupts these structures.
		\item \textbf{Slow decay:}
		\begin{equation*}
			0 < \lambda < 1.
		\end{equation*}
		Here, the motion of wave fronts is strongly coupled with that of the internal waves.	The scenario $\lambda = \frac{1}{2}$ corresponds to the \emph{physical vacuum} regimes. If $0 < \lambda < \frac{1}{2}$ or $\frac{1}{2} < \lambda < 1$, the flow is anticipated to be unstable (see {\sc Jang and Masmoudi } \cite{jangmasmoudi2012well} and {\sc Jang, Liu, and Masmoudi} \cite{jang2025waiting}) , although the rigorous proofs for general data are still open.
	\end{itemize}
	We refer to {\sc Serre} \cite{serre2015expansion} and the surveys \cite{yang2006sing,jang-masmoudi2011vacuum,luosmoller2011on,luozeng2020} for more detailed discussions.
	
	\subsubsection{Physical vacuum problems for ideal gases}	
	Regarding the isentropic compressible Euler equations in a physical vacuum, where the specific entropy $S$ is constant and the local sound speed satisfies $\cs^2 \simeq d_{\Gmt}$, the local well-posedness was first established for the 1D case by {\sc Jang and Masmoudi} \cite{jang-masmoudi2009well} and {\sc Coutand and Shkoller} \cite{coutand-shkoller2011well}. The results for the 3D case were later proven by {\sc Jang and Masmoudi} \cite{jang-masmoudi2015well} and {\sc Coutand and Shkoller} \cite{coutand-shkoller2012well} using different approaches. (See also \cite{coutand-lindblad-shkoller2010apriori} for a priori estimates. For the well-posedness theory of 3D spherically symmetric motions in weighted low-regularity Sobolev spaces, without requiring compatibility conditions on the derivatives at the center of symmetry, we refer to \cite{luoxinzeng2014well}. The same work \cite{luoxinzeng2014well} also established a general uniqueness theorem for classical solutions in 3D when $0<\beta\le 1$, where $(1+\beta)$ is the adiabatic constant,  and the gas-vacuum interface is regular with bounded curvature.) All these results (except the uniqueness in \cite{luoxinzeng2014well}) relied on the Lagrangian coordinate framework, which transforms the free boundary problems into highly nonlinear fixed-boundary problems. Recently, {\sc Ifrim and Tataru} \cite{Ifrim-Tataru2024} demonstrated the Hadamard-style local well-posedness and continuation criteria in the Eulerian coordinates for all space dimensions, with results applicable to low-regularity classes. For global-in-time solutions, we refer to {\sc Had\v{z}i\'{c} and Jang} \cite{hadzic-jang2018expanding} and {\sc Shkoller and Sideris} \cite{shokoller-sideris2019global} for those near expanding affine motions. Non-isentropic problems were studied by {\sc Geng, Li, Wang, and Xu} \cite{geng-li-wang-xu2019well} (1D local well-posedness) and {\sc Rickard, Had\v{z}i\'{c}, and Jang} \cite{rickard2021global} (3D global existence near affine solutions), both formulated in the Lagrangian framework.
	
	In the case of compressible isentropic Euler equations with damping, {\sc Liu} \cite{liu1996comp} constructed a family of particular solutions time-asymptotically equivalent to the Barenblatt self-similar solutions of the porous medium equations (simplified via Darcy’s law). The global existence of smooth solutions and their time-asymptotic convergence to Barenblatt profiles for general initial data near these solutions were proven in {\sc Luo and Zeng} \cite{LUOZENG2016} (1D) and {\sc Zeng} \cite{ZENG2017} (3D, spherical symmetry); see also \cite{ZENG2023} for precise time-asymptotics of vacuum boundaries. Furthermore, {\sc Zeng} \cite{ZENG2022} established almost global existence for 3D perturbations without symmetry assumptions.	For compressible isentropic Euler-Poisson equations modeling gaseous stars, global existence of expanding solutions was shown by {\sc Had\v{z}i\'{c} and Jang} \cite{HADJANG2018} (3D, spherical symmetry) and {\sc Had\v{z}i\'{c}, Jang, and Lam} \cite{HADJANG2024} (without symmetry).
	
	\subsubsection{Main motivations}
	Given these significant advances, the study of vacuum dynamics for high-dimensional non-isentropic Euler equations remains an open problem. This is particularly so for the well-posedness of physical vacuum problems in low-regularity functional spaces allowing  gas-vacuum interfaces to have unbounded curvature.  In many physical scenarios, the specific entropy of gas cannot be treated as constant, making the study of non-isentropic problems essential. Because the specific entropy satisfies a transport equation, in the Lagrangian setting, the entropy term can be fully determined through the initial data and flow maps. Actually, in the Lagrangian framework, all thermal dynamical information (including densities, velocity fields, and entropy) can be reflected by the initial data and flow maps, which is a prominent advantage of such methods. On the other hand, as the flow map is an integral of velocity fields over a time interval, the spatial regularity of flow maps is the same as that of the velocity fields. However, as indicated in \cite[\S5.2]{Shatah-Zeng2008-Geo},  the flow maps are usually less smooth than the free boundary. Correspondingly, by scaling analysis (cf. \S\ref{sec scal ana}), the free boundary formally has $\frac{1}{2}$-more derivative than the velocity fields, which cannot be revealed through flow maps. Another limitation for the Lagrangian framework is that it will be rather difficult to establish the low-regularity theory, because one needs energies containing sufficiently high order derivatives to close the estimates. Searching for low-regularity solutions is significant both physically and mathematically. As there exist singular behaviors (e.g., cusps and angle points) on the free interface in reality, the study of free boundary problems in the low-regularity setting allowing free boundaries to have unbounded curvature can explain and predict more physical phenomena, and the well-posedness theories for rough solutions have garnered much attention in the mathematical community (cf. e.g., \cite{ABZ14Invent,BL15Inven,Wu19Invent,WZZZ21Memo,ifrim2023sharp,ifrim2024sharp} for the incompressible flows and \cite{wang2022rough,disconzietal2022rough,andersson2022well,Ifrim-Tataru2024,disconzi2022} for the compressible flows). We adopt in this paper the Eulerian method (without requiring flow maps) inspired by {\sc Ifrim and Tataru} \cite{Ifrim-Tataru2024} and the subsequent developments \cite{Ifrim-Tataru2023, ifrim2023sharp, disconzi2022} to study the low-regularity well-posedness theory (allowing the gas-vacuum interface to have unbounded curvature) of the physical vacuum problems for the non-isentropic Euler equations, establishing the existence, uniqueness, continuous dependence on initial data, and continuation criteria, in all space dimensions. One note that, in the Eulerian scheme, the non-constant entropy introduces substantial nonlinearity into the system, as the specific entropy contributes a fully nonlinear term in the momentum equation. Extending the analysis of vacuum free boundary problems from isentropic to non-isentropic Euler equations poses challenges that go beyond technical issues. More fundamentally, developing effective strategies for local dynamical motion becomes crucial for investigating the long-term dynamics of non-isentropic ideal gases. This includes understanding global solutions, tracking the evolution of entropy waves, and analyzing the formation of shocks and other singularities.
	
	 The uniqueness theorem obtained in this paper holds for a low-regularity solution class, where the density, velocity field, and specific entropy are Lipschitz continuous. Essentially, this uniqueness result applies to almost all classical solutions. The existence results apply to low-regularity solution classes, which means that the pointwise regularities of the vacuum interface and the boundary velocity may only be $C^{1.5+}$ and $C^{1+}$, respectively. In particular, the free boundary is allowed to have unbounded curvature. The a priori energy estimates are obtained in a neat and sharp manner without loss of derivatives, making the construction of low-regularity solutions and the establishment of continuation criteria feasible. To address the analytic difficulty due to the high nonlinearity caused by the non-constant entropy, we introduce the renormalized effective pressure and an equivalent entropy variable, which transform the fully nonlinear momentum equation into a bilinear one. This reformulation facilitates the exploitation of techniques from multilinear analysis. Additionally, the choices of adapted weighted Sobolev spaces are not directly deduced from the scaling or dimensional analysis, as the $L^2$-type conserved quantity for the specific entropy has a different physical dimension from that of the kinetic and internal energies. This specific choice is motivated by the conservation of (entropy-weighted) mass and the $L^2$-type energy functional of the linearized problems.
		
	We will next reformulate the problems and present our main results.
	\subsection{Reformulation of the Problems}
	Recall that the accelerations of fluid particles are given by:
	\begin{equation*}
		\Dt\vv = - \frac{1}{\rho}\grad{p} = - \e^{\frac{S}{1+\beta}}\grad(\frac{1+\beta}{\beta}p^{\frac{\beta}{1+\beta}}).
	\end{equation*}
	Thus, we define
	\begin{equation}\label{def q}
		q\coloneqq \frac{1+\beta}{\beta}p^{\frac{\beta}{1+\beta}} = \frac{1+\beta}{\beta}\rho^\beta\e^{\frac{\beta}{1+\beta}S}
		\qand
		\sigma \coloneqq \e^{\frac{S}{1+\beta}}.
	\end{equation}
	It is evident that the sets of variables $(\rho, \vv, S)$ and $(q, \vv, \sigma)$ are equivalent. Consequently, the full compressible Euler equations for the variables $(q, \vv, \sigma)$ can be expressed as follows:
	\begin{equation}\label{euler equiv2}
		\begin{cases*}
			\Dt q + \beta q (\divergence\vv) = 0 &in $\Omt$, \\
			\Dt \vv + \sigma\grad q = \vb*{0} &in $\Omt$, \\
			\Dt \sigma = 0 &in $\Omt$,
		\end{cases*}
	\end{equation}
	here $q(t, x) > 0$ in $\Omt$ and $q = 0$ on $\Gmt $. If one extends $q = 0$ outside $\Omt$, then $\Omt$ can be regarded as the support of a non-negative function $q(t)$, i.e., $\Omt = \Set{x \given q(t, x) > 0}.$ The boundary condition \eqref{bdry con phy vac} can be rewritten as
	\begin{equation}\label{phy vac bc}
		0 < \abs{\grad{q}} < \infty \qq{on} \Gmt = \overline{\qty{q > 0}} \cap \qty{q = 0},
	\end{equation}
	provided that $\sigma$ is uniformly bounded from above and below. 
	
	For the sake of conciseness, we denote
	\begin{equation}
		\alpha \coloneqq \beta^{-1} > 0.
	\end{equation}
	With the new variables $\qty(q, \vv, \sigma)$, the conserved physical energy (Hamiltonian) is
	\begin{equation}\label{def Ephy}
		E_{\text{phy}} = \int_{\Omt} q^{1+\alpha} + \frac{1+\beta}{2}q^\alpha \sigma^{-1} \abs{\vv}^2 \dd{x}.
	\end{equation}
	Indeed, it follows from the transport formula in fluid dynamics (see, for instance, \cite[p. 10]{chorin-marsden1993})
	\begin{equation}\label{trans formula}
		\dv{t} \int_{U_t} f \dd{x} = \int_{U_t} \Dt f + f \qty(\divergence\vv) \dd{x}
	\end{equation}
	and the Euler system \eqref{euler equiv2} that
	\begin{equation*}
			\dv{t} E_{\text{phy}} = - \beta \int_{\Omt} \div(q^{1+\alpha} \vv) \dd{x}.
	\end{equation*}
	The boundary condition $q = 0$ on $\Gmt$ yields the conservation law.
	
	We remark here that the finiteness of the physical energy \eqref{def Ephy} and the physical-vacuum type boundary condition \eqref{phy vac bc} imply that the gases should be compactly supported. Thus, in this manuscript, it is always assumed that $\Omt$ is a bounded open set (not necessarily connected).
	
	\subsection{Dimensional Analysis and Function Spaces}\label{sec scal ana}
	
	Denote the dimensions of lengths and times by $L$ and $T$, respectively. Notice that the dimension of the square of sound speed is given by:
	\begin{equation*}
		\qty[c_{\text{s}}^2] = \frac{L^2}{T^2}.
	\end{equation*}
	Thus, for the physical vacuum regimes, the condition \eqref{sound speed decay rate} yields
	\begin{equation*}
		\qty[c_{\text{s}}^2] = L,
	\end{equation*}
	which implies the relation:
	\begin{equation*}
		L = T^2.
	\end{equation*}
	Note that the specific entropy $S$ is non-dimensional, and the dimensions of $q$ and $\vv$ are respectively given by
	\begin{equation*}
		\qty[q] = \qty[c_{\text{s}}^2 \e^{-\frac{S}{1+\beta}}] = \qty[c_{\text{s}}^2] = L = T^2
	\qand
		\qty[\vv] = \frac{L}{T} = L^{\frac{1}{2}} = T.
	\end{equation*}
	In particular, the compressible Euler equations \eqref{euler equiv2} admit the scaling law:
	\begin{equation*}
		\pmqty{q(t, x) \\ \vv(t, x) \\ \sigma(t, x)} \to \pmqty{\tau^{-2}q(\tau t, \tau^2 x) \\ \tau^{-1}\vv(\tau t, \tau^2 x) \\ \sigma(\tau t, \tau^2 x)}.
	\end{equation*}
	In this manuscript, we use the convention that $\pdv*{x}$ has order $+1$, i.e.,
	\begin{equation*}
		\qty{\text{order} = +1} \iff \qty{\text{dimension} = L^{-1}}.
	\end{equation*}
	Therefore, one has a counting device based on the scaling analysis:
	\begin{itemize}
		\item $q$, $\vv$, and $\sigma$ have orders $-1$, $-\frac{1}{2}$, and $0$, respectively;
		\item $\pd_x$ and $\Dt$ have orders $+1$ and $+\frac{1}{2}$, respectively.
	\end{itemize}
	
	For the choice of state spaces, one can first infer from \eqref{def Ephy} that
	\begin{equation*}
		\int_{\Omt} q^{\alpha-1}\qty[\abs{q}^2 + \frac{1+\beta}{2}q\sigma^{-1}\abs{\vv}^2] \dd{x},
	\end{equation*}
	which is conserved for all time, can be viewed as a weighted $L^2$ norm of $(q, \vv)$. For the higher order regularities, first note that the second order evolution equation for $q$ is:
	\begin{equation*}
		\Dt^2 q = \beta q\sigma \laplace{q} + \beta q \grad{q}\vdot\grad{\sigma} + \beta q \qty[\beta(\divergence\vv)^2 + \tr(\grad\vv)^2],
	\end{equation*}
	which, at a leading order, can be viewed as a wave-type equation. Therefore, one can take
	\begin{equation*}
			 \beta q \sigma \laplace = c_{\text{s}}^2 \laplace
	\end{equation*}
	as a characteristic differential operator. Assume that the initial entropy $S_0$ is uniformly bounded. Then, the transport equation yields that the renormalized entropy $\sigma$ is bounded uniformly from above and below, for all time. In other words, given an initial entropy $\sigma_0$ so that $\norm{\sigma_0}_{L^\infty} + \norm*{\sigma_0^{-1}}_{L^\infty} < \infty$, one can now consider the weighted Sobolev norms of $(q, \vv)$ given by
	\begin{equation}\label{energy q v}
		\sum_{m = 0}^{k}\sum_{\abs{\gamma} = 0}^{2m} \int_{\Omt} q^{\alpha-1} \abs{q^m\pd^\gamma q}^2 + q^{\alpha} \abs{q^m \pd^\gamma \vv}^2 \dd{x}.
	\end{equation}
	One may also infer the corresponding norms of $\sigma$ from scaling (dimensional) analysis:
	\begin{equation*}
		\sum_{m = 0}^{k}\sum_{\abs{\gamma} = 0}^{2m} \int_{\Omt} q^{1+\alpha} \abs{q^m\pd^\gamma \sigma}^2 \dd{x}.
	\end{equation*}
	However, this would not be the final choice in this manuscript, since it will unnecessarily raise the regularity indices needed for the local well-posedness. Indeed, it can be derived from the compressible Euler equations \eqref{euler equiv2} that
	\begin{equation*}
		\dv{t} \int_{\Omt} q^{\alpha} \sigma^2 \dd{x} = 0.
	\end{equation*}
	When $\sigma = \const$, this is exactly the conservation of mass. Although the physical dimension of the above integral is different from that of energies, it can still serve as an $L^2$-type mathematical energy functional. Thus, one may consider the higher order norms of $\sigma$ given by:
	\begin{equation}\label{energy sig}
		\sum_{m = 0}^{k}\sum_{\abs{\gamma} = 0}^{2m} \int_{\Omt} q^\alpha \abs{q^m\pd^\gamma \sigma}^2 \dd{x}.
	\end{equation}
	Here we remark that the scaling property of \eqref{energy sig} does not coincide with that of \eqref{energy q v}.
	
	For the simplicity of notation, given a parameter $\lambda > -\frac{1}{2}$, denote by $H^{j, \lambda}_{q}(\Omega)$ the weighted Sobolev spaces
	\begin{equation}\label{def weighted sobolev spaces}
		H^{j, \lambda}_{q}(\Omega) \coloneqq \Set*{f \in \cD'(\Omega) \given \norm{f}_{H^{j, \lambda}_{q}(\Omega)}^2 \coloneqq \sum_{\abs{\gamma} \le j} \int_{\Omega}  \abs{q^{\lambda}\pd^\gamma f}^2 \dd{x} < \infty}.
	\end{equation}
	With this notation, the higher order norms of $(q, \vv, \sigma)$ are defined by:
	\begin{equation}
		\norm{(q, \vv, \sigma)}_{\fkH^{2j}_{q}} \coloneqq \qty(\norm{q}_{H^{2j, j+\ala}_q(\Omega)}^2 + \norm{\vv}_{H^{2j, j+\alb}_q(\Omega)}^2 + \norm{\sigma}_{H^{2j, j+ \alb}_q(\Omega)}^2)^{\frac{1}{2}}.
	\end{equation}
	For non-integer-valued indices, the above definitions can be extended by invoking complex interpolations (non-integral indices will be briefly discussed in \S\ref{sec freq env}, and one can consult \cite[\S2]{Ifrim-Tataru2024} and the references therein for more details).
	
	Here, note that the \emph{critical index} $\varkappa_0$ for the higher order norms of $(q, \vv)$ under scaling is given by the relation:
	\begin{equation*}
		2\varkappa_0 - \qty(\varkappa_0 + \alb) = \frac{1}{2} + \frac{d}{2},
	\end{equation*}
	where $d \ge 1$ is the space dimension. Namely, the critical index $\varkappa_0$ is defined as
	\begin{equation}
		\varkappa_0 \coloneqq \frac{1}{2} + \frac{d}{2} + \frac{1}{2\beta},
	\end{equation}
	for which $\beta$ is the positive index given in the state equation \eqref{state eqn}.
	
	\subsection{Main Results}
	
	Suppose that $\{\Omega_t\}_t$ is a family of bounded open sets (not necessarily connected), and $\varepsilon_* > 0$ is a fixed (arbitrarily small) constant. Assume further that, for each $t$, $q(t) > 0$ in $\Omega_t$, $q(t) = 0$ on $\Gmt = \pd\Omt$, and $\abs{\grad q(t)}>0$ on $\Gmt$. Define the following control parameters for states $(q, \vv, \sigma)$ by:
	\begin{equation}\label{def A*}
		A_* \coloneqq \norm{q}_{C^{1+\varepsilon_*}} + \norm{\vv}_{{C}^{\frac{1}{2}+\varepsilon_*}} + \norm{\sigma}_{C^{\frac{1}{2}+\varepsilon_*}} + \norm*{\sigma^{-1}}_{L^\infty}
	\end{equation}
	and
	\begin{equation}\label{def B}
		B \coloneqq \norm{\grad{q}}_{\wtC^{0, \frac{1}{2}}} + \norm{\grad\vv}_{L^\infty} + \norm{\grad{\sigma}}_{L^\infty},
	\end{equation}
	with the $\wtC^{0, \frac{1}{2}}$-norm given by
	\begin{equation*}
		\norm{f}_{\wtC^{0, \frac{1}{2}}} \coloneqq \norm{f}_{C^0} + \norm{f}_{\wtC^{\frac{1}{2}}}
	\qand
		\norm{f}_{\wtC^{\frac{1}{2}}} \coloneqq \sup_{x, y \in \Omt} \frac{\abs{f(x)-f(y)}}{q^{\frac{1}{2}}(x) + q^{\frac{1}{2}}(y) + \abs{x-y}^{\frac{1}{2}}}.
	\end{equation*}

	Here one may notice that the $\wtC^{\frac{1}{2}}$-seminorm scales like the $\dot{C}^{\frac{1}{2}}$-seminorm, but it is weaker in the sense that it only reflects the $L^\infty$ variations in the regions away from the boundary. We also use the notion $f \in \wtC^{1, \frac{1}{2}}(\Omega)$ if
	\begin{equation*}
		f \in C^1 \qand \norm{\grad f}_{\wtC^{\frac{1}{2}}} < \infty.
	\end{equation*}
	
	With these notations, we state the main results in this manuscript (here we remark that the gases are always assumed to be compactly supported).
	\begin{theorem}[Enhanced Uniqueness]\label{thm uni}
		For each initial data $(q_0, \vv_0, \sigma_0)$ in the Lipschitz class with $q_0$ satisfying the non-degeneracy condition $\abs{\grad{q_0}} > 0$ on $\Gamma_0$, the compressible Euler system \eqref{euler equiv2} admits at most one solution $(q, \vv, \sigma)$ in the class
		\begin{equation*}
			q \in \wtC^{1, \frac{1}{2}}_x \qc \vv \in C^1_x, \qand \sigma \in C^1_x.
		\end{equation*}
		More precisely, the uniqueness holds as long as $A' \in L^\infty_t$, $B \in L^1_t$, and $\inf_{\Gmt} \abs{\grad{q}} > 0$, here $A'$ is the weaker version of $A_*$ given by
		\begin{equation*}
			A' \coloneqq \norm{\grad{q}}_{L^\infty} + \norm{\vv}_{{C}^{\frac{1}{2}}} + \norm{\sigma}_{{C}^{\frac{1}{2}}} + \norm*{\sigma^{-1}}_{L^\infty}.
		\end{equation*}
	\end{theorem}
	
	\begin{remark}
		There is a quantitative version of the uniqueness result. The energy estimates of some difference functionals characterizing the $L^2$-type distance of two solutions will be shown in \S\ref{sec uniqueness}.
	\end{remark}
	
	\begin{remark}
		Such approach can establish the uniqueness in a very general solution class, which covers almost all classical solutions with scaling compatibility. Particularly, the pointwise regularity for the free boundary is merely $C^{1.5}$, for which the classical curvature is not well-defined. Thus, one may not extend the solution to a regular one defined in the whole space, which disables the application of the relative entropy argument in \cite{luoxinzeng2014well} to prove the uniqueness. On the other hand, in the Lagrangian coordinates, the common way to show uniqueness is to establish the energy estimates for the linearized systems satisfied by the difference of two solutions. Due to the requirements for a priori estimates, the uniqueness requires one more derivative than the existence results, specifically, the uniqueness only holds for high-regularity solutions.
	\end{remark}
	
	\begin{defi}\label{def state space}
		The state space $\bbH^{2\varkappa}$ is defined for $\varkappa > \varkappa_0$ (i.e., above the scaling) as the collection of triples $(q, \vv, \sigma)$ satisfying the following:
		\begin{itemize}
			\item \textit{Domain: } $0 \le q \in C^{0}_{\text{c}}(\R^d)$, $\Omega \coloneqq \qty{q > 0}$, and $(\vv, \sigma)$ are functions well-defined in $\overline{\Omega}$.
			\item \textit{Boundary regularity: } $\Gamma \coloneqq \pd\Omega$ is composed of disjoint $C^{1+}$ hypersurfaces. In particular, $\Gamma$ contains no isolated points when the space dimension $d \ge 2$.
			\item \textit{Non-degeneracy: } $q \in C^{1+}(\overline{\Omega})$ and $q = 0, \abs{\grad{q}}>0$ on $\Gamma$. 
			\item \textit{Entropy bounds: } $\sigma \in C^0(\overline{\Omega})$, $\sigma > 0$ in ${\Omega}$, and $\norm{\sigma}_{L^\infty(\Omega)} + \norm*{\sigma^{-1}}_{L^\infty(\Omega)} < \infty$. 
			\item \textit{Function regularities: } $\norm{(q, \vv, \sigma)}_{\fkH^{2\varkappa}_q(\Omega)} < \infty$, i.e.,
			\begin{equation*}
				q \in H^{2\varkappa, \varkappa+\ala}_q(\Omega) \qc  \vv \in H^{2\varkappa, \varkappa+\alb}_q(\Omega), \qand \sigma\in H^{2\varkappa, \varkappa+\alb}_q(\Omega).
			\end{equation*}
		\end{itemize}
	\end{defi}
	
	\begin{remark}
		Since distinct triples in $\bbH^{2\varkappa}$ may be defined in different domains, $\bbH^{2\varkappa}$ is not a linear function space. Instead, it can be regarded as an infinite-dimensional manifold, whose topology will be given in \S\ref{sec state space}. The pointwise regularities of $q$, $\vv$, and $\sigma$ can be ensured by the Sobolev embeddings and the requirement that $\varkappa$ exceeds the scaling. These properties will be further discussed in \S\ref{sec func spaces}.
	\end{remark}
	
	\begin{theorem}[Local Well-posedness]\label{thm well-posedness}
		The full compressible Euler system \eqref{euler equiv2} is locally well-posed in the state space $\bbH^{2\varkappa}$ for any $\varkappa \in \R$ with
		\begin{equation*}
			\varkappa > \varkappa_0 + \frac{1}{2} = 1 + \frac{d}{2} + \frac{1}{2\beta}.
		\end{equation*}
		More precisely, given initial data in $\bbH^{2\varkappa}$, there hold
		\begin{itemize}
			\item existence of solutions in $C([0, T]; \bbH^{2\varkappa})$ for some $T > 0$;
			\item uniqueness of solutions in a larger class (cf. Theorem \ref{thm uni});
			\item weak Lipschitz dependence on the initial data with respect to a $L^2$-type distance functional (cf.Theorem \ref{thm uniqueness});
			\item continuous dependence of solutions on their initial data in the $\bbH^{2\varkappa}$ topology.
		\end{itemize}
	\end{theorem}
	
	\begin{remark}
		The well-posedness holds for a low-regularity class, which is much less regular than those obtained in the Lagrangian schemes. Moreover, we also showed the continuous dependence on initial data, which is not straightforward to obtain in the Lagrangian setting, as those function spaces are defined through complicated energy functionals and thus difficult to identify the topology/continuity.
	\end{remark}
	
	Furthermore, there holds the following continuation criterion, which is consistent with the classical results concerning the quasi-linear hyperbolic systems in the absence of a free boundary.
	\begin{theorem}[Continuation Criteria]
		Let $\varkappa  > \frac{1}{2} + \varkappa_0$ be a real number. Then, the $\bbH^{2\varkappa}$ solutions to the compressible Euler system \eqref{euler equiv2} given by Theorem \ref{thm well-posedness} can be continued whenever the following properties hold:
		\begin{itemize}
			\item Uniform non-degeneracy: there is a constant $c_0 > 0$ so that
			\begin{equation*}
				\inf_{\Gmt} \abs{\grad{q}} \ge c_0 > 0.
			\end{equation*}
			\item Control parameter bounds: for $A_*$ and $B$ given respectively by \eqref{def A*} and \eqref{def B}, there hold
			\begin{equation*}
				A_* \in L^\infty_t \qand B \in L^1_t.
			\end{equation*}
		\end{itemize}
	\end{theorem}
	\begin{remark}
		This criterion implicitly rules out the formation of splash singularities. Additionally, it is necessary to assume that two gas bubbles do not come into contact. While these singularities will not appear in a short period, the results presented in this manuscript cannot exclude their occurrence in the long-term dynamics.
	\end{remark}
	
	Finally, the energy estimates are given by
	\begin{theorem}[A Priori Estimates]
		Suppose that $k \ge 0$ is an integer and $(q, \vv, \sigma)$ is a solution to the compressible Euler system \eqref{euler equiv2} with $\inf_{\Gmt} \abs{\grad{q}} \ge c_0 >0$. Let $A_*$ and $ B$ be the control parameters defined by \eqref{def A*} and \eqref{def B}, respectively. Then, there exists an energy functional $E_{2k}$ satisfying the following properties:
		\begin{itemize}
			\item Coercivity:
			\begin{equation*}
				E_{2k} \simeq_{A_*, c_0} \norm{(q, \vv, \sigma)}_{\fkH^{2k}_q}^2.
			\end{equation*}
			\item Propagation estimate:
			\begin{equation*}
				\abs{\dv{t} E_{2k}} \lesssim_{A_*} B \norm{(q, \vv, \sigma)}_{\fkH^{2k}_q}^2.
			\end{equation*}
		\end{itemize} 
		In particular, Gr{\"o}nwall's inequality yields the bound
		\begin{equation}\label{energy est bound}
			\norm{(q, \vv, \sigma)(t)}_{\fkH^{2k}_{q(t)}(\Omt)}^2 \lesssim_{A_*, c_0} \exp(\int_0^t C(A_*)B(s) \dd{s}) \norm{(q, \vv, \sigma)(0)}_{\fkH^{2k}_{q_0}(\Omega_0)}^2.
		\end{equation}
	\end{theorem}
	\begin{remark}
		Although the energy functionals are initially constructed for integer-valued indices, the arguments in \S\ref{sec rough sol} extend the bound \eqref{energy est bound} applicable to any real index $\varkappa > 0$. In this context, the norms for non-integral indices are interpreted through interpolations.
	\end{remark}
	
	\begin{remark}
		This type of energy estimates enables one to establish the continuation criterion, which is rather difficult to obtain in the Lagrangian setting. Indeed, due to the high nonlinearity of equations, bounds for the time derivatives of energies in the Lagrangian framework are usually of polynomial type, which are difficult to apply to yield a continuation criterion.
	\end{remark}
	
	\subsection{Outline of the Arguments}
	
	This manuscript is structured in a modular fashion, minimizing cross-references between sections. We present technical preparations only when they are immediately used. Moreover, the arguments in \S\S\ref{sec energy est}-\ref{sec existence} are independent of \S\ref{sec uniqueness}.
	
	In \S\ref{sec lin abd space}, the estimates for linearized systems will be shown, which will also motivate the precise choices of weights for the $L^2$ norms. Following this, some basic properties of the weighted Sobolev spaces and state spaces relevant to the physical vacuum problems will be discussed.
	
	\S\ref{sec uniqueness} presents a quantitative uniqueness result, demonstrating the energy estimate of a difference functional. This functional, constructed in a similar manner to the $L^2$-type energies for the linearized problems, characterizes the $L^2$ distance between two solutions and will serve as a crucial tool in \S\ref{sec rough sol}.

	The a priori energy estimates will be given in \S\ref{sec energy est}, with the main theorem deferred until \S\ref{sec energies and control para}. Before that, the heuristics for identifying good unknowns will be explained in detail. Loosely speaking, in order to obtain energy estimates, it is essential to preserve the typical symmetry of the simplified linear systems while keeping the remainders (source terms) controllable. With the help of appropriate interpolations and good knowns, propagation estimates for the energy functionals would be rather simple. The energy coercivity requires adapted remainder controls and degenerate elliptic estimates. Critical manipulations would be more transparent when restricted to a small region close to the boundary, which is achievable through a partition of unity. The last subsection \S\ref{sec other energies} actually serves as a preparation for \S\ref{sec reg data}, whose placement follows the main themes of different sections.
	
	Solutions with sufficiently regular initial data will be constructed in \S\ref{sec existence}. The main philosophy is the application of Euler's polygonal methods to quasi-linear problems. In order to overcome the derivative loss and to control the energy increments simultaneously, one can take carefully prepared regularized data as an intermediate process. Roughly speaking, direct on-scale regularization might disrupt the orthogonality, causing untoward energy increments, while high-scale mollification offers better control on the energy increments but at the cost of exorbitant higher-order bounds. In other words, achieving acceptable behaviors of higher order norms and energy increments in a single step is rather difficult. A bi-scale mollification overcomes these challenges, combining low-scale regularized quantities with high-scale corrections. This would synthesize a philosophically on-scale regularization, ensuring admissible higher order bounds and energy increments. A highly symmetric energy functional is chosen to avoid unnecessary increments during the regularization, which is the reason why a new functional is introduced in \S\ref{sec other energies}. The construction of these regularized data is possibly the most intricate part of this manuscript, which will be discussed extensively in \S\ref{sec reg data}.
	
	Finally, \S\ref{sec rough sol} is devoted to conclude the well-posedness theory in the low-regularity spaces, i.e., to construct the solutions and to show their continuous dependence on the initial data as well as continuation criteria in  fractional state spaces. The crux is to utilize an accurate characterization of these spaces, which will be achieved through the J-method of interpolations and the notion of frequency envelopes. The energy estimates for the integer-valued indices, $L^2$-type difference estimates, multiple appropriate interpolations, and bootstrap arguments will be exploited to conclude the proof.
	


	\section{Linearized Systems and State Spaces}\label{sec lin abd space}

	\subsection{Linear Estimates}
	
	Suppose that $(q, \vv, \sigma)$ is a solution to the compressible Euler equations \eqref{euler equiv2}, and $\qty(s, \vw, \zeta)$ represents the linearized perturbations of $(q, \vv, \sigma)$. Then, $\qty(s, \vw, \zeta)$ satisfies the following linearized system of equations:
	\begin{equation}\label{eqn lin 2}
		\begin{cases*}
			\Dt s + \grad_\vw q+ \beta s \qty(\divergence\vv) + \beta q \qty(\divergence\vw) = 0, \\
			\Dt\vw + \grad_\vw \vv + \sigma\grad{s} + \zeta\grad{q} = \vb*{0}, \\
			\Dt\zeta + \grad_\vw \sigma = 0.
		\end{cases*}
	\end{equation}
	Here $\qty(s, \vw, \zeta)$ are interpreted as functions defined in the time-dependent domain $\Omt$.
	Assume further that the initial entropy $\sigma_0$ is uniformly bounded from above and below, then it follows from its evolution equation that
	\begin{equation}\label{bounds sigma}
		\sigma > 0 \qand \norm{\sigma}_{L^\infty} + \norm*{\sigma^{-1}}_{L^\infty} < \infty \quad \text{for all } t.
	\end{equation}
	The $L^2$-based energy for \eqref{eqn lin 2} can be defined as:
	\begin{equation}\label{def Elin}
		\Elin \coloneqq \frac{1}{2}\int_{\Omt} q^{\alpha-1}\abs{s}^2 + q^{\alpha}\sigma^{-1}\qty(\beta\abs{\vw}^2 + \abs{\zeta}^2) \dd{x}.
	\end{equation}
	We further denote the control parameter $\Blin$ by:
	\begin{equation}\label{def Blin}
		\Blin \coloneqq \norm{\grad{q}}_{L^\infty} + \norm{\grad\vv}_{L^\infty} + \norm{\grad{\sigma}}_{L^\infty}.
	\end{equation}
	Then, it follows from \eqref{euler equiv2}, \eqref{eqn lin 2}, and the transport formula \eqref{trans formula} that
	\begin{equation}\label{est energy lin2}
		\begin{split}
			\dv{t}\Elin &= -\int_{\Omt} q^{\alpha-1}s\vw\vdot\grad{q} + \beta q^{\alpha}\divergence\vw + \beta q^{\alpha}\vw\vdot\grad{s} \dd{x} + \order{\Blin}\Elin \\
			&= -\beta\int_{\Omt}\divergence(q^\alpha s \vw) \dd{x} + \order{\Blin}\Elin.
		\end{split}
	\end{equation}
	In particular, one obtains the following:
	\begin{prop}
		Let $\qty(q, \vv, \sigma)$ be a Lipschitz solution to the compressible Euler system \eqref{euler equiv2}, with the time-dependent domain denoted by $\Omt$. Assume further that $q$ vanishes on $\pd\Omt$, and that the condition \eqref{bounds sigma} holds. Then, any solution $(s, \vw, \zeta)$ to the linearized system \eqref{eqn lin 2} satisfies the following energy estimate:
		\begin{equation}\label{lin est}
			\abs{\dv{t} \Elin} \lesssim \Blin\Elin,
		\end{equation}
		where $\Elin$ and $\Blin$ are defined by \eqref{def Elin} and \eqref{def Blin}, respectively.
	\end{prop}
	Here we remark that the existence of solutions to the linearized problem \eqref{eqn lin 2} in the linear energy space can be derived from \eqref{lin est} and its dual form by applying the Hahn-Banach theorem and duality arguments (see, for instance, \cite[Proposition 3.1]{Ifrim-Tataru2024}, \cite[Proposition 3.7]{Ifrim-Tataru2023}, \cite[Chapter 6]{Lax_book2006}, and \cite[\S24.1]{Hormander_bookIII}). Actually, only the linear estimates will be used later in this manuscript, as the solutions to the nonlinear problems will be constructed via Euler's polygonal methods.
	
	\subsection{Function Spaces}\label{sec func spaces}
	\subsubsection{Weighted Sobolev spaces}
	
	The $L^2$ energy for the linearized problems serves as another motivation for the exact choice of weights in the state spaces. Before giving their exact characterizations, we first consider the notions of weighted Sobolev spaces. One can refer to \cite[\S2.1]{Ifrim-Tataru2024} and the references therein for more details.
	
	Suppose that $\varepsilon_* > 0$ is a (small) constant, and $\Omega \subset \R^d$ is a bounded domain having $C^{1+\varepsilon_*}$ boundaries. Let $r$ be a non-degenerate defining function of $\Omega$ (i.e., $r$ is non-negative, $\Omega = \qty{r(x) > 0}$, $r \in C^{1+\varepsilon_*}(\overline{\Omega})$, and $\abs{\grad{r}} > 0$ on $\pd\Omega$). Then, one can define the weighted Sobolev spaces by:
	\begin{defi}
		Let $\lambda > -\frac{1}{2}$ and $j \ge 0$ be two constants. The space $H^{j, \lambda} = H^{j, \lambda}_{r}(\Omega)$ is defined as the collection of distributions on $\Omega$ satisfying
		\begin{equation*}
			\norm{f}_{H^{j, \lambda}}^2 \coloneqq \sum_{\abs{\nu} \le j} \norm{r^{\lambda}\pd^\nu f}_{L^2(\Omega)}^2 < \infty.
		\end{equation*}
	\end{defi}
	Through complex interpolations, spaces $H^{s, \lambda}$ for any $s \in \R_{\ge 0}$ and $\lambda > -\frac{1}{2}$ are well-defined.
	The weighted Sobolev spaces $H^{s, \lambda}$ satisfy the following properties:
	\begin{itemize}
		\item The value of $r(x)$ is proportional to the distance $\dist(x, \pd\Omega)$ at any point $x \in \Omega$ close to the boundary. Different choices of defining functions will induce the same space, with different but equivalent norms.
		\item The requirement that $\lambda > -\frac{1}{2}$ is due to the fact that no vanishing assumption on $\pd\Omega$ is made for functions in the spaces.
		\item When $\lambda = 0$, the spaces $H^{j, 0}$ are reduced to the classical Sobolev spaces $H^j(\Omega)$.
		\item When $j = 0$, the spaces $H^{0, \lambda}$ are the weighted $L^2$ spaces $L^2(\Omega; r^{2\lambda} \dd{x})$.
	\end{itemize}
	
	Furthermore, there hold the following Hardy-type results (\cite[Lemma 2.2]{Ifrim-Tataru2024}):
	\begin{lemma}\label{lem inclusion}
		Suppose that $s_1 > s_2 \ge 0$, $\lambda_1 > \lambda_2 > -\frac{1}{2}$, and $s_1 - \lambda_1 = s_2 - \lambda_2$. Then, it holds that
		\begin{equation*}
			H^{s_1, \lambda_1} \hookrightarrow H^{s_2, \lambda_2}.
		\end{equation*}
	\end{lemma}
	As a direct corollary, one has the following embeddings into the classical spaces:
	\begin{cor}
		Assume that $s \ge \lambda > 0$. Then, it follows that
		\begin{equation*}
			H^{s, \lambda} \hookrightarrow H^{s-\lambda},
		\end{equation*}
		where $H^{s-\lambda}$ is the standard Sobolev space.
		In particular, there hold the Morrey-type embeddings:
		\begin{equation*}
			H^{s, \lambda} \hookrightarrow C^{\gamma} \qfor 0 \le \gamma \le s-\lambda-\frac{d}{2},
		\end{equation*}
		where the second equality can hold only if $s-\lambda-\frac{d}{2} \notin \mathbb{N}$.
	\end{cor}
	
	\subsubsection{State spaces}\label{sec state space}
	Although the state spaces $\bbH^{2\varkappa}$ given by Definition \ref{def state space} are not linear function spaces, one can still regard them as infinite dimensional manifolds and introduce their topology as the following (cf. \cite[\S2.3]{Ifrim-Tataru2024}):
	\begin{defi}
		A sequence of states $(q_n, \vv_n, \sigma_n)$ converges to $(q, \vv, \sigma)$ in $\bbH^{2\varkappa}$ provided that the following conditions hold:
		\begin{itemize}
			\item \textit{Uniform non-degeneracy: } $\abs*{\grad{q_n}} \ge c_0 > 0$ on $\Gamma_n$, for some generic constant $c_0$.
			\item \textit{Uniform entropy bounds: } $\sup_{n} (\norm{\sigma_n}_{L^\infty(\Omega_n)} + \norm*{\sigma_n^{-1}}_{L^\infty(\Omega_n)}) < \infty$.
			\item \textit{Domain convergence: } $\norm{q_n - q}_{\text{Lip}} \to 0$.
			\item \textit{Norm convergence: } $\forall \epsilon > 0$, there are compactly supported smooth functions $(\wt{q}_n, \wt{\vv}_n, \wt{\sigma}_n)$ and $(\wt{q}, \wt{\vv}, \wt{\sigma})$, so that
			\begin{equation*}
				\norm{(\wt{q}_n, \wt{\vv}_n, \wt{\sigma}_n) - (q_n, \vv_n, \sigma_n)}_{\fkH^{2\varkappa}_{q_n}(\Omega_n)} < \epsilon \qc
				\norm{(\wt{q}, \wt{\vv}, \wt{\sigma}) - (q, \vv, \sigma)}_{\fkH^{2\varkappa}_{q}(\Omega)} < \epsilon,
			\end{equation*}
			and
			\begin{equation*}
				(\wt{q}_n, \wt{\vv}_n, \wt{\sigma}_n) \to (\wt{q}, \wt{\vv}, \wt{\sigma}) \qin C^\infty(\R^d).
			\end{equation*}
		\end{itemize}
	\end{defi}
	With the above definition of convergence, one can establish the continuity of solution maps and the continuous dependence on initial data in the state spaces $\bbH^{2\varkappa}$.
	
	It can be derived from the Sobolev embeddings that the control parameters given by \eqref{def A*} and \eqref{def B} respectively satisfy
	\begin{equation*}
		A_* \lesssim \norm{(q, \vv, \sigma)}_{\fkH^{2\varkappa}_{q}} \qfor \varkappa > \varkappa_0 + \varepsilon_*,
	\end{equation*}
	and
	\begin{equation*}
		B \lesssim \norm{(q, \vv, \sigma)}_{\fkH^{2\varkappa}_q} \qfor \varkappa > \varkappa_0 + \frac{1}{2}.
	\end{equation*}
	
	Moreover, one can infer the regularities on the free boundary by further invoke the trace theorems:
	\begin{equation*}
		\Gmt \in H^{\varkappa - \frac{\alpha}{2}} \hookrightarrow C^{(\varkappa-\varkappa_0)+1}
	\qand
		\vv \in H^{\varkappa - \frac{\alpha+1}{2}}(\Gmt) \hookrightarrow C^{(\varkappa-\varkappa_0)+\frac{1}{2}},
	\end{equation*}
	provided that $(\varkappa-\varkappa_0) \notin 2^{-1}\mathbb{N} $.	Actually, the aforementioned regularities on the free boundary are provided merely for the reference and will not be utilized in the arguments. Given that the vacuum boundary is characteristic, the evolution of the free interface cannot be considered an independent process, distinct to the typical manipulations in the analysis of incompressible problems.
	
	\section{Uniqueness of Classical Solutions}\label{sec uniqueness}
	
	Assume that $\qty(q_1, \vv_1, \sigma_1)$ and $\qty(q_2, \vv_2, \sigma_2)$ are two solutions to the compressible Euler system \eqref{euler equiv2}, with $\Omt^1$ and $\Omt^2$ being the corresponding moving (bounded) domains. To demonstrate the uniqueness, it is natural to define a difference functional that quantifies the distance between the two solutions in some metric spaces. The first explicit challenge encountered here is that the two gas domains may not coincide. One approach to overcome this is to consider their difference within the common domain:
	\begin{equation*}
		\Omt \coloneqq \Omt^1 \cap \Omt^2.
	\end{equation*}
	It is clear that the boundary of $\Omt$ is not necessarily being $C^1$, but it is still Lipschitz whenever $q_1$ and $q_2$ are close in the Lipschitz topology.
	
	\subsection{Difference Functionals}
	
	Consider the difference functional akin to the linear energy \eqref{def Elin} defined as
	\begin{equation}\label{diff functional 1}
		\begin{split}
			\scD \coloneqq \int_{\Omt} &\qty(q_1 + q_2)^{\alpha-1}\abs{q_1 - q_2}^2 + \\
			 &\quad + \qty(q_1 + q_2)^{\alpha}\qty(\sigma_1 + \sigma_2)^{-1}\qty[\beta\abs{\vv_1 - \vv_2}^2 + \abs{\sigma_1 - \sigma_2}^2] \dd{x}.
		\end{split}
	\end{equation}
	Then, there holds the following difference estimate:
	\begin{theorem}\label{thm uniqueness}
		Let $(q_1, \vv_1, \sigma_1)$ and $ (q_2, \vv_2, \sigma_2)$ be two solutions to the compressible Euler system \eqref{euler equiv2} on $[0, T]$ in the regularity class $\wtC^{1, \frac{1}{2}}_x \times C^1_x \times C^1_x$. Assume that $q_1$ and $q_2$ are both non-degenerate on the corresponding boundaries, and that they are sufficiently close in the Lipschitz topology. Then, it holds that
		\begin{equation}\label{est cD}
			\sup_{0 \le t \le T} \scD(t) \lesssim \scD(0),
		\end{equation}
		where the implicit constant depends on the corresponding norms and the non-degeneracy constants for these two solutions.
	\end{theorem}
	
	One can observe that $\scD = 0$ yields $q_1 = q_2$ on $\Gmt$. Thus, the two moving domains $\Omt^1$ and $\Omt^2$ are identical, i.e.,
	\begin{equation*}
		\scD = 0 \quad \iff \quad (q_1, \vv_1, \sigma_1) = (q_2, \vv_2, \sigma_2). 
	\end{equation*}
	In particular, the above theorem concludes the uniqueness.
	
	To show \eqref{est cD}, it suffices to derive an energy estimate for \eqref{diff functional 1}. Note that although either $q_1$ or $q_2$ vanishes on $\Gmt \coloneqq \pd\Omt$, both their sum $(q_1 + q_2)$ and their difference $(q_1 - q_2)$ only vanish on $\Gmt^1 \cap \Gmt^2$. More precisely, one has
	\begin{equation*}
		\abs{q_1 - q_2}_{\restriction_{\Gmt}} = \qty(q_1 + q_2)_{\restriction_{\Gmt}}.
	\end{equation*}
	Thus, one could not expect to utilize \eqref{est energy lin2} to show the energy estimates, as there is no cancellations on the boundary $\Gmt$ in general.
	
	Since it is assumed that $\Gmt^1$ and $\Gmt^2$ are sufficiently close in the Lipschitz topology and both $q_1$ and $q_2$ are non-degenerate on the free boundary, the energy contributions from regions close to $\Gmt$ are less significant compared to those from regions farther away from the boundary. Therefore, one can select an ancillary function $\fka(\mu, \nu)$, which is a smooth homogeneous function of degree $0$, satisfying the following conditions:
	\begin{equation*}
		0 \le \fka \le 1, \quad \spt \fka \subset \qty{\abs{\nu} \le \frac{9}{10}\mu }, \qand \fka \equiv 1 \text{ in } \qty{\abs{\nu} \le \frac{4}{5}\mu}.
	\end{equation*}
	Consider the following reduced distance functional (see also \cite{Ifrim-Tataru2024}):
	\begin{equation}\label{def fD}
	\begin{split}
			\fD \coloneqq  \int_{\Omt} &(q_1 + q_2)^{\alpha -1} \fka(q_1 + q_2, q_1 - q_2)\, \times \\ &\quad\times \qty{\abs{q_1 - q_2}^2 + (q_1 + q_2)(\sigma_1 + \sigma_2)^{-1}\qty[\beta\abs{\vv_1 - \vv_2}^2 + \abs{\sigma_1 - \sigma_2}^2]} \dd{x}.
	\end{split}
	\end{equation}
	Thus, it is clear that $\fD \le \scD$. The following lemma indicates the converse:
	\begin{lemma}\label{lem equiv dist functional}
		Suppose that $q_j \in{\wtC}^{1, \frac{1}{2}}$, $\abs{\grad_{\vn}q_j}_{\restriction_{\Gmt^j}} \ge c_0 > 0$ $(j = 1, 2)$, and $\norm{q_1 - q_2}_{\textup{Lip}} \ll c_0$. Then, it holds that
		\begin{equation}\label{dist equiv}
			\scD \lesssim \fD,
		\end{equation}
		where the implicit constant depends on $c_0$ and $A' \coloneqq A'_1 + A'_2$. Here the control parameter $A'_j$ is defined by:
		\begin{equation}
			A'_j \coloneqq \norm{\grad{q_j}}_{L^\infty} + \norm{\vv_j}_{{C}^{\frac{1}{2}}} + \norm{\sigma_j}_{{C}^{\frac{1}{2}}} + \norm*{\sigma_j^{-1}}_{L^\infty} \quad (j = 1, 2).
		\end{equation}
	\end{lemma}
	
	For the simplicity of notations, we denote by:
	\begin{equation}\label{simple notaion}
		\mu \coloneqq q_1 + q_2 \qc \nu \coloneqq q_1 - q_2 \qc \vw \coloneqq \vv_1 - \vv_2\qc
		\kappa \coloneqq \sigma_1 + \sigma_2\qc \zeta \coloneqq \sigma_1 - \sigma_2,
	\end{equation}
	and
	\begin{equation*}
		\delta(x) \coloneqq \dist\qty(x, \Gmt) \qfor x \in \Omt.
	\end{equation*}
	
	\begin{proof}
		It is clear that $\mu = \abs{\nu}$ on $\Gmt$, and $\fka = 1$ when $\mu \ge \frac{5}{4}\abs{\nu}$.	Since $\overline{\Omega}_t$ is compact and $\mu > 0$ in $\Omt$, there exist two constants ${\varepsilon}_0 > 0$ and $\varepsilon_1 > 1$ such that
		\begin{equation*}
			\mu \ge \varepsilon_1 \abs{\nu} \qq{whenever} \delta(x) \ge \varepsilon_0.
		\end{equation*}
		One may assume that $\varepsilon_1 = \frac{5}{4}$ for the sake of simplicity (otherwise, one can modify the definition of the ancillary function $\fka(\mu, \nu)$). Thus, the missing part of the original distance functional $\scD$ concentrates on a small neighborhood of $\Gmt$.
		
		The modulus of continuity of $\mu$ and the non-degeneracy condition imply that
		\begin{equation}\label{rough est mu}
			\mu(x) \ge \mu(\overline{x}) + \frac{c_0}{2}\delta(x) \qq{whenever} \delta(x) \ll 1,
		\end{equation}
		here $\overline{x} \in \Gmt$ is the point such that $\delta(x) = \abs{x-\overline{x}}$. On the other hand,
		\begin{equation*}
			\abs{\nu}(x) \le \abs{\nu(\overline{x})} + \norm{\grad\nu}_{L^\infty}\delta(x).
		\end{equation*}
		In particular, as long as
		\begin{equation*}
			\mu(\overline{x}) + \frac{c_0}{2}\delta(x) \ge \frac{5}{4}\qty\big[\abs{\nu}(\overline{x}) + \norm{\grad\nu}_{L^\infty}\delta(x)],
		\end{equation*}
		it holds that $\fka(\mu, \nu)(x) = 1$.
		Since $\mu = \abs{\nu}$ on $\Gmt$, it suffices to have
		\begin{equation*}
			\delta(x) \ge \frac{\abs{\nu}(\overline{x})}{2c_0 - 4\norm{\grad\nu}_{L^\infty}}.
		\end{equation*}
		Thus, when $\norm{q_1 - q_2}_{\Lip} \ll c_0$, one has 
		\begin{equation}\label{region a=1}
			\fka(\mu, \nu)(x) = 1 \qq{as long as} \delta(x) \gtrsim \abs{\nu}(\overline{x}).
		\end{equation}
		Namely, the missing part of $\fD$ comparing to $\scD$ is the integral of the integrands of $\scD$ over the boundary layer region $\qty{\delta(x) \lesssim \abs{\nu}(\overline{x})}$.
				
		By taking the foliations of $\Omt$ using lines transversal to $\Gmt$, one can reduce \eqref{dist equiv} into the one-dimensional estimate. Indeed, denote by
		$\overline{\nu}$ the boundary value of $\nu$, it follows that
		\begin{equation*}
			\int_{\qty{\delta\lesssim\abs{\overline\nu}}} \mu^{\alpha-1}\abs{\nu}^2 \dd{x} \lesssim \int_{\Gmt}\qty(\int_{0}^{c\abs{\overline\nu}} \abs{\overline\nu}^{\alpha+1} \dd{\delta}) \dd{\cH^{d-1}} \lesssim \int_{\Gmt} \abs{\overline\nu}^{\alpha+2} \dd{\cH^{d-1}},
		\end{equation*}
		while \eqref{rough est mu}-\eqref{region a=1} yield that
		\begin{equation}\label{control fD T1}
			\fD \gtrsim \int_{\Gmt} \qty(\int_{c_1\abs{\overline\nu}}^{c_2\abs{\overline\nu}}\qty(\abs{\overline\nu}+\delta)^{\alpha-1}\abs{\overline\nu}^2 \dd{\delta})\dd{\cH^{d-1}} \gtrsim \int_{\Gmt} \abs{\overline\nu}^{\alpha+2} \dd{\cH^{d-1}}.
		\end{equation}
		In other words, the missing $\nu$-parts of $\fD$ are controllable.
		
		To deal with the $\vw$-terms, one can define (using the 1-D foliations):
		\begin{equation*}
			\overline\vw \coloneqq \fint_{\qty{\delta \simeq \abs{\overline\nu}}} \vw \dd{\delta}.
		\end{equation*}
		Then, it is routine to check that
		\begin{equation*}
			\begin{split}
				\abs{\overline\vw}^2 &\lesssim \fint_{\qty{\delta \simeq \abs{\overline\nu}}} \abs{\vw}^2 \dd{\delta} \\
				&\lesssim \norm{\kappa}_{L^\infty}\abs{\overline\nu}^{-\alpha}\fint_{\qty{\delta \simeq \abs{\overline\nu}}} \mu^\alpha \kappa^{-1}\abs{\vw}^2 \dd{\delta} \\
				&\lesssim \abs{\overline\nu}^{-(1+\alpha)}\int_{\qty{\delta \simeq \abs{\overline\nu}}} \mu^{\alpha} \kappa^{-1} \abs{\vw}^2 \dd{\delta}.
			\end{split}
		\end{equation*}
		In particular, it holds that
		\begin{equation}\label{control fD T2'}
			\int_{\Gmt} \abs{\overline\nu}^{1+\alpha}\abs{\overline\vw}^2 \dd{\cH^{d-1}} \lesssim \fD.
		\end{equation}
		On the other hand, one has
		\begin{equation*}
			\abs{\vw}^2 \lesssim \abs{\overline\vw}^2 + A'\abs{\overline\nu} \qin \qty{\delta \lesssim \abs{\nu}},
		\end{equation*}
		which, together with \eqref{control fD T1}-\eqref{control fD T2'}, yield
		\begin{equation}
			\begin{split}
			\int_{\qty{\delta \lesssim \abs{\overline\nu}}} \mu^\alpha\kappa^{-1}\abs{\vw}^2 \dd{x} &\lesssim \int_{\Gmt}\qty(\int_0^{c\abs{\overline\nu}}\mu^\alpha\kappa^{-1}\abs{\vw}^2 \dd{\delta} ) \dd{\cH^{d-1}} \\
			&\lesssim \int_{\Gmt} \abs{\overline\nu}^{1+\alpha} \qty(\abs{\overline\vw}^2 + A'\abs{\overline\nu}) \dd{\cH^{d-1}} \\
			&\lesssim \fD.
			\end{split}
		\end{equation}
		
		The estimates for the term involving $\zeta$ follow from the same arguments as those for $\vw$.
		
	\end{proof}

	\subsection{Energy Estimates}
	To compute the evolution of $\fD$, one need first consider the material derivatives of the integrands. Since $\Omt$ is the common domain, it is natural to consider the mean velocity as the evolution one. More precisely, define
	\begin{equation*}
		\vbu \coloneqq \frac{\vv_1 + \vv_2}{2} \qand \Dt \coloneqq \pd_t + \grad_\vbu.
	\end{equation*}
	Then, the evolution equations for the difference and sum quantities are:
	\begin{subequations}\label{evo eqn diff}
			\begin{equation}\label{Dt q1-q2}
			\begin{split}
				\Dt (q_1 - q_2) = &-\frac{1}{2}\grad_{\vv_1-\vv_2}\qty(q_1 + q_2) \\
				&- \frac{\beta}{2}\qty\big[(q_1 + q_2)\divergence(\vv_1 - \vv_2) + (q_1 - q_2)\divergence(\vv_1 + \vv_2)],
			\end{split}
		\end{equation}
		\begin{equation}\label{Dt q1+q2}
			\begin{split}
				\Dt(q_1+q_2) = &-\frac{1}{2}\grad_{\vv_1-\vv_2}(q_1 - q_2) \\
				&-\frac{\beta}{2}\qty\big[(q_1+q_2)\divergence(\vv_1+\vv_2)+(q_1 - q_2)\divergence(\vv_1 - \vv_2)],
			\end{split}
		\end{equation}
		\begin{equation}\label{Dt v1-v2}
			\begin{split}
				\Dt (\vv_1 - \vv_2) = &-\frac{1}{2}\grad_{\vv_1 - \vv_2}(\vv_1 + \vv_2) \\
				&-\frac{1}{2}\qty\big[(\sigma_1 - \sigma_2)\grad(q_1 + q_2) + (\sigma_1 + \sigma_2)\grad(q_1 - q_2) ],
			\end{split}
		\end{equation}
		\begin{equation}\label{Dt tau1-tau2}
			\Dt(\sigma_1 - \sigma_2) = -\frac{1}{2}\grad_{\vv_1 - \vv_2}(\sigma_1 + \sigma_2),
		\end{equation}
		and
		\begin{equation}\label{Dt tau1+tau2}
			\Dt(\sigma_1 + \sigma_2) = -\frac{1}{2}\grad_{\vv_1 - \vv_2}(\sigma_1 - \sigma_2).
		\end{equation}
	\end{subequations}
	
	Define control parameter $B'$ by:
	\begin{equation}\label{def B'}
		B' \coloneqq B_1 + B_2,
	\end{equation}
	where $B_j\ (j = 1, 2)$ is the control parameter given by \eqref{def B}.	Then, one has the following estimate:
	\begin{prop}
		It holds that
		\begin{equation}\label{energy est fD cD}
			\dv{t}\fD \lesssim_{A', c_0} B' \scD.
		\end{equation}
	\end{prop}
	\begin{proof}
		For the simplicity of notations, we still use \eqref{simple notaion} to rewrite \eqref{def fD} as
		\begin{equation*}
			\fD = \int_{\Omt} \fka(\mu, \nu)\qty(\mu^{\alpha-1}\abs{\nu}^2 + \beta\mu^{\alpha}\kappa^{-1}\abs{\vw}^2 + \mu^{\alpha}\kappa^{-1}\abs{\zeta}^2) \dd{x}.
		\end{equation*}
		\paragraph*{Step 1: First simplifications}
		To compute the time derivative of $\fD$, one can first infer from \eqref{trans formula} and \eqref{def B'} that the contribution of $(\divergence\vbu)$ term is directly controlled by $B'\fD$. Similarly, for the material derivatives involving \eqref{evo eqn diff}, only the first two terms of \eqref{Dt q1-q2}, the first term of \eqref{Dt q1+q2}, and the last term of \eqref{Dt v1-v2} need  further estimates. More precisely, one has:
		\begin{equation}\label{dt fD 1}
			\begin{split}
				\dv{t}\fD = &\int_{\Omt} \fka(\mu, \nu)\qty(\mu^{\alpha-1}\abs{\nu}^2 + \beta\mu^{\alpha}\kappa^{-1}\abs{\vw}^2 + \mu^{\alpha}\kappa^{-1}\abs{\zeta}^2) (\divergence\vbu) \dd{x} \\
				&+ \int_{\Omt} \qty(\fka_{,\mu}\Dt\mu + \fka_{,\nu}\Dt\nu)	\qty(\mu^{\alpha-1}\abs{\nu}^2 + \beta\mu^{\alpha}\kappa^{-1}\abs{\vw}^2 + \mu^{\alpha}\kappa^{-1}\abs{\zeta}^2) \dd{x} \\
				&+ \int_{\Omt} \qty[\Dt\qty(\mu^{\alpha-1})\abs{\nu}^2 + \mu^{\alpha-1}\Dt\qty(\abs{\nu}^2)] \dd{x} \\
				&+ \int_{\Omt} \beta\fka(\mu, \nu)\qty[\Dt\qty(\mu^{\alpha})\kappa^{-1}\abs{\vw}^2 + \mu^{\alpha}\Dt\qty(\kappa^{-1})\abs{\vw}^2 + \mu^{\alpha}\kappa^{-1}\Dt\qty(\abs{\vw}^2)] \dd{x} \\
				&+ 	\int_{\Omt} \fka(\mu, \nu)\qty[\Dt\qty(\mu^{\alpha})\kappa^{-1}\abs{\zeta}^2 + \mu^{\alpha}\Dt\qty(\kappa^{-1})\abs{\zeta}^2 + \mu^{\alpha}\kappa^{-1}\Dt\qty(\abs{\zeta}^2)] \dd{x}.	
			\end{split}
		\end{equation}
		It is clear that the first integral on the right of \eqref{dt fD 1} is bounded by $B'\scD$. For the term involving $\Dt \fka$, one first observe that the homogeneity of $\fka$ yields
		\begin{equation*}
			\mu\fka_{,\mu} + \nu\fka_{,\nu} \equiv 0.
		\end{equation*}
		Thus, it follows that
		\begin{equation*}
			\Dt\fka = -\frac{1}{2}\qty\Big[\beta\qty(\fka_{,\mu}\nu+\fka_{,\nu}\mu)\qty(\divergence\vw) + \qty(\fka_{,\mu}\vw\vdot\grad\nu + \fka_{,\nu}\vw\vdot\grad\mu)],
		\end{equation*}
		which implies that the second integral on the right of \eqref{dt fD 1} can be reduced to
		\begin{equation*}
			-\frac{1}{2}\int_{\Omt} \mu^{\alpha-1} \qty(\fka_{,\mu}\grad_\vw\nu + \fka_{,\nu}\grad_\vw\mu) \qty[\abs{\nu}^2 + \mu\kappa^{-1} \qty(\beta\abs{\vw}^2 + \abs{\zeta}^2)] \dd{x} + \err,
		\end{equation*}
		where the error term can be controlled via
		\begin{equation}\label{est err}
			\abs{\err} = \Order{A'}{B'}\scD.
		\end{equation}
		Furthermore, the fact that
		\begin{equation*}
			\Dt(\kappa^{-1}) = -\kappa^{-2}\Dt\kappa = \frac{1}{2}\kappa^{-2}\grad_{\vw}\zeta
		\end{equation*}
		yields
		\begin{equation*}
			\norm{\kappa\abs*{\Dt(\kappa^{-1})}}_{L^\infty} \lesssim \norm*{\kappa^{-1}}_{L^\infty} \norm{\vw}_{L^\infty} \norm{\grad\zeta}_{L^\infty} \lesssim_{A'} B'.
		\end{equation*}
		Namely, those terms containing $\Dt(\kappa^{-1})$ is controllable under $\Order{A'}{B'}\scD$.
		In a similar manner, it is routine to compute that
		\begin{equation}\label{dt  fD 2}
			\begin{split}
				\dv{t}\fD = &-\frac{1}{2}\int_{\Omt} \mu^{\alpha-1} \qty(\fka_{,\mu}\grad_\vw\nu + \fka_{,\nu}\grad_\vw\mu) \qty[\abs{\nu}^2 + \mu\kappa^{-1} \qty(\beta\abs{\vw}^2 + \abs{\zeta}^2)] \dd{x} \\
				&- \int_{\Omt} \fka\qty[\frac{\alpha-1}{2}\mu^{\alpha-2}(\grad_\vw\nu)\abs{\nu}^2 + \mu^{\alpha-1}\nu\qty(\grad_\vw\mu + \beta\mu\divergence\vw)] \dd{x} \\
				&-\int_{\Omt}\beta\fka\qty[\alpha\mu^{\alpha-1}\qty(\frac{1}{2}\grad_\vw\nu)\kappa^{-1}\abs{\vw}^2 + \mu^{\alpha}\kappa^{-1}\vw\vdot(\kappa\grad\nu)] \dd{x} \\
				&-\int_{\Omt} \fka\qty[\alpha\mu^{\alpha-1}\qty(\frac{1}{2}\grad_\vw\nu)\kappa^{-1}\abs{\zeta}^2 + \mu^{\alpha}\kappa^{-1}\zeta\grad_\vw\kappa] \dd{x} \\
				&+\err\\
				=&-\frac{1}{2}\int_{\Omt} \mu^{\alpha-1} \qty(\fka_{,\mu}\grad_\vw\nu + \fka_{,\nu}\grad_\vw\mu) \qty[\abs{\nu}^2 + \mu\kappa^{-1} \qty(\beta\abs{\vw}^2 + \abs{\zeta}^2)] \dd{x} \\
				&-\frac{1}{2}\int_{\Omt}\fka\mu^{\alpha-2}\grad_\vw\nu\qty[(\alpha-1)\abs{\nu}^2 + \mu\kappa^{-1}\abs{\vw}^2 + \alpha\mu\kappa^{-1}\abs{\zeta}^2] \dd{x} \\
				&-\int_{\Omt} \fka\mu^{\alpha-1}\qty(\nu\grad_\vw\mu + \beta\mu\nu\divergence\vw + \beta\mu\grad_\vw\nu) \dd{x} \\
				&+\err.
			\end{split}
		\end{equation}
		Note that the last integral on the right of \eqref{dt  fD 2} yields a cancellation, i.e.
		\begin{equation*}
			\begin{split}
				&-\int_{\Omt} \fka\mu^{\alpha-1}\qty(\nu\grad_\vw\mu + \beta\mu\nu\divergence\vw + \beta\mu\grad_\vw\nu) \dd{x} \\
				&\quad=\beta\int_{\Omt}(\grad_\vw\fka)\mu^\alpha\nu \dd{x}\\
				&\quad=\beta\int_{\Omt}\mu^\alpha\nu\qty(\fka_{,\mu}\grad_\vw\mu + \fka_{,\nu}\grad_\vw\nu) \dd{x}.
			\end{split}
		\end{equation*}
	Collecting the above computations, one obtains that
	\begin{equation}\label{def I1 I2 I3}
		\begin{split}
			\dv{t}\fD = &- \int_{\Omt} \mu^{\alpha-1}\kappa^{-1}\qty(\mu\fka_{,\mu}\grad_\vw\nu + \mu\fka_{,\nu}\grad_\vw\mu + \alpha\fka\grad_\vw\nu)\qty(\beta\abs{\vw}^2 + \abs{\zeta}^2) \dd{x} \\
			&+\int_{\Omt} \mu^{\alpha-2}\nu^2\qty(\beta\mu^2\nu^{-1}\fka_{,\nu} - \frac{1}{2}\mu\fka_{,\mu} - \frac{\alpha-1}{2}\fka)\grad_\vw\nu \dd{x} \\
			&+ \int_{\Omt} \mu^{\alpha-1}\nu \qty(\beta\mu\fka_{,\mu}-\frac{1}{2}\nu\fka_{,\nu})\grad_\vw\mu \dd{x} + \err \\
			\eqqcolon&\, I_1 + I_2 + I_3 + \err.
		\end{split}
	\end{equation}
	
	\paragraph*{Step 2: Estimates of $I_1$, $I_2$, and $I_3$}
	The homogeneity of $\fka$ implies that
	\begin{equation}
		\abs{I_1} \lesssim_{A'} \int_{\Omt} \mu^{\alpha-1}\abs{\vw}\qty(\abs{\vw}^2 + \abs{\zeta}^2) \dd{x}.
	\end{equation}
	Note that the following Hardy-type inequality holds for $f_i \in C^1_\textup{c}[0, \infty)$ $(i = 1, 2, 3)$ and $\alpha>0$:
		\begin{equation}\label{L3 Hardy}
			\norm*{x^{\alpha-1}f_1 f_2 f_3}_{L^1} \lesssim_{\alpha} \sum_{\text{distinct } i, j, k } \norm{f'_i}_{L^\infty}\norm*{x^\frac{\alpha}{2}f_j}_{L^2}\norm*{x^\frac{\alpha}{2}f_k}_{L^2}.
		\end{equation}
	Indeed, it is direct to compute that
	\begin{equation*}
			\int_0^\infty x^{\alpha-1}f_1f_2f_3 \dd{x} = \int_0^\infty \frac{1}{\alpha}(x^\alpha)' f_1 f_2 f_3 \dd{x} = - \frac{1}{\alpha}\int_0^\infty x^\alpha (f_1 f_2 f_3)' \dd{x},
	\end{equation*}
	which yields \eqref{L3 Hardy}.
	Using the foliation as in the proof of Lemma \ref{lem equiv dist functional}, the estimate of $I_1$ is reduced to the one dimensional case. Since
	\begin{equation*}
		\mu(x) \approx \abs{\nu}(\overline{x}) + \delta(x) \qq{when} \delta(x) \ll 1,
	\end{equation*}
	it follows from \eqref{L3 Hardy} that
	\begin{equation}\label{est I1}
		\abs{I_1} \lesssim_{A'} B'\scD.
	\end{equation}
	
	For $I_3$, one first observes that
	\begin{equation*}
		\beta\mu\fka_{,\mu}-\frac{1}{2}\nu\fka_{,\nu} = \qty(\beta+\frac{1}{2})\mu\fka_{,\mu}.
	\end{equation*}
	Note that $\fka$ is homogeneous of degree $0$, one may assume that
	\begin{equation*}
		\fka(\mu, \nu) = \psi\qty(\frac{\nu}{\mu})
	\end{equation*}
	for some smooth function $\psi$ defined on $\R$. Thus, one has
	\begin{equation*}
		\fka_{,\mu} = - \frac{\nu}{\mu^2} \psi'\qty(\frac{\nu}{\mu}).
	\end{equation*}
	Since it is required that $\fka \equiv 1$ when $\abs{\nu} \ll \mu$, the function $\psi$ satisfies the property that $\psi \equiv 0$ in a small neighborhood of the origin. In particular, Taylor's expansion yields
	\begin{equation*}
		\psi'\qty(\frac{\nu}{\mu}) = \order{\frac{\abs{\nu}}{\mu}}.
	\end{equation*}
	Namely, one obtains
	\begin{equation*}
		\mu^\alpha\nu\fka_{,\mu} = \order{\mu^{\alpha-3}\abs{\nu}^3}.
	\end{equation*}
	Denote by
	\begin{equation}
		J \coloneqq \int_{\Omt} \mu^{\alpha-3}\abs{\nu}^3\abs{\vw} \dd{x}.
	\end{equation}
	Then, one gets the estimate
	\begin{equation}\label{est I3}
		\abs{I_3} \lesssim_{A'} J.
	\end{equation}
	
	To estimate $I_2$, one may first denote by
	\begin{equation*}
		\fka_1 \coloneqq \beta\mu^2\nu^{-1}\fka_{,\nu} - \frac{1}{2}\mu\fka_{,\mu} - \frac{\alpha-1}{2}\fka.
	\end{equation*}
	It is clear that $\fka_1$ is still a smooth homogeneous function of degree $0$ (note that $\mu \simeq \abs{\nu}$ in $\spt \qty{\pd\fka}$, so $\mu^2\nu^{-1}\fka_{,\nu}$ is not a troublesome term).
	Whence, $I_2$ can be estimated via the integration by parts:
	\begin{equation*}
		I_2	= - \frac{1}{3} \int_{\Omt}\nu^3 \divergence(\mu^{\alpha-2}\fka_1 \vw) \dd{x}.
	\end{equation*}
	In particular, it follows that
	\begin{equation}\label{est I2}
		\abs{I_2} \lesssim_{A'} J + B'\scD.
	\end{equation}
	It remains to show that $J$ is bounded by $\Order{A'}{B'}\scD$. 
	
	\paragraph*{Step 3: Estimate of $J$}
	First note that, assuming following relation
	\begin{equation}\label{goal relation mu nu}
		\norm*{\mu^{-\frac{5}{2}}\nu^2}_{L^\infty} \lesssim_{A'}B',
	\end{equation}
	the estimate of $J$ can be derived via H\"older's inequality. In particular, provided that $\Gmt^1 = \Gmt^2$ and $\abs{\grad \nu} = 0$ on $\Gmt$, one can directly obtain the following estimate:
	\begin{equation}\label{good nu relation}
		 \abs{\nu}\lesssim_{A'} B' \delta^\frac{3}{2},
	\end{equation}
	which yields \eqref{goal relation mu nu}.
	However, one could not always expect \eqref{goal relation mu nu} to hold near the boundary, since $\mu$ itself could be rather small (and even vanish) on $\Gmt$. Inspired by \cite{Ifrim-Tataru2024}, such obstacle can be overcome by introducing a well-chosen boundary layer. More precisely, define a function $\lambda$ on $\Gmt$ by:
	\begin{equation}\label{def lambda}
		\lambda (\overline{x}) \coloneqq C\abs{\nu}(\overline{x}) + \qty[(B')^{-1}\abs{\nu}(\overline{x})]^\frac{2}{3} + \qty[(B')^{-1}\abs{\grad\nu}(\overline{x})]^2 \qfor \overline{x} \in \Gmt,
	\end{equation}
	where $C \gg 1$ is a large constant. In other words, by defining such a boundary function, one has the trivial estimates:
	\begin{equation*}
		\abs{\nu} \le \frac{1}{C}\lambda \qc \abs{\nu} \le B'\lambda^{\frac{3}{2}}, \qand \abs{\grad\nu} \le B'\lambda^{\frac{1}{2}} \qq{on} \Gmt.
	\end{equation*}
	Therefore, after removing a boundary layer with thickness comparable to $\lambda$, one can expect that the ``good'' bound \eqref{good nu relation} holds outside the layer. Actually, let
	\begin{equation}
		\Ombd \coloneqq \Omt \cap \qty[\bigcup_{\overline{x}\in\Gmt}B\qty(\overline{x}, \varepsilon\lambda(\overline{x}))],
	\end{equation}
	where $\varepsilon>0$ is a sufficiently small constant. (Here one may not be able to choose the tubular neighborhoods, because $\Gmt$ is merely Lipschitz.) 
	Then, \eqref{good nu relation} holds for $x \notin \Ombd$. Indeed, one may assume that $\overline{x} \in \Gmt$ is one of the points such that $\delta(x) = \abs{x-\overline{x}}$, which implies $\delta(x) \ge \varepsilon\lambda(\overline{x})$. On the other hand, Taylor's expansion gives that
	\begin{equation}\label{rough est diff nu}
		\begin{split}
			\nu(x) = &\,\nu(\overline{x}) + \qty(x-\overline{x}) \cdot \int_0^1 \grad{\nu}(\overline{x} + s(x-\overline{x})) \dd{s} \\
			=&\, \nu(\overline{x}) + (x-\overline{x})\cdot\grad{\nu}(\overline{x}) + \Order{A'}{B'}\abs{x-\overline{x}}\qty(\abs{x-\overline{x}}^{\frac{1}{2}} + \mu^{\frac{1}{2}}(\overline{x})).
		\end{split}
	\end{equation}
	Since $\mu \equiv \abs{\nu}$ on $\Gmt$ and $\delta(x) \ge \varepsilon\lambda(\overline{x})$, it is routine to calculate that
	\begin{equation}\label{est J Om int}
		\begin{split}
			\abs{\nu}(x) &\le \abs{\nu}(\overline{x}) + B'\lambda(\overline{x})^\frac{1}{2}\delta(x) + \Order{A'}{B'}\qty(\delta^\frac{3}{2}(x) + \delta(x)\abs{\nu}^{\frac{1}{2}}(\overline{x})) \\
			&\lesssim_{A'} B' \qty(\lambda^\frac{1}{2}(\overline{x})+\delta^\frac{1}{2}(x))\qty[\lambda(\overline{x}) + \delta(x)] \\
			&\lesssim_{A'} B' \delta^\frac{3}{2}(x).
		\end{split}
	\end{equation}
	Therefore, the integral $J$ over the region $\Omt\setminus\Ombd$ can be controlled by $\Order{A'}{B'}\scD$. It remains to estimate $J_{\restriction_{\Ombd}}$.
	
	In order to localize the integral into a small region, one needs to estimate the variation of $\lambda$. Otherwise, if $\lambda$ is rapidly varying, the open cover of $\Ombd$ using balls with scale $\varepsilon\lambda(\overline{x})$ may not be locally finite. Suppose that $\overline{x}$ and $\overline{y}$ are two boundary points. Then it follows from the same calculations as in \eqref{rough est diff nu} that
	\begin{equation*}
		\abs{\nu(\overline{x})- \nu(\overline{y})} \le \abs{\grad\nu}(\overline{x})\abs{\overline{x}-\overline{y}} + B'\abs{\overline{x}-\overline{y}}\qty(\abs{\nu(\overline{x})}^\frac{1}{2} + (1+A')^\frac{1}{2}\abs{\overline{x}-\overline{y}}),
	\end{equation*}
	and
	\begin{equation*}
		(B')^{-1}\abs{\nu(\overline{x})-\nu(\overline{y})} \lesssim_{A'} \abs{\overline{x}-\overline{y}}\qty(\lambda^\frac{1}{2}(\overline{x})+\abs{\nu}^\frac{1}{2}(\overline{x})+\abs{\overline{x}-\overline{y}}^\frac{1}{2}).
	\end{equation*}
	Moreover, it holds that
	\begin{equation*}
		\begin{split}
			&\abs\Big{\abs{\grad\nu}^2(\overline{x})-\abs{\grad\nu}^2(\overline{y})} = \abs\Big{\qty\big[\grad\nu(\overline{x})-\grad\nu(\overline{y})]\vdot\qty\big[\grad\nu(\overline{x})+\grad{\nu}(\overline{y})]} \\
			&\quad\lesssim_{A'} \! B' \qty(\mu^\frac{1}{2}(\overline{x})+\mu^\frac{1}{2}(\overline{y})+\abs{\overline{x}-\overline{y}}^\frac{1}{2})\qty\Big(2\abs{\grad\nu(\overline{x})}+\abs{\grad{\nu}(\overline{x})-\grad{\nu}(\overline{y})}) \\
			&\quad\lesssim_{A'} \! (B')^2\qty[ \lambda^\frac{1}{2}(\overline{x})\abs{\nu}^\frac{1}{2}(\overline{x}) + \lambda^\frac{1}{2}(\overline{x})\abs{\overline{x}-\overline{y}}^\frac{1}{2} + \abs{\overline{x}-\overline{y}}].
		\end{split}
	\end{equation*}
	Thus, one has
	\begin{equation}
		\abs{\lambda(\overline{x})-\lambda(\overline{y})} \lesssim_{A'} \abs{\overline{x}-\overline{y}}^\frac{2}{3} \qty(\lambda^\frac{1}{3}(\overline{x}) + \abs{\overline{x}-\overline{y}}^\frac{1}{3}) + \lambda^\frac{1}{2}(\overline{x})\qty(\abs{\overline{x}-\overline{y}}^\frac{1}{2} + \abs{\nu}^\frac{1}{2}(\overline{x})).
	\end{equation}
	In particular, it follows that $\abs{\overline{x}-\overline{y}} \lesssim \lambda(\overline{x}) \implies \lambda(\overline{y}) \lesssim \lambda(\overline{x})$.
	
	Now, fix $\overline{x} \in \Gmt$ and take a cylinder $\calC_{\overline{x}}$ centered at $\overline{x}$ with radius $\varepsilon\lambda(\overline{x})$ and height $2\varepsilon\lambda(\overline{x})$, whose vertical axis is transversal (normal) to $\Gmt$. The compactness of $\Gmt$ and the slowly varying property of $\lambda$ imply that $\Ombd $ can be covered by finitely many such cylinders. The Besicovitch covering theorem ensures that each point belongs to at most $N_d$ such cylinders, where $N_d$ is a dimensional constant. Therefore, it suffices to show that
	\begin{equation}
		J_{\restriction_{\calC_{\overline{x}}}} \lesssim_{A'} B' \scD_{\restriction_{\calC_{\overline{x}}}}.
	\end{equation}
	It follows from the definition of $\lambda$ that
	\begin{equation*}
		\abs{\nu(x)} \lesssim B' \lambda^\frac{3}{2}(\overline{x}) \qand \abs{\grad\nu(x)} \lesssim B' \lambda^\frac{1}{2}(\overline{x}) \qfor x \in \calC_{\overline{x}}.
	\end{equation*}
	The fact that $\mu = \abs{\nu}$ on $\Gmt$ implies
	\begin{equation*}
		\mu(x) \simeq_{A'} \lambda(\overline{x}) \qin \calC_{\overline{x}}.
	\end{equation*}
	One also has the following improved bound in $\calC_{\overline{x}}$:
	\begin{equation*}
		\abs{\nu(x)} \le \abs{\nu(\overline{x})} + \norm{\grad\nu}_{L^\infty(\calC_{\overline{x}})}\abs{x-\overline{x}} \le \abs{\nu(\overline{x})} + \Order{A'}{B'}\lambda^\frac{1}{2}(\overline{x})\delta(x).
	\end{equation*}
	Moreover, one can derive the bound for $\vw$ by:
	\begin{equation*}
		\scD_{\restriction_{\calC_{\overline{x}}}} \gtrsim \int_{\qty{\delta \simeq \lambda(\overline{x})} \cap \calC_{\overline{x}}} \mu^{\alpha}\kappa^{-1}\abs{\vw}^2 \dd{x} \gtrsim \lambda^{d+\alpha}(\overline{x})\fint_{\qty{\delta \simeq \lambda(\overline{x})}\cap\calC_{\overline{x}}} \abs{\vw}^2 \dd{x},
	\end{equation*}
	which yields that
	\begin{equation}\label{point est w}
		\begin{split}
			\abs{\vw(x)} &\lesssim \qty(\fint_{\qty{\delta \simeq \lambda(\overline{x})}\cap\calC_{\overline{x}}} \abs{\vw}^2 \dd{x})^\frac{1}{2} + \norm{\grad\vw}_{L^\infty}\lambda(\overline{x}) \\
			&\lesssim (\scD_{\restriction_{\calC_{\overline{x}}}})^{\frac{1}{2}}\lambda^{-\frac{\alpha+d}{2}}(\overline{x}) + B'\lambda(\overline{x}),
		\end{split}
	\end{equation}
	for all $x \in \calC_{\overline{x}}$.
	Recall the definition of $\lambda$ in \eqref{def lambda}, one can roughly consider two main cases: either the main contributions to $\lambda$ come from the first term on the right, or from the last two items. We call them the \emph{large-$\nu$ scenario} and the \emph{small-$\nu$ scenario}, respectively. 
	
	\paragraph*{Scenario 1. $\nu$ is large}
		
		At present, one has $\lambda(\overline{x}) \simeq \abs{\nu(\overline{x})}$, and, in particular,
		\begin{equation}\label{case 1 nu est}
			\abs{\nu(\overline{x})} \gtrsim \qty((B')^{-1}\abs{\nu(\overline{x})})^\frac{2}{3} \implies B'\abs{\nu(\overline{x})}^\frac{1}{2} \gtrsim 1.
		\end{equation}
		Thus, it can be derived from the pointwise bound \eqref{point est w} that
		\begin{equation}\label{J case 1 first est}
			J_{\restriction_{\calC_{\overline{x}}}} \lesssim \int_{\calC_{\overline{x}}} \mu^\alpha \abs{\vw} \dd{x} \lesssim \lambda^{d+\alpha}(\overline{x})\qty((\scD_{\restriction_{\calC_{\overline{x}}}})^{\frac{1}{2}}\lambda^{-\frac{\alpha+d}{2}}(\overline{x}) + B'\lambda(\overline{x})).
		\end{equation}
		On the other hand, as $\lambda(\overline{x}) \simeq \abs{\nu(x)}$, one has the coercivity estimate
		\begin{equation*}
			\scD_{\restriction_{\calC_{\overline{x}}}} \ge \int_{\calC_{\overline{x}}} \mu^{\alpha-1}\nu^2 \dd{x} \gtrsim \lambda^{\alpha+d+1}(\overline{x}),
		\end{equation*}
		which, together with \eqref{J case 1 first est}, yield that
		\begin{equation*}
			J_{\restriction_{\calC_{\overline{x}}}} \lesssim_{A'} B' \scD_{\restriction_{\calC_{\overline{x}}}} + \lambda^{-\frac{1}{2}}(\overline{x})\scD_{\restriction_{\calC_{\overline{x}}}}.
		\end{equation*}
		Therefore, the bound \eqref{case 1 nu est} implies that
		\begin{equation}\label{J case 1 est}
			J_{\restriction_{\calC_{\overline{x}}}} \lesssim_{A'} B' \scD_{\restriction_{\calC_{\overline{x}}}}.
		\end{equation}
		
		\paragraph*{Scenario 2. $\nu$ is small} 
		
		Here one encounters two subcases, either $\lambda(\overline{x}) \approx \qty((B')^{-1}\abs{\nu}(\overline{x}))^\frac{2}{3}$ or $\lambda(\overline{x}) \approx \qty((B')^{-1}\abs{\grad\nu}(\overline{x}))^2$. Fortunately, for both situations, $\scD_{\restriction_{\calC_{\overline{x}}}}$ admits the coercivity estimate:
		\begin{equation}\label{coercivity cD case 2}
			\scD_{\restriction_{\calC_{\overline{x}}}} \ge \int_{\calC_{\overline{x}}} \mu^{\alpha-1}\nu^2 \dd{x} \gtrsim (B')^2 \lambda^{\alpha+d+2}(\overline{x}).
		\end{equation}
		Indeed, when $\lambda(\overline{x}) \approx \qty((B')^{-1}\abs{\grad{\nu}(\overline{x})})^2$, there exists a generic portion of $\calC_{\overline{x}}$, in which the following estimate holds:
		\begin{equation}\label{case2 point est nu}
			\abs{\nu}(x) \gtrsim B'\lambda^{\frac{3}{2}}(\overline{x}).
		\end{equation} 
		Actually, it is routine to drive that
		\begin{equation}\label{est nu point diff}
			\abs{\nu(x)-\nu(\overline{x})} \le \abs{(x-\overline{x})\vdot\grad{\nu}(\overline{x})} + \Order{A'}{B'}\abs{x-\overline{x}}\qty(\abs{\nu}^\frac{1}{2}(\overline{x}) + \abs{x-\overline{x}}^\frac{1}{2}).
		\end{equation}
		Since $\grad{\nu}(\overline{x})$ is a fixed non-vanishing vector, there exists a generic portion of the cylinder $\calC_{\overline{x}}$ (i.e. those points satisfying $(x-\overline{x}) \not\perp \grad{\nu}(\overline{x})$ and $\abs{x-\overline{x}} \simeq \lambda(\overline{x})$), in which there holds:
		\begin{equation*}
			\abs{\nu(x)-\nu(\overline{x})} \gtrsim B'\lambda^\frac{3}{2}(\overline{x}).
		\end{equation*}
		The smallness of $\abs{\nu(\overline{x})}$ implies \eqref{case2 point est nu}. On the other hand, given $\lambda(\overline{x}) \approx \qty((B')^{-1}\abs{\nu}(\overline{x}))^\frac{2}{3}$, the bound \eqref{coercivity cD case 2} can also be shown from \eqref{est nu point diff}.	With the coercivity bound \eqref{coercivity cD case 2}, one can refine \eqref{point est w} to
		\begin{equation*}
			\abs{\vw(x)} \lesssim \lambda^{-\frac{\alpha+d}{2}}(\overline{x})(\scD_{\restriction_{\calC_{\overline{x}}}})^\frac{1}{2} \qin \calC_{\overline{x}}.
		\end{equation*}
		Since \eqref{est nu point diff} also implies that $\abs{\nu(x)} \lesssim B' \lambda^\frac{3}{2}(\overline{x})$, one can estimate $J_{\restriction_{\calC_{\overline{x}}}}$ by:
		\begin{equation}\label{J case 2 est}
			\begin{split}
				J_{\restriction_{\calC_{\overline{x}}}} &= \int_{\calC_{\overline{x}}} \mu^{\alpha-3}\abs{\nu}^3\abs{\vw} \dd{x} \\
				&\lesssim \int_{\calC_{\overline{x}}} \mu^{\alpha-2} \qty(B'\lambda^\frac{3}{2}(\overline{x}))^2 \abs{\vw} \dd{x} \\
				&\lesssim_{A'} \lambda^{\alpha-2+3+d}(\overline{x})(B')^2\lambda^{-\frac{\alpha+d}{2}}(\overline{x})(\scD_{\restriction_{\calC_{\overline{x}}}})^\frac{1}{2} \\
				&\lesssim_{A'} B' \scD_{\restriction_{\calC_{\overline{x}}}},
			\end{split}
		\end{equation}
		whose last inequality follows from the coercivity bound \eqref{coercivity cD case 2}.
	
	In conclusion, the energy estimate \eqref{energy est fD cD} follows from the combination of \eqref{est err}, \eqref{def I1 I2 I3}, \eqref{est I1}-\eqref{goal relation mu nu}, \eqref{est J Om int}, \eqref{J case 1 est}, and \eqref{J case 2 est}.
	
	\end{proof}

	\section{A Priori Estimates}\label{sec energy est}
	
	Let $k \ge 1$ be an integer. The objective of this section is to establish the uniform control of $\fkH^{2k}$ norms of $(q, \vv, \sigma)$ using the initial data and control parameters. The main theorem of this section is presented in \S\ref{sec energies and control para}. Before that, we will discuss some heuristics and technical preparations.
	
	\subsection{Simplified Linear Systems}
	
	Consider the following simplified linearized system of \eqref{euler equiv2} with source terms:
	\begin{equation}\label{simple lin sys}
		\begin{cases*}
			\Dt s + \grad_{\vw} q + \beta q (\divergence\vw) = f, \\
			\Dt \vw + \sigma \grad{s} = \vb*{g},
		\end{cases*}
	\end{equation}
	where $s$ and $\vw$ are the linearized variables of $q$ and $\vv$, respectively. The above system can be regarded as the linearization around a background solution with constant velocity and entropy. Define the $L^2$-based energy functional for \eqref{simple lin sys} by
	\begin{equation}
		\cE (s, \vw) \coloneqq \frac{1}{2}\int_{\Omt} q^{\alpha-1} \abs{s}^2 + \beta q^{\alpha}\sigma^{-1}\abs{\vw}^2 \dd{x}.
	\end{equation}
	Then, it follows from the computations similar to \eqref{est energy lin2} that
	\begin{equation}\label{est simp lin}
		\dv{t}\cE = \int_{\Omt}\qty[ q^{\alpha-1}(sf) + \beta q^{\alpha} \sigma^{-1} (\vw\vdot \vb*{g})] \dd{x} + \order{\norm{\grad\vv}_{L^\infty}}\cE.
	\end{equation}
	The estimate \eqref{est simp lin} will be used to control the top order norms.
	
	Note that the direct usage of $(s, \vw) \coloneqq (q^k\pd^{2k}q, q^k\pd^{2k}\vv)$ may cause the source terms $(f, \vb*{g})$ in \eqref{simple lin sys} uncontrollable. Since the compressible Euler equations can lead to wave-type equations for $q$, one can try to use $\Dt^2 q$ to represent $q\pd^2 q$ at a leading order. Indeed, observe that
	\begin{equation}\label{eqn Dt2 q}
		\Dt^2 q = \beta q\sigma \laplace{q} + \beta q \grad{q}\vdot\grad{\sigma} + \beta q \qty(\beta(\divergence\vv)^2 + \tr\qty[(\grad\vv)^2]).
	\end{equation}
	Given the adapted elliptic estimates, it is hopeful to use $\Dt^{2k} q$ as a replacement of $q^k\pd^{2k}q$ in the energy estimates. It clear that the commutator $\comm{\Dt}{\Dt^{2k}} \equiv 0$. Thus, applying $\Dt^{2k} q$ can preserve much more symmetry of the linearized system. Similarly, one can calculate that
	\begin{equation}\label{eqn Dt2 v}
		\Dt^2 \vv = \beta q\sigma \grad(\divergence\vv) + \sigma\qty[\beta(\divergence\vv)+(\grad\vv)^*]\grad{q},
	\end{equation}
	where $(\grad\vv)^*$ represents  the transpose of the second order tensor $(\grad\vv)$. Therefore, it is possible to utilize $\Dt^{2k} \vv$ to control  $q^k\pd^{2k-1}(\divergence{\vv}) $. Thanks to the div-curl estimates, it suffices to control further $q^k \pd^{2k-1}(\Curl\vv)$ and $q^k\pd^{2k}\sigma$. Because both $(\Curl{\vv})$ and $\sigma$ satisfy transport equations, the estimates of them would not be the main obstacles. 
	
	In order to control the remainders, we first consider the interpolations adapted to the physical vacuum problems.
	
	\subsection{Interpolation Inequalities}
	Here we collect some results concerning the interpolations, whose derivations can be found in \cite[\S2.5]{Ifrim-Tataru2024}. Suppose that $\Omega$ is a $C^{1+}$ domain with a non-degenerate defining function $r \in C^{1+}(\overline{\Omega})$. Then, there hold:
	\begin{prop}\label{prop interp}
		Assume that $1 \le p_0, p_m \le \infty$, $\lambda_0 > - \frac{1}{p_0}$, $\lambda_m > - \frac{1}{p_m}$ and
		\begin{equation*}
			m - \lambda_m - \frac{d}{p_m} > -\lambda_0 - \frac{d}{p_0}.
		\end{equation*}
		Define
		\begin{equation*}
			\theta_j \coloneqq \frac{j}{m}\qc \frac{1}{p_j} \coloneqq \frac{\theta_j}{p_m} + \frac{1-\theta_j}{p_0} \qc \lambda_j \coloneqq \theta_j \lambda_m + (1-\theta_j)\lambda_0.
		\end{equation*}
		Then, for $0 < j < m$, it holds that
		\begin{equation*}
			\norm*{r^{\lambda_j}\pd^j f}_{L^{p_j}} \lesssim \norm*{r^{\lambda_0}f}_{L^{p_0}}^{1-\theta_j}\cdot\norm*{r^{\lambda_m}\pd^m f}_{L^{p_m}}^{\theta_j}.
		\end{equation*}
	\end{prop}
	\begin{prop}\label{interp L infty}
		Let $1 < p_m < \infty$, $-\frac{1}{p_m} < \lambda_m < m- \frac{d}{p_m}$, 
		\begin{equation*}
			\theta_j \coloneqq \frac{j}{m}\qc \frac{1}{p_j} \coloneqq \frac{\theta_j}{p_m}, \qand \lambda_j \coloneqq \theta_j \lambda_m.
		\end{equation*}
		Then, it holds for $0 < j < m$ that
		\begin{equation*}
			\norm*{r^{\lambda_j}\pd^j f}_{L^{p_j}} \lesssim \norm{f}_{L^\infty}^{1-\theta_j}\cdot \norm*{r^{\lambda_m}\pd^m f}_{L^{p_m}}^{\theta_j}.
		\end{equation*}
	\end{prop}
	
	\begin{prop}\label{interp C1/2}
		Let $1 < p_m < \infty$, $-\frac{1}{p_m} < \lambda_m < m - \frac{1}{2}-\frac{d}{p_m}$,
		\begin{equation*}
			\theta_j \coloneqq \frac{j-\frac{1}{2}}{m-\frac{1}{2}}\qc \frac{1}{p_j} \coloneqq \frac{\theta_j}{p_m}, \qand \lambda_j \coloneqq \theta_j\lambda_m.
		\end{equation*}
		Then, it holds for $0 < j < m$ that
		\begin{equation*}
			\norm*{r^{\lambda_j}\pd^j f}_{L^{p_j}} \lesssim \norm{f}_{\dot{C}^{\frac{1}{2}}}^{1-\theta_j} \cdot \norm*{r^{\lambda_m}\pd^m f}_{L^{p_m}}^{\theta_j}.
		\end{equation*}
	\end{prop}
	
	\begin{prop}\label{interp Ctilde1/2}
		Let $1 < p_m < \infty$, $\frac{m}{2}-1 < \lambda_m < m-\frac{1}{2}-\frac{d}{p_m}$, 
		\begin{equation*}
			\theta_j \coloneqq \frac{j}{m}\qc \frac{1}{p_j} \coloneq \frac{\theta_j}{p_m}, \qand \lambda_j \coloneqq \theta_j\lambda_m + (1-\theta_j)\qty(-\frac{1}{2}).
		\end{equation*}
		Then, it holds for $0 < j < m$ that
		\begin{equation*}
			\norm*{r^{\lambda_j}\pd^j f}_{L^{p_j}} \lesssim \norm{f}_{\wtC^{\frac{1}{2}}}^{1-\theta_j} \cdot \norm*{r^{\lambda_m}\pd^m f}_{L^{p_m}}^{\theta_j}.
		\end{equation*}
	\end{prop}
	
	\begin{remark}
		Recall that the seminorm $\norm{f}_{\wtC^\frac{1}{2}}$ is defined by
		\begin{equation*}
			\norm{f}_{\wtC^{\frac{1}{2}}(\Omega)} \coloneqq \sup_{x, y \in \Omega}\ \frac{\abs{f(x)-f(y)}}{r^\frac{1}{2}(x) + r^{\frac{1}{2}}(y) + \abs{x-y}^\frac{1}{2}}.
		\end{equation*}
		For interpolations, it behaves formally like $\norm*{r^{-\frac{1}{2}}f}_{L^\infty}$. Such a seminorm will be used to overcome the deficiency of $q$-factors.
	\end{remark}
	
	\subsection{Constructing the Energy Functional}\label{sec construct energy}
	
	First, recall the scaling analysis of orders:
	\begin{itemize}
		\item $q$, $\vv$, and $\sigma$ have orders $-1$, $-\frac{1}{2}$, and $0$, respectively;
		\item $\pd_x$ and $\Dt$ have orders $+1$ and $+\frac{1}{2}$, respectively.
	\end{itemize}
	Define the uniformly-scaling $2k$-norm of $(q, \vv, \sigma)$ by
	\begin{equation}
		\norm{(q, \vv, \sigma)}_{{\mathscr{H}}^{2k}_{\sharp}} \coloneqq \qty(\norm{q}_{H^{2k, k +\frac{\alpha-1}{2}}}^2 + \norm{\vv}_{H^{2k, k+\frac{\alpha}{2}}}^2 + \norm{\sigma}_{H^{2k, k+\frac{\alpha+1}{2}}}^2)^\frac{1}{2}.
	\end{equation}
	Note that the norms of  $\sigma$ here contain more weights than those used in the state spaces. Hence, there holds
	\begin{equation*}
		\norm{(q, \vv, \sigma)}_{{\mathscr{H}}^{2k}_{\sharp}} \lesssim_{\norm{q}_{L^\infty}} \norm{(q, \vv, \sigma)}_{\fkH^{2k}_{q}}.
	\end{equation*}
	
	\subsubsection{Admissible errors}
	
	In higher order energy estimates, it is crucial to identify which terms can be considered as small errors. Generally, one might aim to bound the lower order terms by the norms of the highest order derivatives using interpolations. However, the interpolation relations in the weighted Sobolev spaces we encounter here differ from the classical ones due to the decaying property of weights near the boundary. Thus, it makes sense to use classical interpolations away from the boundary and the previously mentioned interpolations near the free boundary. Unfortunately, when applying Proposition \ref{prop interp}, the coefficients of the highest order norms on the right-hand side cannot be made as small as needed. To address this, one can first consider the sharp control parameter:
	\begin{equation}\label{def A sharp}
		A^\sharp \coloneqq \norm{\grad{q}-\vb*{\vartheta}}_{L^\infty} + \norm{\vv}_{\dot{C}^\frac{1}{2}} + \norm{\sigma - \gamma}_{L^\infty},
	\end{equation}
	where $\vb*{\vartheta}$ and $\gamma$ are constants to be determined. Then, one can anticipate that $A^\sharp \ll 1$ when restricted to a rather small compact set (by assigning $\vb*{\vartheta} \coloneqq \grad{q}(x_0)$ and $\gamma \coloneqq \sigma(x_0)$ for some point $x_0$), as $A^\sharp$ actually characterizes the modulus of continuities of $(q, \vv, \sigma)$ there. In particular, during the interpolation procedures, the coefficients of the highest order norms can be sufficiently small, at least in a small region close to the boundary.
	
	Quantitatively speaking, consider a multilinear form $\calM = \calM(\scrJ q, \scrJ (\pd\vv), \scrJ(\pd\sigma))$, where $\scrJ f$ represents the jets of $f$:
	\begin{equation*}
		\scrJ f \coloneqq \qty{\pd^\tau f}_{\{\abs{\tau} \ge 0\}}.
	\end{equation*}
	In particular, the multilinear form $\calM$ does not contain undifferentiated $\vv$ or $\sigma$ factors. The following lemma gives an explicit version of the above heuristic arguments:
	\begin{lemma}\label{lem err est A}
		Let $\calM = \calM(\scrJ q, \scrJ (\pd\vv), \scrJ(\pd\sigma)) $ be a multilinear form of order $(k-1)$ and involve exactly $2k$ spacial derivatives. Namely, $\calM$ can be written as (here all indices are positive integers):
		\begin{equation}\label{expression M A}
			\calM = q^\mu (\pd q)^\lambda \prod_j^J \pd^{a_j} q \prod_l^L \pd^{b_l} \vv \prod_m^M \pd^{c_m} \sigma\qc a_j \ge 2, b_l, c_m \ge 1, 
		\end{equation}
		with
		\begin{equation}\label{rel a b c}
			\begin{cases*}
				-\mu-J-\frac{L}{2} + \sum a_j + \sum b_l + \sum c_m = k-1, \\
				\lambda + \sum a_j + \sum b_l + \sum c_m = 2k.
			\end{cases*}
		\end{equation}
		Then, it holds that
		\begin{equation}\label{est multilinear M}
			\norm{\calM}_{H^{0, \frac{\alpha-1}{2}}} \lesssim  (A^\sharp)^{(J+L+M-1)} \norm{\grad{q}}_{L^\infty}^\lambda \norm{(q, \vv, \sigma)}_{{\mathscr{H}}^{2k}_{\sharp}}.
		\end{equation}
		
		Similarly, if
		\begin{equation}\label{ abc 2}
			\begin{cases*}
				-\mu-J-\frac{L}{2} + \sum a_j + \sum b_l + \sum c_m = k-\frac{1}{2}, \\
				\lambda + \sum a_j + \sum b_l + \sum c_m = 2k,
			\end{cases*}
		\end{equation}
		then $\calM$ satisfies the estimate
		\begin{equation}\label{est M 2}
			\norm{\calM}_{H^{0, \frac{\alpha}{2}}} \lesssim  (A^\sharp)^{(J+L+M-1)} \norm{\grad{q}}_{L^\infty}^\lambda \norm{(q, \vv, \sigma)}_{{\mathscr{H}}^{2k}_{\sharp}}.
		\end{equation}
	\end{lemma}
	\begin{proof}
		One can first consider the proof of \eqref{est multilinear M} under \eqref{rel a b c}, and the other case will follow from the same arguments.	In order to apply Propositions \ref{interp L infty}-\ref{interp C1/2}, one may rewrite $\calM$ as:
		\begin{equation}\label{decomp M}
			\calM = (\pd q)^\lambda \prod_j^{J} q^{\nu_j}\pd^{a_j}q \prod_l^{L} q^{\nu_{l}} \pd^{b_l}\vv \prod_m^{M} q^{\nu_m}\pd^{c_m} \sigma.
		\end{equation}
		One can formally consider the interpolations:
		\begin{equation*}
			\norm{q^{\frac{\alpha-1}{p_j}} q^{\nu_j} \pd^{a_j -1} \pd q}_{L^{p_j}} \lesssim \norm{\pd q -\vb*{\vartheta}}_{L^\infty}^{1-\theta_j} \norm{q^{\gamma_j}\pd^{\beta_j}\pd q}_{L^2}^{\theta_j},
		\end{equation*}
		with
		\begin{equation*}
			\theta_j = \frac{a_j -1}{\beta_j}\qc \frac{1}{p_j} = \frac{\theta_j}{2}, \qand \gamma_j = \frac{\nu_j}{\theta_j} + \frac{\alpha-1}{p_j \theta_j}.
		\end{equation*}
		To control the right hand side by $(A^\sharp)^{1-\theta_j}\norm{q}_{H^{2k, k+\frac{\alpha-1}{2}}}^{\theta_j}$, one can infer from Lemma \ref{lem inclusion} that it suffices to have
		\begin{equation*}
			1+\beta_j - \gamma_j = 2k - k - \qty(\frac{\alpha-1}{2}),
		\end{equation*}
		i.e.,
		\begin{equation*}
			\theta_j = \frac{a_j - 1 - \nu_j}{k-1} \qand \beta_j = \frac{(a_j -1)(k-1)}{a_j -1 - \nu_j}.
		\end{equation*}
		Therefore, it follows that
		\begin{equation*}
			\norm{q^{\frac{\alpha-1}{p_j}} q^{\nu_j} \pd^{a_j -1} \pd q}_{L^{\frac{2(k-1)}{a_j -1-\nu_j}}} \lesssim \norm{\grad{q}-\vb*{\vartheta}}_{L^\infty}^{1-\frac{a_j -1-\nu_j}{k-1}}\cdot\norm{q}_{H^{2k, k+\frac{\alpha-1}{2}}}^{\frac{a_j -1-\nu_j}{k-1}}.
		\end{equation*}
		Similarly, one can derive from Propositions \ref{interp L infty}-\ref{interp C1/2} and Lemma \ref{lem inclusion} that
		\begin{equation*}
			\norm{q^{\frac{\alpha-1}{r_l}}q^{\nu_l}\pd^{b_l}\vv}_{L^{r_l}} \lesssim \norm{\vv}_{\dot{C}^\frac{1}{2}}^{1-\theta_l}\norm{\vv}_{H^{2k, k+\frac{\alpha}{2}}}^{\theta_l},
		\end{equation*}
		for
		\begin{equation*}
			\theta_l = \frac{b_l - \nu_l - \frac{1}{2}}{k-1}\qc \frac{1}{r_l} = \frac{\theta_l}{2};
		\end{equation*}
		and
		\begin{equation*}
			\norm{q^{\frac{\alpha-1}{s_m}}q^{\nu_m}\pd^{c_m}\sigma}_{L^{s_m}} \lesssim \norm{\sigma - \gamma}_{L^\infty}^{1-\theta_m}\norm{\sigma}_{H^{2k, k+\frac{\alpha+1}{2}}}^{\theta_m},
		\end{equation*}
		for
		\begin{equation*}
			\theta_m = \frac{c_m - \nu_m}{k - 1}\qc \frac{1}{s_m} = \frac{\theta_m}{2}.
		\end{equation*}
		Note that \eqref{rel a b c} yields that
		\begin{equation*}
			\sum \theta_j + \sum \theta_l + \sum \theta_m = 1,
		\end{equation*}
		and hence \eqref{est multilinear M}, as long as the applications of interpolations and inclusions to the decomposition \eqref{decomp M} are legitimate. Indeed, it remains to check the existence of a decomposition 
		\begin{equation*}
			\mu = \sum \nu_j + \sum\nu_l + \sum\nu_m,
		\end{equation*} 
		for which
		\begin{equation*}
			\begin{cases*}
				\nu_j \ge 0, \\
				\frac{a_j -1}{\theta_j} \le 2k-1,
			\end{cases*}
			\quad
			\begin{cases*}
				\nu_l \ge 0, \\
				\frac{b_l-\frac{1}{2}}{\theta_l} \le 2k-\frac{1}{2},
			\end{cases*}
			\qand
			\begin{cases*}
				\nu_m \ge 0, \\
				\frac{c_m}{\theta_m} \le 2k.
			\end{cases*}
		\end{equation*}
		In other words, one ought to verify that
		\begin{equation*}
			0 \le \nu_j \le \frac{k}{2k-1}(a_j-1)\qc 0\le \nu_l \le \frac{k+\frac{1}{2}}{2k-\frac{1}{2}}\qty(b_l - \frac{1}{2}), \qand 0 \le \nu_m \le \frac{k+1}{2k}c_m.
		\end{equation*}
		It suffices to check that
		\begin{equation}\label{range mu}
			0 \le \mu \le \sum_j \frac{k}{2k-1}(a_j-1) + \sum_l \frac{k+\frac{1}{2}}{2k-\frac{1}{2}}\qty(b_l - \frac{1}{2}) + \sum_m \frac{k+1}{2k}c_m.
		\end{equation}
		Indeed, elementary calculations imply that
		\begin{equation*}
			\begin{split}
				&\sum \frac{k}{2k-1}(a_j-1) + \sum \frac{k+\frac{1}{2}}{2k-\frac{1}{2}}\qty(b_l - \frac{1}{2}) + \sum \frac{k+1}{2k}c_m \\
				&\ = \qty(\frac{1}{2} + \frac{1}{4k-2})\sum(a_j -1) + \qty(\frac{1}{2} + \frac{3}{4k-2})\sum\qty(b_l - \frac{1}{2}) + \qty(\frac{1}{2}+\frac{2-\frac{1}{k}}{4k-2})\sum c_m \\
				&\ \ge \qty(\frac{1}{2} + \frac{1}{4k-2})\qty(\sum a_j + \sum b_l + \sum c_m - J - \frac{L}{2}) \\
				&\ \ge \frac{k(k-1)}{2k-1} + \frac{k}{2k-1}\mu.
			\end{split}
		\end{equation*}
		In particular, \eqref{range mu} holds for $0 \le \mu \le k$ (and the equality holds only if $\mu = k, J = 1$, and $L = M = 0$), which is ensured by \eqref{rel a b c}.
		
	\end{proof}
	
	Define the control parameters:
	\begin{equation}\label{def A0}
		A_0 \coloneqq \norm{\grad{q}}_{L^\infty} + \norm{\vv}_{\dot{C}^\frac{1}{2}} + \norm{\sigma}_{L^\infty},
	\end{equation}
	and
	\begin{equation}\label{def B sharp}
		B^\sharp \coloneqq \norm{\grad q}_{\wtC^{\frac{1}{2}}}
		+ \norm{\grad \vv}_{L^\infty} + \norm{q^\frac{1}{2}\grad\sigma}_{L^\infty}.
	\end{equation}
	The following variation of the above lemma also holds:
	\begin{lemma}\label{lem err est B}
		Suppose that $\calM = \calM(\scrJ q, \scrJ(\pd\vv), \scrJ(\pd\sigma)) $ is a multilinear form of order $(k-\frac{1}{2})$, and it involves exactly $(2k+1)$ spacial derivatives. In other words, $\calM$ has the form \eqref{expression M A} with
		\begin{equation}\label{rel a b c 2}
			\begin{cases*}
				-\mu-J-\frac{L}{2} + \sum a_j + \sum b_l + \sum c_m = k-\frac{1}{2}, \\
				\lambda+\sum a_j + \sum b_l + \sum c_m = 2k + 1.
			\end{cases*}
		\end{equation}
		Then, if $(J + L + M) \ge 2$, it follows that
		\begin{equation}\label{est M B}
			\norm{\calM}_{H^{0, \frac{\alpha-1}{2}}} \lesssim (A_0)^{(\lambda+J+L+M-2)}  B^\sharp \norm{(q, \vv, \sigma)}_{{\mathscr{H}}^{2k}_{\sharp}}.
		\end{equation}
		
		Similarly, if \eqref{rel a b c 2} is replaced by
		\begin{equation}
			\begin{cases*}
				-\mu-J-\frac{L}{2} + \sum a_j + \sum b_l + \sum c_m = k, \\
				\lambda+\sum a_j + \sum b_l + \sum c_m = 2k + 1,
			\end{cases*}
		\end{equation}
		then, it holds for $(J + L + M) \ge 2$ that
		\begin{equation}\label{est M B 2}
			\norm{\calM}_{H^{0, \frac{\alpha}{2}}} \lesssim (A_0)^{(\lambda+J+L+M-2)}  B^\sharp \norm{(q, \vv, \sigma)}_{{\mathscr{H}}^{2k}_{\sharp}}.
		\end{equation}
	\end{lemma}
	\begin{proof}
		It follows from \eqref{rel a b c 2} that
		\begin{equation*}
			\lambda + \mu + J + \frac{L}{2} = k +\frac{3}{2}.
		\end{equation*}
		Thus, $L$ must be a positive odd integer. If there exists a factor being exactly $\pd\vv$ or $\pd\sigma$, whose $L^\infty$ norms can be bounded by $B^\sharp$, \eqref{est M B} follows directly from Lemma \ref{lem err est A}. Now, suppose that $(J+L+M)= 2$, and one is about to estimate
		\begin{equation*}
			\norm{q^{k} \cdot \pd^2 q \cdot \pd^{2k-1}\vv}_{H^{0, \frac{\alpha-1}{2}}}.
		\end{equation*}
		It can be deduced from Propositions \ref{interp L infty}, \ref{interp Ctilde1/2}, and H\"older's inequality that
		\begin{equation*}
			\begin{split}
				&\norm{q^{\frac{\alpha-1}{2}+k} \cdot \pd^2 q \cdot \pd^{2k-1} \vv }_{L^2} \\
				&\ \le \norm{q^{\frac{2k-2}{2k-1}\qty(\frac{\alpha}{2}+k)} \cdot \pd^{2k-2}\pd\vv}_{L^{\frac{4k-2}{2k-2}}} \cdot \norm{ q^{\qty[\frac{1}{2k-1}\qty(\frac{\alpha-1}{2}+k)-\frac{1}{2}\qty(1-\frac{1}{2k-1})]}\cdot \pd \pd q}_{L^{4k-2}} \\
				&\ \lesssim \norm{\pd\vv}_{L^\infty}^{\frac{1}{2k-1}}\norm{q^{\frac{\alpha}{2}}q^k\pd^{2k}\vv}_{L^2}^{\frac{2k-2}{2k-1}} \cdot \norm{\pd q}_{\wtC^{\frac{1}{2}}}^{1-\frac{1}{2k-1}}\norm{q^{\frac{\alpha-1}{2}}q^k\pd^{2k} q}_{L^2}^{\frac{1}{2k-1}} \\
				&\ \lesssim B^\sharp \norm{(q, \vv, \sigma)}_{{\mathscr{H}}^{2k}_{\sharp}}.
			\end{split}
		\end{equation*}
		It can be seen that $\norm{\grad q}_{\wtC^\frac{1}{2}}$ is necessary for such interpolation, and $\norm{\grad{q}}_{\dot{C}^\frac{1}{2}}$	is not an admissible candidate here. Other terms with $(J+L+M)=2$ can be treated in a similar manner.  For more general items, one can decompose $\calM$ as in \eqref{decomp M} and use Propositions \ref{prop interp}-\ref{interp Ctilde1/2} to derive the estimate as in the proof of Lemma \ref{lem err est A}. Note that two-step interpolations might be needed here to balance the indices, one step uses $B^\sharp$, the other step uses $A_0$. Scaling analysis yields that exactly one $B^\sharp$ factor appears on the right hand side of \eqref{est M B} (see \cite{Ifrim-Tataru2024}). The estimate \eqref{est M B 2} can be established with the same arguments.
		
	\end{proof}
	
	\begin{remark}
		Lemmas \ref{lem err est A}-\ref{lem err est B} hint that a multilinear form $\calM$ written as \eqref{expression M A} can be controlled by interpolations with the desired control parameters, whenever $(J+L+M) \ge 2$. Such terms are called ``balanced terms'' in \cite{Ifrim-Tataru2024}. Note that $(J+L+M)$ is exactly the amount of those factors having strictly positive order. 
	\end{remark}
	
	\begin{defi}[Admissible errors]\label{def adm err}
		Let  $\calM = \calM(\scrJ q, \scrJ(\pd\vv), \scrJ(\pd\sigma)) $ be a multilinear form written as \eqref{expression M A}:
		\begin{equation*}
			\calM = q^\mu (\pd q)^\lambda \prod_{j}^J \pd^{a_j} q \prod_{l}^L \pd^{b_l} \vv \prod_{m}^M \pd^{c_m} \sigma\qc a_j \ge 2, b_l, c_m \ge 1. 
		\end{equation*}
		Then, $\calM$ is called an \emph{admissible error} if $(J+L+M) \ge 2$, i.e., $\calM$ contains at least two factors having strictly positive orders.
	\end{defi}
	
	\subsubsection{Good unknowns}
		Due to Lemmas \ref{lem err est A}-\ref{lem err est B}, it is known that admissible errors could be controlled. Now, one can consider the evolution equations for $(\Dt^{2k}q, \Dt^{2k}\vv)$ in the manner of \eqref{simple lin sys}. Observe that
		\begin{equation*}
			\Dt \qty(\Dt^{2k}q) = -\beta\Dt^{2k}\qty\big[q(\divergence\vv)].
		\end{equation*}
		The Leibniz rule yields that all multilinear forms involved in
		\begin{equation*}
			\Dt^{2k}\qty\big[q(\divergence{\vv})] - q \Dt^{2k}(\divergence{\vv})
		\end{equation*}
		are admissible errors in the sense of Definition \ref{def adm err}. The commutator formula
		\begin{equation*}
			\comm{\Dt}{\divergence} \vw = - \tr(\grad\vv \vdot \grad\vw)
		\end{equation*}
		implies that the multilinear forms in
		\begin{equation*}
			q\Dt^{2k}(\divergence{\vv}) - q \divergence{(\Dt^{2k}\vv)}
		\end{equation*}
		are all admissible errors. Unfortunately, the term
		\begin{equation*}
			\grad q \vdot \Dt^{2k} \vv
		\end{equation*}
		is not admissible, when all spacial derivatives fall onto $\vv$. Thus, as a substitution, one can consider the unknown (see also \cite{Ifrim-Tataru2024}):
		\begin{equation*}
			s_{2k} \coloneqq \Dt^{2k}q - \grad q \vdot \Dt^{2k-1}\vv.
		\end{equation*}
		It is straightforward to check that the commutator
		\begin{equation*}
			\comm{\Dt}{\grad{q}} \vdot \Dt^{2k-1}\vv
		\end{equation*}
		is admissible whenever $k \ge 2$. If $k = 1$, one can modify the definition of $s_2$ to 
		\begin{equation*}
			s_2 \coloneqq \Dt^2 q - \frac{1}{2}\grad{q} \vdot \Dt\vv.
		\end{equation*}
		For the evolution of $\Dt^{2k}\vv$, note that
		\begin{equation*}
			\begin{split}
				\Dt \Dt^{2k}\vv + \sigma \grad{s_{2k}} = &- \sigma \Dt^{2k}\grad{q} + \sigma\grad(\Dt^{2k}q - \grad{q}\vdot \Dt^{2k-1}\vv) \\ 
				= &-\sigma\comm{\Dt^{2k-1}}{\grad}\Dt q - \sigma\Dt^{2k-1}\comm{\Dt}{\grad}q \\
				&- \sigma\grad{q} \vdot \grad{\Dt^{2k-1}\vv} - \sigma\grad^2 q \vdot \Dt^{2k-1}\vv
			\end{split}
		\end{equation*}
		It can be derived from the commutator formula
		\begin{equation}\label{comm Dt D}
			\comm{\Dt}{\grad} \varphi = - (\grad\vv)^* \grad{\varphi}
		\end{equation}
		that the term
		\begin{equation*}
			\sigma\Dt^{2k-1}\comm{\Dt}{\grad}q + \sigma \qty(\grad{\Dt^{2k-1}\vv})^*\grad{q}
		\end{equation*}
		is admissible. Therefore, $(\Dt \Dt^{2k}\vv + \sigma \grad{s_{2k}})$ itself is also an admissible error.
		
		In summary, one can consider the following ``good unknowns''
		\begin{equation}\label{def good unknown}
			\begin{cases*}
				s_2 \coloneqq \Dt^2 q - \frac{1}{2}\grad{q} \vdot \Dt\vv, \\
				s_{2k} \coloneqq \Dt^{2k}q - \grad q \vdot \Dt^{2k-1}\vv \quad (k \ge 2),
			\end{cases*}
			\qand
			\vw_{2k} \coloneqq \Dt^{2k}\vv  \quad (k \ge 1).
		\end{equation}
		Then, it follows from the above arguments that $(s_{2k}, \vw_{2k})$ satisfies the linear system \eqref{simple lin sys} with admissible source terms. Namely, one has
		\begin{equation}
			\begin{cases*}
				\Dt s_{2k} + \vw_{2k}\vdot\grad{q} + \beta q (\divergence\vw_{2k}) = f_{2k}, \\
				\Dt \vw_{2k} + \sigma \grad{s_{2k}} = \vb*{g}_{2k},
			\end{cases*}
		\end{equation}
		where $f_{2k}$ and $\vb*{g}_{2k}$ are multilinear forms of $\scrJ(q, \pd\vv, \pd\sigma)$ having at least two factors of strictly positive order.
		
	\subsection{Expressions of the Energy Functional}\label{sec energies and control para}
	Denote by $\vom$ the vorticity of the gas:
	\begin{equation*}
		\vom \coloneqq \Curl \vv.
	\end{equation*}
	Namely, $\vom$ is a second order anti-symmetric tensor given by
	\begin{equation*}
		\omega_{ij} \coloneqq \pd_i v_j - \pd_j v_i.
	\end{equation*}
	Based on the derivations in \S\ref{sec construct energy}, one can consider the energy functional defined via
	\begin{equation}\label{evo eqn s2k w2k}
		\begin{split}
			\cE_{2k} \coloneqq &\int_{\Omt} q^{\alpha-1} \abs{s_{2k}}^2 + \beta q^{\alpha}\sigma^{-1}\abs{\vw_{2k}}^2 \dd{x} \\
			 &\quad +\int_{\Omt} q^{\alpha+1} + \frac{\beta+1}{2}q^\alpha\sigma^{-1}\abs{\vv}^2 + q^\alpha\sigma^2 \dd{x} \\
			 &\quad+ \norm{\vom}_{H^{2k-1, k+\frac{\alpha}{2}}}^2 + \norm{\sigma}_{H^{2k, k+\frac{\alpha}{2}}}^2,
		\end{split}
	\end{equation}
	where $s_{2k}$ and $\vw_{2k}$ are given by \eqref{def good unknown}.
	
	Recall the control parameters $A_*$ and $B$ defined by \eqref{def A*} and \eqref{def B} respectively:
	\begin{equation*}
		A_* \coloneqq \norm{q}_{C^{1+\varepsilon_*}} + \norm{\vv}_{{C}^{\frac{1}{2}+\varepsilon_*}} + \norm{\sigma}_{C^{\frac{1}{2}+\varepsilon_*}} + \norm*{\sigma^{-1}}_{L^\infty}
	\end{equation*}
	and
	\begin{equation*}
		B \coloneqq \norm{\grad{q}}_{\wtC^{0, \frac{1}{2}}} + \norm{\grad\vv}_{L^\infty} + \norm{\grad{\sigma}}_{L^\infty}.
	\end{equation*}
	For the control parameters defined in  \eqref{def A sharp} and \eqref{def A0}-\eqref{def B sharp}, it is obvious that
	\begin{equation}\label{rel ctrl paras}
		A^\sharp, A_0 \le A_* \qand
		B^\sharp \lesssim_{A_*} B.
	\end{equation}
	
	The main objective of this whole section is to show:
	\begin{theorem}\label{thm energy est}
		The energy functional $\cE_{2k}$ satisfies the following two properties:
		\begin{enumerate}[label=(\roman*)]
			\item \textbf{Coercivity:} 
			\begin{equation}\label{energy coercivity est}
				\cE_{2k} \simeq_{A_*, c_0} \norm{(q, \vv, \sigma)}_{\fkH^{2k}_{q}}^2.
			\end{equation}
			\item \textbf{Propagation estimate:}
			\begin{equation}\label{energy est}
				\abs{\dv{t} \cE_{2k}} \lesssim_{A_*} B \norm{(q, \vv, \sigma)}_{\fkH^{2k}_q}^2.
			\end{equation}
		\end{enumerate}
	\end{theorem}

	\subsection{Propagation Estimates}
	
	It follows from \eqref{evo eqn s2k w2k}, the linear estimate \eqref{est simp lin}, the admissibility of $(f_{2k}, \vb*{g}_{2k})$, the relation between the control parameters \eqref{rel ctrl paras}, and interpolation Lemmas \ref{lem err est A}-\ref{lem err est B} that
	\begin{equation*}
		\abs{\dv{t}\int_{\Omt} q^{\alpha-1} \abs{s_{2k}}^2 + \beta q^{\alpha}\sigma^{-1}\abs{\vw_{2k}}^2 \dd{x}} \lesssim_{A_*}{B}\norm{(q, \vv, \sigma)}_{\fkH^{2k}_{q}}^2.
	\end{equation*}
	Furthermore, it can be deduced directly from the compressible Euler system \eqref{euler equiv2} that
	\begin{equation*}
		\dv{t}\int_{\Omt} q^{\alpha+1} + \frac{\beta+1}{2}q^\alpha\sigma^{-1}\abs{\vv}^2 + q^\alpha\sigma^2 \dd{x} = 0.
	\end{equation*}
	To show the evolution estimate of $\vom$, one first observes that
	\begin{equation*}
		\Dt\vom + (\grad\vv)\vdot\vom + \vom\vdot(\grad\vv)^* + \grad{\sigma} \wedge \grad{q} = 0,
	\end{equation*}
	where $\grad{\sigma}\wedge\grad{q}$ is an anti-symmetric tensor given by
	\begin{equation*}
		(\grad{\sigma}\wedge\grad{q})_{ij} = \pd_i \sigma \cdot \pd_j q - \pd_j \sigma \cdot \pd_i q.
	\end{equation*}
	Thus, it follows that
	\begin{equation*}
		\Dt (q^k\pd^{2k-1}\vom) = \comm{\Dt}{q^k \pd^{2k-1}}\vom - q^k \pd^{2k-1} \qty\big[(\grad\vv)\vdot\vom + \vom\vdot(\grad\vv)^* + \grad{\sigma} \wedge \grad{q}].
	\end{equation*}
	Since all the terms on the right hand side are admissible, one can derive that
	\begin{equation*}
		\abs{\dv{t} \norm{q^{\frac{\alpha}{2}}q^k\pd^{2k-1}\vom}_{L^2}^2} \lesssim_{A_*} B\norm{(q, \vv, \sigma)}_{\fkH^{2k}_{q}}^2.
	\end{equation*}
	The other intermediate terms in $\norm{\vom}_{H^{2k-1, k+\frac{\alpha}{2}}}$ can be handled by interpolations.
	
	The estimates for $\sigma$ are similar and even simpler. Indeed, it holds that
	\begin{equation*}
		\Dt\qty(q^k \pd^{2k}\sigma) = \comm{\Dt}{q^k \pd^{2k}}\sigma.
	\end{equation*}
	The commutator formula \eqref{comm Dt D} and Proposition \ref{interp L infty} yield the estimate:
	\begin{equation*}
		\abs{\dv{t}\norm{q^{\frac{\alpha}{2}}q^k\pd^{2k}\sigma}_{L^2}^2} \lesssim_{A_*} B\norm{(q, \vv, \sigma)}_{\fkH^{2k}_{q}}^2.
	\end{equation*}
	The intermediate terms can also be controlled via interpolations.
	
	Thus, the proof of \eqref{energy est} can be concluded from the previous arguments. \\
	\qed
	
	\subsection{Coercivity Estimates}\label{sec coercivity est}
	In this subsection, one never needs to make reference to the dynamical problems. Indeed, the energy coercivity is purely an elliptic issue.
	
	\subsubsection{Differential operators}
	
	As can be seen from \eqref{eqn Dt2 q}, the leading order term of $\Dt^2 q$ is 
	$\beta\sigma (q \laplace q)$.	Since $\sigma$ is assumed to be uniformly bounded, one can merely consider the operator $(q \laplace)$ acting on a scalar. However, such a differential operator is not self-adjoint in the Hilbert space $L^2(q^{\alpha-1}\dd{x}) =: L^2_*$, which may cause some unnecessary difficulties during the construction of solutions. Instead, one can consider the following linear differential operator
	\begin{equation}\label{def L1}
		\cL_1 u \coloneqq \beta q\laplace u + \grad{q} \vdot \grad{u}.
	\end{equation}
	Then, $\cL_1$ is self-adjoint in $L^2_*$. More concretely, it is routine to derive that
	\begin{equation}\label{L1 int formula}
		\ev{\cL_1 u, \psi}_{L^2_*} = -\beta \int_{\Omt} q^{\alpha} \grad{u}\vdot\grad{\psi}\dd{x}.
	\end{equation}
	In view of \eqref{eqn Dt2 v}, one can similarly consider a differential operator $\cL_2$ given by
	\begin{equation}\label{def L2}
		\cL_2 \vw \coloneqq \beta \grad(q\divergence\vw) + (\grad\vw)^*\grad{q}.
	\end{equation}
	It is clear that the highest order terms only involves $(\divergence\vw)$, so the  estimate for the curl is missing when using merely $\cL_2\vw$. As a compensation, one can consider a differential operator $\cL_3$:
	\begin{equation}\label{def L3}
		\cL_3 \vw \coloneqq -\beta q^{-\alpha} \Div\qty(q^{1+\alpha}\Curl\vw).
	\end{equation}
	Then it is standard to check that
	\begin{equation}\label{def L2+L3}
		\qty\big[(\cL_2 + \cL_3)\vw]_i = \beta\sum_j  q^{-\alpha}\pd_j \qty(q^{1+\alpha}\pd_j w_i) + \qty(\pd_i q \cdot \pd_j w_j - \pd_j q \cdot \pd_i w_j).
	\end{equation}
	Therefore, in the Hilbert space $L^2_{**} \coloneqq L^2(q^\alpha \dd{x})$, one has
	\begin{equation}\label{int formula L2+L3}
		\begin{split}
			&\ev{(\cL_2 + \cL_3)\vw, \vb*{\phi}}_{L^2_{**}} \\
			 &\ = - \frac{\beta^2}{1+\beta} \int_{\Omt} q^{1+\alpha}\qty\big[(1+\alpha){(\grad\vw)}\textbf{:}{(\grad{\vb*{\phi}})} + (\divergence\vw)(\divergence\vb*{\phi}) - \tr(\grad\vw\vdot\grad\vb*{\phi})] \dd{x}, 
		\end{split}
	\end{equation}
	where ${(\grad\vw)}\textbf{:}{(\grad{\vb*{\phi}})}$ is the standard pointwise inner product of tensors. In particular, $(\cL_2 + \cL_3)$ is self-adjoint in $L^2_{**}$. Moreover, one can check that $\cL_3$ itself is self-adjoint in $L^2_{**}$.
	
	\subsubsection{Bounds for the energy functionals}\label{sec bounds of good unknown}
	
	To prove the coercivity of the energy functionals, one may first note that:
	\begin{lemma}\label{lem recur}
		The following recurrence relations hold for $j \ge 2$:
		\begin{equation}\label{eqn reur}
			s_{2j} = \sigma\cL_1 s_{2(j-1)} + f_{2j} \qand \vw_{2j}=\sigma\cL_2 \vw_{2(j-1)} + \vb*{g}_{2j},
		\end{equation}
		where $f_{2j}$ and $\vb*{g}_{2j}$ are both admissible errors, i.e. they are multilinear forms of $\scrJ(q, \pd\vv, \pd\sigma)$ containing at least two factors of strictly positive order.
	\end{lemma}
	The above lemma can be proved via direct calculations, and one can refer to \cite[Lemma 5.1]{Ifrim-Tataru2024} for details. Although the equations are different, most of the calculations still work here, since $\sigma$ has order exactly $0$ and it satisfies the equation $\Dt\sigma=0$.
	
	\begin{proof}[Proof of the ``$\,\lesssim\,$'' part of \eqref{energy coercivity est}]{\hfill}	
		
	It follows from the induction formula \eqref{lem recur} and the interpolation Lemma \ref{lem err est A}.
		
	\end{proof}
	
	\subsubsection{Elliptic estimates}\label{sec elliptic est}
	
	In order to show the energy coercivity, it necessitates to establish the appropriate elliptic estimates for $\cL_1 $ and $(\cL_2 + \cL_3)$. More precisely, there holds the following  lemma:
	\begin{lemma}\label{lem ellipic est}
		Suppose that $\varepsilon_* > 0$ is a (small) constant, and $r \in C^{1+\varepsilon_*}(\overline{\Omega})$ is a non-degenerate defining function of $\Omega$, i.e. $r>0$ in $\Omega$, $r=0$ and $\abs{\grad{r}} \ge c_0 > 0$ on $\pd\Omega$. Then, it holds for each constant $\ell \ge 0$ that
		\begin{equation}\label{est L1}
			\norm*{r^{1+\ell}\pd^2u}_{L^2_*}  \lesssim_{c_0} \norm*{r^\ell\cL_1 u}_{L^2_*} + C(\norm{r}_{C^{1+\varepsilon_*}}, c_0)\norm*{r^{\ell-1}u}_{L^2_*(\{r\gtrsim\overline{\epsilon}\})},
		\end{equation}
		and
		\begin{equation}\label{est L2+L3}
			\norm*{r^{1+\ell}\pd^2\vw}_{L^2_{**}}  \lesssim_{c_0} \norm*{r^\ell(\cL_2 + \cL_3)\vw}_{L^2_{**}} + C(\norm{r}_{C^{1+\varepsilon_*}}, c_0)\norm*{r^{\ell-1}\vw}_{L^2_*(\{r\gtrsim\overline{\epsilon}\})},
		\end{equation}
		where $\overline{\epsilon} > 0$ is a small constant depending on $\norm{\grad r}_{\dot{C}^{\varepsilon_*}}$ and $c_0$. The operators $\cL_1$ and $(\cL_2 + \cL_3)$ are given by \eqref{def L1} and \eqref{def L2+L3} respectively with $q$ replaced by $r$. The weighted $L^2$ spaces are defined by $L^2_* \coloneqq L^2(\Omega; r^{\alpha-1}\dd{x})$ and $L^2_{**} \coloneqq L^2(\Omega; r^\alpha \dd{x})$.
	\end{lemma}
	\begin{remark}
		In \S\ref{sec elliptic est}, with a slight abuse of notations, we still denote by
		\begin{equation*}
			A^\sharp \coloneqq \norm{\grad{r}-\vb*{\vartheta}}_{L^\infty} \qand A_* \coloneqq \norm{r}_{C^{1+\varepsilon_*}},
		\end{equation*}
		where $\vb*{\vartheta}$ is a constant vector to be determined.
	\end{remark}
	\begin{proof}
		One may assume that $\ell = 0$, since other cases can be handled in a similar manner.	Take a cut-off function $\eta$ and let $\psi \coloneqq \alpha\eta^2 r\pd_j \pd_j u$ be the test function. One can derive from \eqref{L1 int formula} that
		\begin{equation*}
			\begin{split}
				\ev{\cL_1 u, \psi}_{L^2_*} = &\,\int r^{1+\alpha} \eta^2 \abs{\grad\pd_j u}^2 \dd{x} \\
				&- \int r^\frac{\alpha+1}{2}\eta\pd_j\pd_j u \cdot r^{\frac{\alpha-1}{2}} \eta \grad{r}\vdot\grad{u} \dd{x} \\
				&-2 \int r^{\frac{\alpha+1}{2}}\eta\pd_j\pd_j u \cdot r^{\frac{\alpha-1}{2}}\grad{\eta}\vdot\grad{u} \dd{x} \\
				&+(1+\alpha)\int r^\frac{1+\alpha}{2} \eta\grad\pd_j u \cdot r^{\frac{\alpha-1}{2}} \eta \pd_j{r} \grad{u} \dd{x} \\
				&+2\int r^{\frac{1+\alpha}{2}}\eta\grad\pd_j u \vdot r^{\frac{1+\alpha}{2}}\pd_j\eta\grad{u} \dd{x},
			\end{split}
		\end{equation*}
		which leads to
		\begin{equation*}
			\norm{\eta r \pd^2 u}_{L^2_*} \lesssim \norm{\eta\cL_1 u}_{L^2_*} + \norm{\eta\pd r \pd u}_{L^2_*} + \norm{r\pd\eta\pd u}_{L^2_*}.
		\end{equation*}
		To estimate $\norm{\pd u}_{L^2_*}$, one can take another cut-off function $\xi$ and a constant vector $\vb*{\vartheta}$. Plugging $\phi \coloneqq \alpha \xi^2\grad_{\vb*{\vartheta}}u$ into \eqref{L1 int formula} yields
		\begin{equation*}
		\begin{split}
				\ev{\cL_1 u, \phi}_{L^2_*} = &\,\frac{1}{2}\int r^{\alpha-1} \grad_{\vb*{\vartheta}}r \cdot\xi^2 \abs{\grad u}^2 \dd{x} \\
				&\,+\order{\abs{\vb*{\vartheta}}}\int r^\alpha\xi \abs{\grad\xi}\abs{\grad{u}}^2 \dd{x}.
		\end{split}
		\end{equation*}
		Recall that $r$ is assumed to be non-degenerate and $\grad{r} \in C^{\varepsilon_*}$. If $\spt (\xi)$ is of rather small size and close to the boundary, one can take a constant vector $\vb*{\vartheta}$ so that $\abs{\vb*{\vartheta}}>\frac{c_0}{2}$ and $A^\sharp \ll 1$ in $\spt(\xi)$. Thus, one arrives at
		\begin{equation*}
			\norm{\xi\pd u}_{L^2_*} \lesssim_{c_0} \norm{\xi\cL_1 u}_{L^2_*} + \norm{r\pd\xi\pd u}_{L^2_*}.
		\end{equation*}
		The modulus of continuity of $\pd r$ ensures the existence of a universal constant $\overline{\epsilon} > 0$ (depending on $A_*$) so that $A^\sharp$ is small enough as needed inside each ball with radius at most $4\overline{\epsilon}$. Therefore, one can consider a fixed point $\overline{x} \in \pd\Omega$ and a cutoff function $\xi$ supported in $B(\overline{x}, 4\overline{\epsilon})$, such that $\xi \equiv 1$ in $B(\overline{x}, 2\overline{\epsilon})$. Then, one may assume that $\abs{\pd\xi} \le \overline{\epsilon}^{-1}$. It is routine to check that
		\begin{equation*}
			\norm{r\pd\xi\pd u}_{L^2_*(\{r \ll \overline{\epsilon}\})} \ll \norm{\pd u}_{L^2_*(\{r \ll \overline{\epsilon}\}\cap\spt(\pd\xi))}.
		\end{equation*}
		Consider a finite open cover of $\{\dist(x, \pd\Omega) \le \overline{\epsilon}\}$ using balls $B(\overline{x}, 2\overline{\epsilon}) $ with $\overline{x} \in \pd\Omega$. The Besicovitch covering theorem ensures that there exists a cover so that each point belongs to at most $N_d$ doubled balls (i.e., $B(\overline{x}, 4\overline{\epsilon}) $) in the open cover, here $N_d$ is a dimensional constant. Therefore, the above arguments yield the estimate within a thin boundary layer:
		\begin{equation*}
			\norm{\pd u}_{L^2_*(\{r\ll\overline{\epsilon}\})} \lesssim_{c_0} \norm{\cL_1 u}_{L^2_*(\{r\lesssim\overline{\epsilon}\})}.
		\end{equation*}
		On the other hand, for the regions away from the boundary, one can invoke the classical elliptic estimates and interpolation inequalities to obtain that
		\begin{equation*}
			\norm{\pd u}_{L^2_*(\{r\gtrsim\overline{\epsilon}\})} \lesssim_{c_0} \norm{\cL_1 u}_{L^2_*(\Omega)} + C(A_*, c_0)\norm{r^{-1}u}_{L^2_*(\{r\gtrsim\overline{\epsilon}\})}.
		\end{equation*}
		The constant $C(A_*, c_0)$ here might be quite large, but it depends only on the property of the defining function anyway. It is unnecessary to worry about those regions away from the boundary while $0 < r \ll \overline{\epsilon}$. Indeed, one can obtain the interior local elliptic estimate by taking cutoff functions $\eta$ with $\pd\eta$ supported in the region $\{r\gtrsim\overline{\epsilon}\}$.  In summary, \eqref{est L1} is concluded from the previous arguments. The verification of \eqref{est L2+L3} follows from taking $\eta^2 r \laplace\vw$ and $\xi^2 \grad_{\vb*{\vartheta}}\vw$ as test functions in \eqref{int formula L2+L3}.
		
	\end{proof} 
	With essentially the same arguments, one can derive the following variant:
	\begin{lemma}\label{lem ellip est L1 L2 L3}
		Suppose that the assumptions in Lemma \ref{lem ellipic est} hold. Given a constant $\lambda \ge 0$, let
		\begin{equation}\label{est ellip L1 lambda}
			\cL_{1,\lambda} \coloneqq \cL_1 + \lambda\beta{\grad r}\vdot\grad(\cdot) \qand \cL_{2,\lambda} \coloneqq \cL_2 + \lambda\beta\qty\big[\div(\cdot)]\grad{r}.
		\end{equation}
		Then, $\cL_{1,\lambda}$ and $(\cL_{2,\lambda}+\cL_3)$ also satisfy the estimates \eqref{est L1} and \eqref{est L2+L3}, respectively.
	\end{lemma}
	
	For the higher order estimates, one can first assume that the function $u$ is supported in a tiny compact region close to the boundary. Thus, the lower order term on the right hand side of \eqref{est L1} is unnecessary. In other words, one can now apply the simpler form of \eqref{est L1}:
	\begin{equation}\label{est L1 simple}
		\norm*{r^{1+\ell}\pd^2 u}_{L^2_*} + \norm*{r^\ell \pd u}_{L^2_*} \lesssim_{c_0} \norm*{r^\ell \cL_{1, \lambda} u}_{L^2_*}.
	\end{equation}
	By taking adapted coordinates, one may assume that $A^\sharp \ll 1$ and
	\begin{equation}\label{assumption p.o.u.}
		\grad{r} = \vb{e}_{(n)} + \order*{A^\sharp} \qin \spt(u).
	\end{equation}
	Notice further that
	\begin{equation*}
		\comm{\cL_1}{\pd_j}u = -\beta\pd_j r\laplace u - \grad(\pd_j r) \vdot\grad{u}.
	\end{equation*}
	Then, it is clear that
	\begin{equation}\label{comm L1 pd bar}
		\comm{\cL_1}{\overline{\pd}}u = \order*{A^\sharp}\pd^2 u - \overline{\pd}(\grad{r})\vdot\grad{u},
	\end{equation}
	and
	\begin{equation*}
		\begin{split}
			\pd_n\cL_1 u &= \cL_1\pd_n u +\beta\pd_n{r}\cdot\pd_n\pd_n{u} + \beta\pd_n{r}\cdot\overline{\pd}^2 u + \pd_n(\grad{r})\vdot\grad{u}\\
			&= \cL_{1, 1} (\pd_n{u}) + \beta\pd_n r\cdot\overline{\pd}^2 u + \pd_n(\grad{r})\vdot\grad{u}.
		\end{split}
	\end{equation*}
	where $\overline{\pd}$ means $\pd_j$ with $j \neq n$.
	Similar arguments lead to
	\begin{equation*}
		\cL_{1, 2}(\pd_n\pd_n u) = \pd_n \pd_n \cL_1 u -2\beta\pd_n r\cdot\pd_n\overline{\pd}^2 u - \pd_n\pd_n(\grad{r})\vdot\grad{u}  + \order{\abs*{\pd^2r\cdot\pd^2 u}}.
	\end{equation*}
	Inductively, the following relation holds for $\lambda \in \mathbb{N}$:
	\begin{equation}\label{comm L1 pd n pd n}
		\cL_{1, \lambda}(\pd_n^\lambda u) = \pd_n^\lambda\cL_1 u -\lambda\beta\pd_n r \cdot\pd_n^{\lambda-1}\overline{\pd}^2u - \pd_n^\lambda(\grad{r})\vdot\grad{u} + \calM\qty\big(\scrJ(r), \scrJ(\pd^2 r), \scrJ(\pd^2 u)),
	\end{equation}
	here $\calM$ is a multilinear form, which is linear in $\scrJ(\pd^2 u)$. 
	\begin{remark}
		For the sake of convenience, one may note that for the particular case $u = r$, there hold
		\begin{equation}\label{L1 pdbar r}
			\comm{\cL_1}{\overline{\pd}}r = - \pd_n r\cdot\pd_n\overline{\pd}r + \order*{A^\sharp}\pd^2 r,
		\end{equation}
		and
		\begin{equation}\label{L1 pdn r}
			\cL_{1, (\lambda+\alpha)} \pd_n^\lambda r = \pd_n^\lambda \cL_1 r - \lambda\beta\pd_n r \cdot \pd_n^{\lambda-1}\overline{\pd}^2 r + \calM',
		\end{equation}
		where $\calM'$ is a multilinear form at least bilinear in $\scrJ(\pd^2 r)$.
	\end{remark}
	
	The combination of \eqref{comm L1 pd n pd n} and \eqref{est L1 simple} implies that
	\begin{equation*}
		\begin{split}
			\norm*{r^2\pd_n^4 u}_{L^2_*} &\lesssim_{c_0} \norm*{r \cL_{1, 2} \pd_n^2 u}_{L^2_*} \\
			&\lesssim_{c_0} \norm*{r\pd_n^2 \cL_1 u}_{L^2_*} + \order{\norm{\pd r}_{L^\infty}}\norm*{r\overline{\pd}^2\pd_n u}_{L^2_*} + R \\
			&\lesssim_{A_*, c_0} \norm*{(\cL_1)^2 u}_{L^2_{*}} + \norm*{\cL_{1, 1}\pd_n u}_{L^2_*} + R \\
			&\lesssim_{A_*, c_0} \norm*{(\cL_1)^2 u}_{L^2_*} + \norm*{\overline{\pd}^2 u}_{L^2_*} + R,
		\end{split}
	\end{equation*}
	where the remainder term $R$ is given by
	\begin{equation*}
		R \coloneqq \norm*{r\pd^2 r\cdot \pd^2 u}_{L^2_*} + \norm*{r \pd^3 r\cdot \pd u}_{L^2_*}.
	\end{equation*}
	Moreover, Lemma \ref{lem inclusion} and Proposition \ref{prop interp} lead to
	\begin{equation*}
		\norm*{\overline{\pd}^2 u}_{L^2_*} \lesssim \norm*{r^2\pd^2\overline{\pd}^2 u}_{L^2_*} + \norm*{r^2\overline{\pd}^2u}_{L^2_*}.
	\end{equation*}
	It follows from \eqref{est L1 simple}-\eqref{comm L1 pd bar} that
	\begin{equation*}
		\norm*{r^2 \pd^2 \overline{\pd}^2 u}_{L^2_*} \lesssim \norm*{\cL_1^2 u}_{L^2_*} + \order*{A^\sharp}\norm*{\overline{\pd}^2 u}_{L^2_*} + R.
	\end{equation*}
	Combining the above estimates and Proposition \ref{prop interp} together, one can derive from the smallness assumption $A^\sharp \ll 1$ that
	\begin{equation}\label{est L1 square u}
		\norm*{r^2\pd^4 u}_{L^2_*} \lesssim_{A_*, c_0} \norm*{\cL_1^2 u}_{L^2_*} + \norm*{u}_{L^2_*} + \order{\norm*{r\pd^2r\cdot\pd^2u}_{L^2_*} + \norm*{r\pd^3r\cdot\pd u}_{L^2_*}}.
	\end{equation}
	Note that there is a partition of unity dividing the domain into the union of the boundary layer and the interior, so that the boundary layer consists of multiple small pieces with uniform sizes, whose double enlargements satisfy the generic finite overlapping property. Thus, the estimate \eqref{est L1 square u} actually holds for all $u$, not only for those supported in a small set near the boundary. The cost for invoking the partition of unities would be the largeness of the coefficient of $\norm{u}_{L^2_*}$, but fortunately, it depends only on $\norm{\grad{r}}_{\dot{C}^{\varepsilon_*}}$ and $ c_0$. One can also repeat the previous process to obtain the higher-order versions.
	
	More precisely, there holds the following lemma:
	\begin{lemma}\label{lem L1k est}
		Suppose that $r \in C^{1+\varepsilon_*}(\overline{\Omega})$ is a non-degenerate defining function of $\Omega$, i.e. $r>0$ in $\Omega$, $r=0$ and $\abs{\grad{r}} \ge c_0 > 0$ on $\pd\Omega$. Assume further that $\ell > 0$ is an integer and
		\begin{equation*}
			r \in H^{2\ell, \ell+\frac{\alpha-1}{2}}_r (\Omega).
		\end{equation*} 
		Then, for a large constant $C$ determined by $c_0$ and $\norm{\grad r}_{\dot{C}^{\varepsilon_*}}$, it holds that
		\begin{equation}\label{est L1 high}
			\begin{split}
				\norm{u}_{H^{2\ell, \ell+\frac{\alpha-1}{2}}_r (\Omega)} \lesssim_{A_*, c_0}  &\, \norm*{(\cL_1[r])^\ell u}_{H^{0, \frac{\alpha-1}{2}}} + C\norm{u}_{H^{0, \frac{\alpha-1}{2}}} + \\ 
				&\quad + \norm*{r^\ell \pd^{2\ell-1}r\cdot \pd u}_{H^{0, \frac{\alpha-1}{2}}} + \norm{\calM}_{H^{0, \frac{\alpha-1}{2}}}, 
			\end{split}
		\end{equation}
		where the term $\calM$ consists of multilinear forms of $(\scrJ(r), \scrJ(\pd^2 u))$, which are linear in $\scrJ(\pd^2 u)$ and have at least one $\pd^{\ge 2}r$ factors.
	\end{lemma}
	
	Concerning the estimates for vector fields, one can first compute as the above to obtain that
	\begin{equation}
		\cL_{2,\lambda}(\pd_n^\lambda \vw) = \pd_n^\lambda\cL_2 \vw -\lambda\beta\pd_n r \cdot\pd_n^{\lambda-1}\overline{\pd}(\div\vw) + \calM\qty\big(\scrJ(r), \scrJ(\pd^2 r), \scrJ(\pd \vw)).
	\end{equation}
	Thus, similar arguments applied to $(\cL_2 + \cL_3)$ imply the following lemma, which is akin to the previous one:
	\begin{lemma}\label{lem l2k est}
		With the same assumptions and the notations as in Lemma \ref{lem L1k est}, it holds that
		\begin{equation}\label{est L2 high}
			\norm{\vw}_{H^{2k, k+\frac{\alpha}{2}}} \lesssim_{A_*, c_0} \norm*{(\cL_2[r])^k \vw}_{H^{0, \frac{\alpha}{2}}} + \norm{\Curl \vw}_{H^{2k-1, k+\frac{\alpha}{2}}} + C\norm{\vw}_{H^{0, \frac{\alpha}{2}}} + \norm{\calM}_{H^{0, \frac{\alpha}{2}}},
		\end{equation}
		where $\calM$ consists of multilinear forms of $(\scrJ(r), \scrJ(\pd\vw))$, which are linear in $\scrJ(\pd\vw)$ and have at least one $\pd^{\ge 2}r$ factors.
	\end{lemma} 
	
	\subsubsection{Coercivity of the energy functional}\label{subsubsec coercivity}
	
	Next, one can use the recurrence relations given in Lemma \ref{lem recur} and the elliptic estimates established above to show the energy coercivity.
	Now, let us focus on the estimates for $q$. Since \eqref{eqn reur} only holds for $j \ge 2$, to close the induction process, one still needs to control $q\equiv s_0$ using $s_2$. Indeed, it follows from \eqref{eqn Dt2 q} and \eqref{def good unknown} that
	\begin{equation*}
		s_2 = \sigma\beta q\laplace q + \frac{1}{2}\sigma\grad{q}\vdot\grad{q} + \beta q\grad{q}\vdot\grad{\sigma} + \order{q\cdot\pd\vv\cdot\pd\vv}.
	\end{equation*}
	Thus, it is straightforward to see that
	\begin{equation*}
		\pd_j s_2 \approx \sigma \beta \pd_j q \laplace q + \sigma\beta q\laplace \pd_j q + \sigma\grad{q}\vdot\grad(\pd_j q) + \order{\pd q\cdot\pd q\cdot\pd\sigma} + \order{q\pd q\cdot \pd^2\sigma},
	\end{equation*}
	here $F\approx G$ means that $(F-G)$ is an admissible error. After applying an appropriate partition of unity, one may assume now that $\pd^2 q$, $\vv$, and $\pd\sigma$ all vanish outside a small tiny region near the boundary. Suppose further that in the prescribed compact set, it holds that
	\begin{equation*}
		\abs{\grad{q}-\vb{e}_{(n)}} \le A^\sharp \ll 1.
	\end{equation*}
	Then, it is routine to compute that
	\begin{equation*}
		\overline{\pd}s_2 \approx \sigma \cL_1 \overline{\pd}s_0 + \order*{A^\sharp}\sigma\pd^2s_0 + \order{\pd q\cdot\pd q\cdot\pd\sigma} + \order{q\pd q\cdot\pd^2\sigma},
	\end{equation*}
	and
	\begin{equation*}
		\pd_n s_2 \approx \sigma\cL_{1, 1}(\pd_n s_0) + \sigma\beta \pd_n q \overline{\pd}^2 s_0 + \order{\pd q\cdot\pd q\cdot\pd\sigma} + \order{q\pd q\cdot\pd^2\sigma}.
	\end{equation*}
	Inductively, one can derive that
	\begin{equation}\label{rel s_2 s_0}
		\pd_n^\lambda s_2 \approx \sigma\cL_{1, \lambda}(\pd_n^\lambda s_0) + \lambda\beta\sigma\pd_n q\cdot\pd_n^{\lambda-1}\overline{\pd}^2 s_0 + \order{\pd q\cdot\pd q\cdot\pd^{\lambda}\sigma} + \order{q\pd q\cdot\pd^{1+\lambda}\sigma}.
	\end{equation}
	On the other hand, it follows from Lemmas \ref{lem err est A}, \ref{lem recur}, and \ref{lem L1k est}  that
	\begin{equation}
		\begin{split}
			\norm{s_2}_{H^{2(k-1), (k-1)+\frac{\alpha-1}{2}}} \lesssim_{A_*, c_0}&\, \norm{s_{2k}}_{H^{0, \frac{\alpha-1}{2}}} + C\norm{s_2}_{H^{0, \frac{\alpha-1}{2}}} + \norm{\sigma}_{H^{2k, k+\frac{\alpha}{2}}} + \\
			&\quad + \order*{A^\sharp}\norm{(q, \vv, \sigma)}_{\fkH^{2k}_{q}}. 
		\end{split}
	\end{equation}
	Thus, it can be derived from \eqref{rel s_2 s_0}, \eqref{L1 pdbar r}-\eqref{L1 pdn r}, \eqref{est L1 simple}-\eqref{comm L1 pd bar}, and the interpolations that
	\begin{equation}\label{est q 2k norm}
			\norm{q}_{H^{2k, k+\frac{\alpha-1}{2}}} \lesssim_{A_*, c_0} \norm{s_{2k}}_{H^{0, \frac{\alpha-1}{2}}} + C\norm{q}_{H^{0, \frac{\alpha-1}{2}}} + \norm{\sigma}_{H^{2k, k+\frac{\alpha}{2}}} + \order*{A^\sharp}\norm{(q, \vv, \sigma)}_{\fkH^{2k}_{q}}.
	\end{equation}
	
	For the estimates of $\vv$, one can first note that
	\begin{equation*}
		\vw_2 = \sigma\cL_2\vv.
	\end{equation*}
	Thus, Lemmas \ref{lem err est A}, \ref{lem recur}, and \ref{lem l2k est} lead to
	\begin{equation}\label{est v 2k norm}
		\begin{split}
			\norm{\vv}_{H^{2k, k+\frac{\alpha}{2}}} \lesssim_{A_*, c_0} &\, \norm{\vw_{2k}}_{H^{0, \frac{\alpha}{2}}} + C\norm{\vv}_{H^{0, \frac{\alpha}{2}}} + \norm{\vom}_{H^{2k-1, k+\frac{\alpha}{2}}} + \norm{\sigma}_{H^{2k, k+\frac{\alpha}{2}}} + \\
			 &\quad + \order*{A^\sharp}\norm{(q, \vv, \sigma)}_{\fkH^{2k}_{q}}.
		\end{split}
	\end{equation}
	\begin{proof}[Proof of the ``$\,\gtrsim$'' part of \eqref{energy coercivity est}]\hfill
		
		By taking an adapted partition of unity, one can assume that $A^\sharp \ll 1$. Then \eqref{energy coercivity est} follow from \eqref{est q 2k norm}, \eqref{est v 2k norm}, and the smallness of $A^\sharp$. For the estimates in a region away from the boundary, one can apply the standard elliptic estimates and interpolation inequalities. Note that, by invoking an adapted partition of unity, the coefficients of $\norm{q}_{H^{0, \frac{\alpha-1}{2}}}$ and $\norm{\vv}_{H^{0, \frac{\alpha}{2}}}$ may become quite large, but they are still generic constants depending only on the space dimension, $\alpha, k$, $c_0 $ and $A_*$.
		
	\end{proof}
	
	\subsection{Other Equivalent Energy Functionals}\label{sec other energies}
	
	The energy functionals defined in \S\ref{sec energies and control para} is designed for the simplicity of a priori estimates. However, those energies may cause unnecessary technical difficulties during the construction of solutions. Thus, we introduce some equivalent higher-order energies here, which will be used in the next section. 
	
	
	In order to make the expressions of the higher order norms of $\vv$ and $\sigma$ more explicit, one may consider the operators (here $r$ is a non-degenerate defining function of $\Omega$):
	\begin{equation}\label{def L0 L4}
		(\cL_0[r])\vw \coloneqq \grad^2 r\vdot\vw \qc \cL_4 \coloneqq \cL_2 + \cL_0,
	\end{equation}
	and
	\begin{equation}
		(\cL_5[r]) u \coloneqq r\laplace u + (1+\alpha)\grad{r}\vdot\grad{u}.
	\end{equation}
	Then, it is clear that $\cL_0[r]$, $\cL_4[r]$, and $\cL_5[r]$ are all self-adjoint operators in $L^2(r^\alpha \dd{x})$. 
	
	It follows from the same arguments that $\cL_5[r]$ also satisfies the elliptic estimate akin to \eqref{est L1 high}:
	\begin{equation}\label{est L5}
		\norm{u}_{H^{2k, k+\frac{\alpha}{2}}_r (\Omega)} \lesssim_{A_*, c_0}   \norm*{(\cL_5[r])^k u}_{H^{0, \frac{\alpha}{2}}} + C\norm{u}_{H^{0, \frac{\alpha}{2}}} + \norm{\calM}_{H^{0, \frac{\alpha}{2}}},
	\end{equation}
	where $\calM$ is a multilinear form of $\scrJ(r, \pd u)$ linear in $\scrJ(\pd u)$, and containing at least one factor of $\pd^{\ge 2} r$. For $k = 1$ case, the $\calM$-term will not appear.
	
	It is clear that the operator $\cL_4$ satisfies the commuting relation:
	\begin{equation*}
		\cL_4 \cL_3 = 0 = \cL_3 \cL_4,
	\end{equation*}
	which implies that
	\begin{equation}\label{comm L2 L3}
		\cL_2 \cL_3 = - \cL_0 \cL_3 \qand \cL_3 \cL_2 = - \cL_3 \cL_0.
	\end{equation}
	Moreover, it follows from \eqref{def L3} and \eqref{def L0 L4} that
	\begin{equation}\label{L3 L0}
		-(\cL_3 \cL_0 \vw)_i = \sum_{j, k} \beta r^{-\alpha}\pd_j\qty\big(r^{1+\alpha}\pd_j\pd_k r\cdot \pd_i w_k - r^{1+\alpha}\pd_i\pd_k r\cdot\pd_j w_k).
	\end{equation}
	In particular, both $(\cL_2[q]\cL_3[q]) \vv$ and $(\cL_3[q]\cL_2[q]) \vv$ are admissible errors in the sense of Definition \ref{def adm err}. Similar to \eqref{est ellip L1 lambda}, one can define
	\begin{equation}
		\cL_{3, \lambda}[r] \vw \coloneqq \cL_3[r] \vw + \lambda\beta(\grad\vw)\vdot\grad{r} - \lambda\beta(\div\vw)\grad{r},
	\end{equation}
	here $\lambda$ is a positive parameter. Then, following the assumption \eqref{assumption p.o.u.}, it can be deduced from the similar arguments that
	\begin{equation}
		\cL_{3, \lambda}(\pd_n^\lambda \vw) = \pd_n^\lambda \cL_3 \vw + \lambda\beta\qty\big[(\grad\vw)-(\div\vw)]\grad{r} + \calM\qty\big[\scrJ(r), \scrJ(\pd^2 r), \scrJ(\pd\vw)],
	\end{equation} 
	where $\calM$ is a multilinear form. Moreover, it follows from the same proof that \eqref{est L2+L3} still holds when $(\cL_2 + \cL_3)$ is replaced by $(\cL_{2,\lambda} + \cL_{3, \lambda})$. Thus, under the same assumptions and notations, \eqref{est L2 high} can be replaced by
	\begin{equation}\label{est L2 and L3}
		\norm{\vw}_{H^{2k, k+\frac{\alpha}{2}}} \lesssim_{A_*, c_0} \norm*{(\cL_2[r])^k \vw}_{H^{0, \frac{\alpha}{2}}} + \norm*{(\cL_3 [r])^k \vw}_{H^{0, \frac{\alpha}{2}}} + C\norm{\vw}_{H^{0, \frac{\alpha}{2}}} + \norm{\calM}_{H^{0, \frac{\alpha}{2}}},
	\end{equation}
	where $\calM$ consists of multilinear forms of $(\scrJ(r), \scrJ(\pd\vw))$, which are linear in $\scrJ(\pd\vw)$ and have at least one $\pd^{\ge 2}r$ factors.
	
	In summary, one can	define the modified good unknowns
	\begin{equation}\label{def mod good unknown}
		\fks_{2k} \coloneqq \sigma^{-k} s_{2k}\qc \fkw_{2k} \coloneqq \sigma^{-k}\vw_{2k} \qc \vom_{2k} \coloneqq (\cL_3[q])^k \vv, \qand \sigma_{2k} \coloneqq (\cL_5[q])^k \sigma,
	\end{equation}
	and the variant energy functional
	\begin{equation}\label{def fkE 2k}
		\begin{split}
			\fkE_{2k} &\coloneqq \int_{\Omt} q^{\alpha-1} \qty[\abs{\fks_{2k}}^2 + \beta q \sigma^{-1}\abs{\fkw_{2k}}^2 + \beta q \sigma^{-1}\abs{\vom_{2k}}^2 + q\abs{\sigma_{2k}}^2] \dd{x} \\
		 	&\qquad + \int_{\Omt} q^{1+\alpha} + \frac{1+\beta}{2}q^\alpha\sigma^{-1}\abs{\vv}^2 + q^\alpha \abs{\sigma}^2 \dd{x} \\
		 	&\eqqcolon \fkE_{2k, \text{high}} + \fkE_{2k, \text{low}}.
		\end{split}
	\end{equation}
	Then, the modified good unknowns $\fks_{2j}$ and $\fkw_{2j}$ satisfy the following recurrence relations:
	\begin{equation}
		\fks_{2j} = \cL_1 \fks_{2(j-1)} + \sigma^{-k}f_{2j}, \ (j\ge 2) \qand \fkw_{2j}=\cL_2 \fkw_{2(j-1)} + \sigma^{-k}\vb*{g}_{2j}, \ (j \ge 1),
	\end{equation}
	here $f_{2j}$ and $\vb*{g}_{2j}$ are still admissible errors in the sense of Definition \ref{def adm err}. 
	Thus, it follows from the combination of arguments in \S\ref{sec coercivity est} and \eqref{est L5}-\eqref{est L2 and L3} that $\fkE_{2k} $ also satisfies \eqref{energy coercivity est}.
	
	We further remark that, the trivial identity
	\begin{equation*}
		\int_{\Omt} q^{\alpha}\sigma^{-1} \qty(\abs{\fkw_{2k}}^2 + \abs{\vom_{2k}}^2) \dd{x} = \int_{\Omt} q^\alpha \qty(\abs{\sigma^{-\frac{1}{2}}\fkw_{2k}}^2 + \abs{\sigma^{-\frac{1}{2}}\vom_{2k}}^2) \dd{x}
	\end{equation*}
	hints one to consider the $H^{0, \alb}$-norms of $(\cL_2)^k(\sigma^{-\frac{1}{2}}\vv)$ and $(\cL_3)^k(\sigma^{-\frac{1}{2}}\vv) $. Indeed, when omitting the admissible errors, one has the relations
	\begin{equation}
		\sigma^{-\frac{1}{2}}\fkw_{2k} \approx (\cL_2[q])^k\qty(\sigma^{-\frac{1}{2}}\vv) \qand \sigma^{-\frac{1}{2}}\vom_{2k} \approx (\cL_3[q])^k\qty(\sigma^{-\frac{1}{2}}\vv).
	\end{equation}
	In other words, the new energy functional also characterizes the following norm:
	\begin{equation}\label{equiv fkE 2k new norm}
		\fkE_{2k} \simeq_{A_*, c_0} \norm{\qty(q, \sigma^{-\frac{1}{2}}\vv, \sigma)}_{\fkH^{2k}_{q}}^2.
	\end{equation}
	In particular, there holds the equivalence:
	\begin{equation}\label{equiv two norms}
		\norm{\qty(q, \sigma^{-\frac{1}{2}}\vv, \sigma)}_{\fkH^{2k}_{q}} \simeq_{A_*, c_0} \norm{\qty(q, \vv, \sigma)}_{\fkH^{2k}_{q}}.
	\end{equation}
	
	On the other hand, note that
	\begin{equation}\label{eqn evo modified good unknowns}
		\begin{cases*}
			\Dt \fks_{2k} + \fkw_{2k} \vdot \grad{q} + \beta q (\div\fkw_{2k}) = \sigma^{-k}\fkf_{2k}, \\
			\Dt \fkw_{2k} + \sigma\grad{\fks_{2k}} = \sigma^{-k} \fkg_{2k},
		\end{cases*}
	\end{equation}
	where $\fkf_{2k}$ and $\fkg_{2k}$ are both admissible errors. Furthermore, one can calculate that
	\begin{equation*}
		\begin{split}
			\Dt (\cL_3[q]\vv) = &-(1+\beta) (\Dt\vom)\vdot\grad{q} - (1+\beta)\vom\vdot(\Dt\grad{q}) \\
			 &+ \beta(\Dt q)\qty\big[\laplace\vv-\grad(\div\vv)] + \beta q\Dt\qty\big[\laplace\vv-\grad(\div\vv)] \\
			 \approx &\,(1+\beta)(\grad{\sigma}\wedge\grad{q})\vdot\grad{q} \\
			 \approx &\,(1+\beta)\qty[\abs{\grad{q}}^2\grad{\sigma}-(\grad{\sigma}\vdot\grad{q})\grad{q}],
		\end{split}
	\end{equation*}
	here we still use the notation $F \approx G$ iff $(F-G)$ is an admissible error. Inductions and the commutator formula \eqref{comm Dt D} yield that
	\begin{equation}\label{eqn Dt vom 2k}
		\Dt \vom_{2k} \approx \order{q^{k-1}(\pd q)^2 \pd^{2k-1}\sigma} + \order{(\pd q)^{k+1} \pd^k \sigma}.
	\end{equation}
	Moreover, it is direct to compute that
	\begin{equation*}
		\Dt (\cL_5[q]\sigma) \approx 0,
	\end{equation*}
	which implies that
	\begin{equation}\label{eqn Dt sigma 2k}
		\Dt \sigma_{2k} \approx 0.
	\end{equation}
	Thus, the new energy functional $\fkE_{2k}$ also satisfies the propagation estimate \eqref{energy est}.
	
	\section{Existence of High-regularity Solutions}\label{sec existence}
	
	In this section, we will construct solutions to the compressible Euler system for sufficiently regular initial data, specifically in the state space $\bbH^{2k}$ with $k \gg 1$. Low-regularity solutions will be constructed in the next section as limits of these highly regular ones.
	
	There are two classical methods to solve an ordinary differential equation: Picard's iterations and Euler's polygonal integration methods. Picard's method yields a unique solution but requires more regularity in the source terms. Euler's method can handle more general source terms, but the uniqueness of solutions needs further investigation. Since one is dealing with a quasilinear free boundary problem, it is often necessary to convert the problem into one defined on a fixed domain before applying Picard's iterations. This can be achieved by using Lagrangian coordinates, which transform the free boundary problem into a fixed-boundary but highly nonlinear one. The advantage is that it eliminates the issue of moving domains, but the downside is the sharply increased difficulty in handling the nonlinear terms. As in \cite{Ifrim-Tataru2024}, one can utilize Euler's iterations and the Arzel\`{a}-Ascoli theorem to obtain solutions to the compressible Euler system with highly regular initial data.
	
	One major challenge is controlling the energy increments at each iteration step. For example, if the time step size is $\varepsilon$, only $\order{\varepsilon}$-bounded energy increments are acceptable. Otherwise, the scheme would blow up after several iterations as $\varepsilon \to 0$. Another challenge is overcoming the derivative loss during the Euler iterations. Regarding the quasilinear nature, a direct application would not preserve the regularity of solutions. Following the ideas in \cite{Ifrim-Tataru2024}, one can overcome these obstacles by applying a two-step iteration scheme:
	\begin{itemize}
		\item Regularization;
		\item Euler's polygonal methods in an appropriate sense.
	\end{itemize}
	
	\subsection{Euler's Polygonal Methods: One-step Iterates}
	Given a large integer $k \gg 1$, an initial state $\qty{q_0, \vv_0, \sigma_0} \in \bbH^{2k}$, and a small time step-size $0 < \varepsilon \ll 1$, one can construct a sequence of states $\qty{q(\ell\varepsilon), \vv(\ell\varepsilon), \sigma(\ell\varepsilon) }_{\ell \ge 0}$ satisfying
	\begin{itemize}
		\item \textbf{Initial values:}
			\begin{equation*}
				\qty\big{q(0), \vv(0), \sigma(0)} = \qty\big{q_0, \vv_0, \sigma_0}.
			\end{equation*}
		\item \textbf{Energy increments:}
		
		For the energy functional given by \eqref{def fkE 2k}, it holds that
		\begin{equation*}
			\fkE_{2k}\qty\big{q[(\ell+1)\varepsilon], \vv[(\ell+1)\varepsilon], \sigma[(\ell+1)\varepsilon]} \le (1+C\varepsilon) \fkE_{2k}\qty\big{q(\ell\varepsilon), \vv(\ell\varepsilon), \sigma(\ell\varepsilon)},
		\end{equation*}
		here $C$ is a generic constant depending only on the initial size $\norm{(q_0, \vv_0, \sigma_0)}_{\fkH^{2k}_{q_0}}$.
		\item \textbf{Approximate solutions:}
		
		For domains given by
		\begin{equation*}
			\Omega_{(\ell+1)\varepsilon} \coloneqq \qty\big{q[(\ell+1)\varepsilon] > 0} \qand \Omega_{\ell\varepsilon} \coloneqq \qty\big{q[\ell\varepsilon] > 0},
		\end{equation*}
		there holds
		\begin{equation*}
			\Omega_{(\ell+1)\varepsilon} = \qty\Big(\textup{Id} + \varepsilon\vv[\ell\varepsilon])\Omega_{\ell\varepsilon}.
		\end{equation*}
		Furthermore, the following relations hold in $\Omega_{(\ell+1)\varepsilon} \cap \Omega_{\ell\varepsilon}$:
		\begin{equation*}
			\begin{cases*}
				q[(\ell+1)\varepsilon]-q[\ell\varepsilon] + \varepsilon \qty\big{(\vv[\ell\varepsilon]\vdot\grad){q[\ell\varepsilon]}+ \beta q[\ell\varepsilon]\div{\vv[\ell\varepsilon]}} = \Order{C^1}{\varepsilon^{1+}}, \\
				\vv[(\ell+1)\varepsilon] - \vv[\ell\varepsilon] + \varepsilon\qty\big{(\vv[\ell\varepsilon]\vdot\grad)\vv[\ell\varepsilon] + \sigma[\ell\varepsilon]\grad{q[\ell\varepsilon]}} = \Order{C^1}{\varepsilon^{1+}}, \\
				\sigma[(\ell+1)\varepsilon] - \sigma[\ell\varepsilon] + \varepsilon (\vv[\ell\varepsilon]\vdot\grad)\sigma[\ell\varepsilon] = \Order{C^1}{\varepsilon^{1+}}.
			\end{cases*}
		\end{equation*}
	\end{itemize}
	
	To obtain a sequence satisfies the above three conditions, it suffices to establish the one-step iteration:
	
	\begin{theorem}\label{thm one step iterate}
		Assume that $k \gg 1$ is a sufficiently large integer and $ M > 1$ is a constant. Let $0 < \varepsilon \ll 1$ be a time step-length and $\qty{q_0, \vv_0, \sigma_0} \in \bbH^{2k}$ an initial state satisfying the energy bound
		\begin{equation}\label{energy bound 0 state}
			\fkE_{2k}(q_0, \vv_0, \sigma_0) \le M
		\end{equation}
		and the non-degeneracy condition
		\begin{equation}
			\abs{\grad{q_0}} \ge c_0 > 0 \qq{on} \Gamma_0.
		\end{equation}
		Then, there exists a one-step iterate $(q_1, \vv_1, \sigma_1)$ satisfying the following properties:
		\begin{enumerate}[label=(\roman*)]
			\item \textbf{Energy increment:}
			\begin{equation}\label{energy increments one step}
				\fkE_{2k}(q_1, \vv_1, \sigma_1) \le \qty\big[1+C(M)\varepsilon] \fkE_{2k}(q_0, \vv_0, \sigma_0).
			\end{equation}
			\item \textbf{Approximate solution:}
			
			For $\Omega_0$ and $\Omega_1$ defined respectively by
			\begin{equation*}
				\Omega_0 \coloneqq \qty{q_0 > 0} \qand \Omega_1 \coloneqq \qty{q_1 > 0},
			\end{equation*}
			there hold
			\begin{equation}\label{approx sol one step}
				\begin{cases*}
					q_1 - q_0 + \varepsilon\qty\big[(\vv_0\vdot\grad)q_0 + \beta q_0\div\vv_0] = \Order{C^1}{\varepsilon^2} &in\, $\Omega_1 \cap \Omega_0$, \\
					\vv_1 - \vv_0 + \varepsilon\qty\big[(\vv_0\vdot\grad)\vv_0 + \sigma_0 \grad q_0] = \Order{C^1}{\varepsilon^2} &in\, $\Omega_1 \cap \Omega_0$, \\
					\sigma_1 - \sigma_0 + \epsilon\qty\big[(\vv_0\vdot\grad)\sigma_0] = \Order{C^1}{\varepsilon^2} &in\, $\Omega_1 \cap \Omega_0$, \\
					\Omega_1 = \qty\big(\textup{Id}+\varepsilon\vv_0)\Omega_0 + \Order{C^1}{\varepsilon^2}.
				\end{cases*}
			\end{equation}
		\end{enumerate}
	\end{theorem}
	
	\begin{remark}
		For the simplicity of notations, in \S\ref{sec existence}, we will denote by $C(M)$ a constant depending on both $M$ and $c_0$. Indeed, this will not hurt the existence results, because the property that $\abs{\grad{q}} \gtrsim c_0$ on $\Gmt$ can be propagated, at least for a short time.
	\end{remark}
		
	\subsection{Convergence of the Iteration Scheme}
	In stead of proving Theorem \ref{thm one step iterate}, we first observe how to obtain a regular solution from the iterations. For a given small time scale $\varepsilon>0$, Theorem \ref{thm one step iterate} yields a discrete approximate solution $\qty{q^{(\varepsilon)}, \vv^{(\varepsilon)}, \sigma^{(\varepsilon)}}$, which is defined at discrete times $t = \ell\varepsilon, \ell \in \mathbb{N}$. The energy increment bound \eqref{energy increments one step} implies that
	\begin{equation*}
		\fkE_{2k}\qty[q^{(\varepsilon)}(\ell\varepsilon), \vv^{(\varepsilon)}(\ell\varepsilon), \sigma^{(\varepsilon)}(\ell\varepsilon)] \le \qty\big[1+C(M)\varepsilon]^{\ell} \fkE_{2k}(q_0, \vv_0, \sigma_0) \le M \exp{C(M)\varepsilon\ell}.
	\end{equation*}
	Thus, these approximate solutions are defined up to a time $T = T(M)$, uniformly in $\varepsilon \ll 1$. Moreover, there hold the uniform bounds
	\begin{equation}\label{energy bound approx sol}
		\sup_{0 \le t \le T} \norm{\qty(q^{(\varepsilon)}, \vv^{(\varepsilon)}, \sigma^{(\varepsilon)})}_{\fkH^{2k}_{q^{(\varepsilon)}}} \lesssim_M 1 \qc \forall\, 0 < \varepsilon \ll 1.
	\end{equation}
	On the other hand, as it is required that $k \gg 1$, it follows from the Sobolev embeddings that
	\begin{equation*}
		\sup_{0 \le t \le T} \norm\big{q^{(\varepsilon)}}_{C^3} + \norm\big{\vv^{(\varepsilon)}}_{C^3} + \norm\big{\sigma^{(\varepsilon)}}_{C^3} \lesssim_{M} 1 \qc \forall 0< \varepsilon\ll 1.
	\end{equation*}
	Therefore, iterating the relation \eqref{approx sol one step} yields the difference bound
	\begin{equation*}
		\abs{q^{(\varepsilon)}(t, x) - q^{(\varepsilon)}(s, y)} + \abs{\grad{q^{(\varepsilon)}}(t, x) - \grad{q^{(\varepsilon)}}(s, y)} \lesssim_M \abs{t-s} + \abs{x-y},
	\end{equation*}
	for all $t,s \in \varepsilon\mathbb{N}\cap[0, T] $.
	The same bounds also hold for $\vv^{(\varepsilon)}$ and $\sigma^{(\varepsilon)}$. Plugging these Lipschitz bounds into \eqref{approx sol one step}, one arrives at
	\begin{equation*}
		\begin{cases*}
			q^{(\varepsilon)}(t) = q^{(\varepsilon)}(s) - (t-s)\qty[(\vv^{(\varepsilon)}(s)\vdot\grad)q^{(\varepsilon)}(s)+\beta q^{(\varepsilon)}(s) \div\vv^{(\varepsilon)}(s)] + \order{\abs{t-s}^2}, \\
			\vv^{(\varepsilon)}(t) = \vv^{(\varepsilon)}(s) - (t-s)\qty[(\vv^{(\varepsilon)}(s)\vdot\grad)\vv^{(\varepsilon)}(s)+\sigma^{(\varepsilon)}(s)\grad q^{(\varepsilon)}(s)] + \order{\abs{t-s}^2}, \\
			\sigma^{(\varepsilon)}(t) = \sigma^{(\varepsilon)}(s) - (t-s)\qty[(\vv^{(\varepsilon)}(s)\vdot\grad)\sigma^{(\varepsilon)}(s)]+\order{\abs{t-s}^2}, \\
			\Omega^{(\varepsilon)}(t) = \qty(\textup{Id}+(t-s)\vv^{(\varepsilon)}(s))\Omega^{(\varepsilon)}(s) + \order{\abs{t-s}^2},
		\end{cases*}
	\end{equation*}
	for all $t, s \in \varepsilon\mathbb{N}\cap[0, T]$.
	Thus, a careful application of the Arzel{\`a}-Ascoli theorem yields the existence of a subsequence converging uniformly to $(q, \vv, \sigma)$. By taking  weak limits, the energy bound \eqref{energy bound approx sol} also holds for $(q, \vv, \sigma)$.
	
	In conclusion, one has obtained
	\begin{theorem}\label{thm existence}
		Let $k$ be a sufficiently large integer and $M \ge 1$ a constant. For an initial state $(q_0, \vv_0, \sigma_0) \in \bbH^{2k}$ with bound $\norm{(q_0, \vv_0, \sigma_0)}_{\fkH^{2k}_{q_0}} \le M$, there is a constant $T = T(M) > 0$, so that the compressible Euler system \eqref{euler equiv2} admit a solution $(q, \vv, \sigma)$ in the space
		\begin{equation*}
			(q, \vv, \sigma) \in L^\infty\qty([0, T]; \bbH^{2k})\cap C\qty([0, T]; \bbH^{2k-2}).
		\end{equation*}
		Furthermore, the solution satisfies the bound
		\begin{equation*}
			\sup_{0 \le t \le T} \norm{(q, \vv, \sigma)}_{\fkH^{2k}_{q}} \lesssim_M 1.
		\end{equation*}
	\end{theorem}
	
	\subsection{Constructions of the One-step Iterates}
	In this subsection, we prove Theorem \ref{thm one step iterate} under the assumption of the following ancillary regularization proposition, the proof of which will be deferred for now.
	\begin{prop}\label{prop reg state}
		Given a state $(q_0, \vv_0, \sigma_0) \in \bbH^{2k}$ with energy bound \eqref{energy bound 0 state}, there exists a regular state $(\hat{q}, \hat{\vv}, \hat{\sigma})$ satisfying the following properties:
		\begin{enumerate}[label=(\roman*)]
			\item \textbf{Pointwise approximation:}
			\begin{equation}
				\norm{\hat{q} - q_0}_{L^\infty} \lesssim \varepsilon^2
			\end{equation}
			and
			\begin{equation}\label{reg pointwise bound}
				\left\{
				\begin{split}
					\hat{q}-q_0 &= \Order{C^2}{\varepsilon^2}\\
					\hat{\vv}-\vv_0 &= \Order{C^2}{\varepsilon^2}\\ \hat{\sigma}-\sigma_0 &= \Order{C^2}{\varepsilon^2}
				\end{split}\right. \qin \widehat{\Omega} \cap \Omega_0,
			\end{equation}
			where $\widehat{\Omega} \coloneqq \{\widehat{q}>0\}$.
			\item \textbf{Energy increment:}
			\begin{equation}\label{energy incre reg}
				\fkE_{2k}(\hat{q}, \hat{\vv}, \hat{\sigma}) \le \qty\big[1+C(M)\varepsilon]\fkE_{2k}(q_0, \vv_0, \sigma_0).
			\end{equation}
			\item \textbf{Higher order bound:}
			\begin{equation}\label{higher order bounds hat}
				\norm{(\hat{q}, \hat{\vv}, \hat{\sigma})}_{\fkH^{2k+2}_{\hat{q}}(\widehat{\Omega})} \lesssim_M \varepsilon^{-1}.
			\end{equation}
		\end{enumerate}
	\end{prop}
	
	With the help of Proposition \ref{prop reg state}, one can now construct the one-step iterates in Theorem \ref{thm one step iterate}. Note that, since one is dealing with a free boundary problem, the ODE structure becomes more transparent when computing the time derivative along the particle path. Specifically, in Lagrangian coordinates, the free boundary problem can be viewed as an ODE in certain state spaces, where the spatial variables are confined to a fixed domain. Based on this observation, one can first define the discrete flow transport by:
	\begin{equation}\label{transport x y}
		\vb*{x} \coloneqq \vb*{y} + \varepsilon\hat{\vv}(\vb*{y}).
	\end{equation}
	Hence, $(\vb*{x}-\vb*{y})$ could be viewed as the average fluid displacements. Inserting the source terms into the discrete evolutions, one can further define for $x \equiv y + \varepsilon\hat{\vv}(y)$
	\begin{equation}\label{def q1 v1 sigma1}
		\begin{cases*}
			q_1 (x) \coloneqq \hat{q}(y) - \varepsilon\beta\hat{q}(y)(\div\hat{\vv})(y), \\
			\vv_1 (x) \coloneqq \hat{\vv}(y) - \varepsilon\hat{\sigma}(y)(\grad{\hat{q}})(y), \\
			\sigma_1 (x) \coloneqq \hat{\sigma} (y).
		\end{cases*}
	\end{equation}
	It follows directly from \eqref{transport x y} that
	\begin{equation}
		\pdv{x}{y} = \textup{Id} + \varepsilon \grad{\hat{\vv}}(y),
	\end{equation}
	which implies that
	\begin{equation}
		\vb*{x} = \vb*{y} + \Order{M}{\varepsilon}
	 \qand
		\det(\pdv{x}{y}) = 1 + \Order{M}{\varepsilon}.
	\end{equation}
	Thus, \eqref{approx sol one step} can be derived from \eqref{reg pointwise bound} and the above constructions. It remains to show the energy bound \eqref{energy increments one step}.
	
	Note that \eqref{def q1 v1 sigma1} is exactly the discrete version for the compressible Euler system, which yields that the corresponding good unknowns satisfy the discrete ``evolution'' equations:
	\begin{equation}\label{discrete evo eqn good unknown}
		\begin{cases*}
			\fks_{2k}^1 (x) = \hat{\fks}_{2k} (y) - \varepsilon\qty[\hat{\fkw}_{2k} \vdot \grad{\hat{q}} + \beta\hat{q}(\div\hat{\fkw}_{2k}) - \hat{\sigma}^{-k}\hat{\fkf}_{2k}](y) + \varepsilon^2  \hat{\sigma}^{-k}\mathfrak{r}_{2k}(y), \\
			\fkw_{2k}^1 (x) = \hat{\fkw}_{2k}(y) - \varepsilon\qty[\hat{\sigma}\grad{\hat{\fks}_{2k}}-\hat{\sigma}^{-k}\hat{\fkg}_{2k}](y) + \varepsilon^2 \hat{\sigma}^{-k} \vb*{\mathfrak{R}}_{2k} (y), \\
			x \equiv y + \varepsilon\hat{\vv}(y),
		\end{cases*}
	\end{equation}
	where $(\hat{\fkf}_{2k}, \hat{\fkg}_{2k})$ are the multilinear forms given in \eqref{eqn evo modified good unknowns}, and $(\fkr_{2k} $, $\fkR_{2k})$ are multilinear forms of $\scrJ(\hat{q}, \pd\hat{\vv}, \hat{\sigma}, \varepsilon\grad{\hat{\vv}})$. Indeed, to derive the coefficients of the $\varepsilon$-terms, it can be seen from the constructions \eqref{transport x y}-\eqref{def q1 v1 sigma1} that they are merely the first order linearization about the derivative $\Dt$ in the dynamical problems. 
	
	To obtain the properties of the remainders $(\fkr_{2k}, \fkR_{2k})$, one may first assign $\varepsilon$ having order $-\frac{1}{2}$, since it plays the role of discrete time step-length. Thus, the error term $\varepsilon\grad\hat{\vv}$ in the Jacobian has exactly order zero, which is consistent to the scaling analysis in the continuous case. Direct calculations yield that $(\fkr_{2k}, \fkR_{2k})$ satisfy the following properties:
	\begin{itemize}
		\item They have orders $k$ and $k+\frac{1}{2}$, respectively.
		\item Besides the powers of the Jacobian, they contain exactly $(2k+2)$ derivatives applied to the factors $\scrJ(\hat{q}, \pd\hat{\vv}, \hat{\sigma}, \varepsilon\grad{\hat{\vv}})$.
		\item They are both composed of admissible errors, i.e., containing at least two factors having strictly positive order.
	\end{itemize}
	It follows from Lemma \ref{lem err est B} that
	\begin{equation}
		\norm*{\hat{\sigma}^{-k}\hat{\fkf}_{2k}}_{H^{0, \frac{\alpha-1}{2}}_{\hat{q}}(\widehat{\Omega})} + \norm*{\hat{\sigma}^{-k}\hat{\fkg}_{2k}}_{H^{0, \frac{\alpha}{2}}_{\hat{q}}(\widehat{\Omega})} \lesssim_{\hat{A}_*} \hat{B} \norm{(\hat{q}, \hat{\vv}, \hat{\sigma})}_{\fkH^{2k}_{\hat{q}}} \lesssim_M 1.
	\end{equation}
	Similarly, one can derive from \eqref{higher order bounds hat} and the interpolations that
	\begin{equation}
		\norm{\fkr_{2k}}_{H^{0, \frac{\alpha-1}{2}}_{\hat{q}}(\widehat{\Omega})} + \norm{\fkR_{2k}}_{H^{0, \frac{\alpha}{2}}_{\hat{q}}(\widehat{\Omega})} \lesssim_M \norm{(\hat{q}, \hat{\vv}, \hat{\sigma})}_{\fkH^{2k+1}_{\hat{q}}} \lesssim_M \varepsilon^{-1}.
	\end{equation}
	
	Before comparing the energies between $(q_1, \vv_1, \sigma_1)$ and $(\hat{q}, \hat{\vv}, \hat{\sigma})$, one may first note that these two set of quantities are defined in different domains. Thanks to the constructions \eqref{transport x y}-\eqref{def q1 v1 sigma1}, one can obtain that
	\begin{equation*}
		\dd{x} = \qty\big(1+\Order{M}{\varepsilon})\dd{y},
	\end{equation*} 
	and
	\begin{equation*}
		q_1 (y + \varepsilon\hat{\vv}(y)) = \hat{q}(y) + \Order{M}{\epsilon}.
	\end{equation*}
	Therefore, it is standard to calculate that
	\begin{equation}\label{energy hat cal 1}
		\begin{split}
			&\int_{\Omega_1} \abs{q_1(x)}^{\alpha-1}\abs{\fks_{2k}^1 (x)}^2 + \beta\abs{q_1(x)}^\alpha\sigma_1^{-1}(x)\abs{\fkw_{2k}^1(x)}^2 \dd{x} \\
			&\quad = \int_{\widehat{\Omega}} \abs{\hat{q}(y)}^{\alpha-1}\abs{\fks_{2k}^1\qty\big[y+\varepsilon\hat{\vv}(y)]}^2 + \beta \abs{\hat{q}(y)}^\alpha \hat{\sigma}^{-1}(y) \abs{\fkw_{2k}^1\qty\big[y+\varepsilon\hat{\vv}(y)]}^2 \dd{y} + \Order{M}{\varepsilon}.
		\end{split}
	\end{equation}
	Combining \eqref{discrete evo eqn good unknown}-\eqref{energy hat cal 1} yields that
	\begin{equation}
		\begin{split}
			&\hspace*{-2em}\int_{\Omega_1} \abs{q_1(x)}^{\alpha-1}\abs{\fks_{2k}^1 (x)}^2 + \beta\abs{q_1(x)}^\alpha\sigma_1^{-1}(x)\abs{\fkw_{2k}^1(x)}^2 \dd{x} \\
			 \le &\,\int_{\widehat{\Omega}} \hat{q}^{\alpha-1}(y) \qty[\abs{\hat{\fks}_{2k}}^2 + \beta \hat{q}\hat{\sigma}^{-1}\abs{\hat{\fkw}_{2k}}^2](y) \dd{y} \\
			&-2\varepsilon\int_{\widehat{\Omega}} \hat{q}^{\alpha-1}(y) \qty\Big[\hat{\fks}_{2k}\qty(\hat{\fkw}_{2k} \vdot \grad{\hat{q}} + \beta\hat{q}\div\hat{\fkw}_{2k})+\beta\hat{q}\hat{\fkw}_{2k}\vdot\grad{\hat{\fks}_{2k}}](y) \dd{y} \\
			&+ \varepsilon^2 \int_{\widehat{\Omega}} \hat{q}^{\alpha-1}(y)\qty[\abs{\hat{\fkw}_{2k} \vdot \grad{\hat{q}} + \beta\hat{q}(\div\hat{\fkw}_{2k})}^2 + \beta\hat{q}\hat{\sigma}^{-1}\abs{\hat{\sigma}\grad{\hat{\fks}_{2k}}}^2](y) \dd{y} \\
			&+ \Order{M}{\varepsilon} \\
			\le &\, \int_{\widehat{\Omega}} \hat{q}^{\alpha-1}(y) \qty[\abs{\hat{\fks}_{2k}}^2 + \beta \hat{q}\hat{\sigma}^{-1}\abs{\hat{\fkw}_{2k}}^2](y) \dd{y} \\
			&+  \Order{M}{\varepsilon^2} \qty(\norm{\hat{\fks}_{2k}}_{H^{1, \frac{\alpha}{2}}_{\hat{q}}}^2 + \norm{\hat{\fkw}_{2k}}_{H^{1, \frac{\alpha+1}{2}}_{\hat{q}}}^2) +\Order{M}{\varepsilon}, \\
			\le &\, \int_{\widehat{\Omega}} \hat{q}^{\alpha-1}(y) \qty[\abs{\hat{\fks}_{2k}}^2 + \beta \hat{q}\hat{\sigma}^{-1}\abs{\hat{\fkw}_{2k}}^2](y) \dd{y} + \Order{M}{\varepsilon},
		\end{split}
	\end{equation}
	where the second inequality follows from the cancellation, and the third one follows from the interpolation Proposition \ref{prop interp} together with the energy bound \eqref{higher order bounds hat}.	Similar arguments applied to the discrete version of \eqref{eqn Dt vom 2k}-\eqref{eqn Dt sigma 2k} yield that
	\begin{equation*}
		\fkE_{2k, \text{high}}^1 \le \widehat{\fkE}_{2k, \text{high}} + \Order{M}{\varepsilon}.
	\end{equation*}
	The lower order energy bounds can be derived from the pointwise estimates \eqref{reg pointwise bound} routinely.
	
	Thus, one concludes the proof of Theorem \ref{thm one step iterate}, provided that Proposition \ref{prop reg state} holds. \\ \qed

	\subsection{Regularization Operators}\label{sec reg op}
	The remaining parts of this section are devoted to the proof of Proposition \ref{prop reg state}. Before proceeding, we first consider the regularization operators adapted to the physical vacuum problems. For more motivations and detailed discussions, one can refer to \cite[\S2.4]{Ifrim-Tataru2024}.

	The initial goal of the regularization is to generate a regular datum corresponding to each dyadic scale $2^h$ for $h \ge 0$. However, given a $C^{1+}$ domain $\Omega$ and its non-degenerate defining function $r$, directly truncating the frequency at scale $2^h$ for a function $f$ defined on $\Omega$ is not suitable for the weighted Sobolev spaces used in this manuscript. Due to the ellipticity of $\cL_1$, whose leading term is $r\laplace$, it is natural to choose a frequency $\xi$ such that $r\xi^2 \lesssim 2^{2h}$, necessitating the regularization kernels at the scale
	\begin{equation*}
		\var{x} \approx r^\frac{1}{2}2^{-h}.
	\end{equation*}
	On the other hand, when $r \ll 2^{-2h}$, one encounters the scenario
	\begin{equation*}
		\var{x} \gg r,
	\end{equation*}
	which would require the information of $f$ outside $\Omega$. One direct scheme to overcome this obstacle is to neglect the information of $f$ in the region $\qty{r \ll 2^{-2h}}$. More precisely, denote by
	\begin{equation}
		\Omega^{[h]} \coloneqq \Set*{x \in \Omega \given r(x) \simeq 2^{-2h}} \qfor h\ge 1,
	\end{equation}
	the dyadic boundary layer associated to the frequency $2^{h}$. One can also define the full boundary strip via
	\begin{equation}
		\Omega^{[>h]} \coloneqq \Set*{x \in \Omega \given r(x) \lesssim 2^{-2h}} = \bigcup_{j>h}\Omega^{[j]},
	\end{equation}
	and the interior counterpart by
	\begin{equation}
		\Omega^{[<h]} \coloneqq \Set*{x \in \Omega \given r(x) \gtrsim 2^{-2h}} = \bigcup_{0 \le j < h} \Omega^{[j]},
	\end{equation}
	here
	\begin{equation*}
		\Omega^{[0]} \coloneqq \Set*{x \in \Omega \given r(x) \gtrsim 1}.
	\end{equation*}
	The enlargement of $\Omega$ at the frequency scale $2^{h}$ is given by
	\begin{equation}
		\check{\Omega}^{[h]} \coloneqq \Set*{x \in \R^d \given \dist(x, \Omega) \le c2^{-2h}},
	\end{equation}
	here $c > 0$ is a generic small constant. Consider a regularization kernel $K^h (x, y)$ defined by
	\begin{equation}\label{def Kh}
		K^h(x, y) = \theta_{>h}(x)K^h_0(x-y) + \sum_{0 \le \ell \le h}\theta_{\ell}(x)K^\ell_0(x-y),
	\end{equation}
	where $\theta_{\ell}$ is a cutoff function associated to $\Omega^{[\ell]}$, and $\theta_{>h}$ is identical to $1$ outside  $\Omega$. The primitive kernel $K^\ell_0 (z)$ is defined as
	\begin{equation*}
		K^\ell_0(z) \coloneqq 2^{2\ell d} \varphi\qty(2^{2\ell}z),
	\end{equation*}
	where $\varphi$ is a smooth function given by
	\begin{equation*}
		\varphi(x) = \sum_{1 \le m \le n} \kappa_{m}(z)\phi_m(z).
	\end{equation*}
	Here $\{\kappa_m\}$ is a partition of unities of a neighborhood of $\Omega$, so that there are unit vectors $\vb{e}_{(m)}$ uniformly outward transversal to $\pd\Omega$ in $\spt(\kappa_m)$. The smooth functions $\phi_m $ satisfy
	\begin{enumerate}[label=(\roman*)]
		\item support conditions: $\spt(\phi_m) \subset B(\vb{e}_{(m)}, \delta_*)$, $\delta_* \ll 1$;
		\item average one: $\int_{\R^d} \phi_m = 1$;
		\item moment conditions: $\int z^\gamma \phi_m (z) \dd{z} = 0$, $1 \le \abs{\gamma} \le N$, for some large $N$.
	\end{enumerate}
	Thus, the kernel $K^h$ admits the following properties:
	\begin{enumerate}
		\item $K^h \colon \check{\Omega}^{[h]}\times\Omega^{[<h]} \to \R$;
		\item $\spt(K^h) \subset \Set*{(x, y) \in \check{\Omega}^{[h]}\times\Omega^{[<h]} \given \abs{x-y} \lesssim 2^{-2h} + 2^{-h}r^\frac{1}{2}(y)}$;
		\item $\abs{\pd_x^\mu \pd_y^\nu K^h(x, y)} \lesssim \abs*{2^{-2h}+2^{-h}r^\frac{1}{2}(y)}^{-(d+\abs{\mu}+\abs{\nu})}$ for multi-indices $\mu, \nu$;
		\item $\int K^h(x, y) \dd{y} = 1$;
		\item $\int (x-y)^\gamma K^h(x, y) \dd{y} = 0$, $1 \le \abs{\gamma} \le N$.
	\end{enumerate}
	
	For a function $u$ defined in $\Omega$, define the regularization of $u$ at the $h$-dyadic-scale by:
	\begin{equation}
		\vps_{\le h}(u) \coloneqq \int K^h(x, y) u(y) \dd{y}.
	\end{equation}
	Thus, the regularized data is defined in an enlarged domain of $\Omega$. Furthermore, the following proposition holds (cf. \cite[Proposition 2.10]{Ifrim-Tataru2024}):
	\begin{prop}\label{prop reg op}
		Suppose that  $r$ is a non-degenerate defining function of $\Omega$, and $r_1$ is a non-degenerate defining function of $\Omega_1$ with $\norm{r-r_1}_{L^\infty} \ll 2^{-2h}$. Then, for each parameter $\lambda > 0$, the regularization operator $\vps_{\le h}$ admits the following properties:
		\begin{enumerate}
			\item \textbf{Regularization bounds:}
			\begin{equation*}
				\norm{\vps_{\le h} (u)}_{H^{2k+2j, k+j+\frac{\lambda-1}{2}}_{r_1}(\Omega_1)} \lesssim 2^{2jh}\norm{u}_{H^{2k, k+\frac{\lambda-1}{2}}_r(\Omega)} \qc j \ge 0.
			\end{equation*}
			\item \textbf{Difference bounds:}
			\begin{equation*}
				\norm{(\vps_{\le h+1}-\vps_{\le h})u}_{H^{2k-2j, k-j+\frac{\lambda-1}{2}}_{r_1}(\Omega_1)} \lesssim 2^{-2jh}\norm{u}_{H^{2k, k+\frac{\lambda-1}{2}}_r(\Omega)} \qc 0 \le j \le k.
			\end{equation*}
			\item \textbf{Error bounds:}
			\begin{equation*}
				\norm{(\textup{Id}-\vps_{\le h})u}_{H^{2k-2j, k-j+\frac{\lambda-1}{2}}_{r}(\Omega)} \lesssim 2^{-2jh} \norm{u}_{H^{2k, k+\frac{\lambda-1}{2}}_r (\Omega)} \qc 0 \le j \le k.
			\end{equation*}
		\end{enumerate}
	\end{prop}

	\subsection{Constructions of the Regularized Data}\label{sec reg data}
	
	Now, we turn to the proof of Proposition \ref{prop reg state}. Assume that $\varepsilon = 2^{-2h}$. A straightforward application of the regularization operator $\vps_{\le h}$ to $(q_0, \vv_0, \sigma_0)$ may not preserve the energy bound \eqref{energy incre reg}, as it may lose significant information about high-frequency cancellations. To correct this, it might be necessary to add some high-frequency components to the regularized data. Conversely, regularizations at a higher scale, say $\vps_{\le 2h}$, will not maintain the higher order bound \eqref{higher order bounds hat}. Therefore, an adapted regularization is possibly not achievable in a single step. In view of the energy expression \eqref{def fkE 2k}, the $H^{0, \ala}$ norm of $(\cL_1)^k q$ represents the highest order information. To obtain the higher order bounds at a frequency scale $2^{h}$ while controlling the energy loss, one can consider a two-scale regularization. More precisely, take the appropriate lower- and higher-frequency scales:
	\begin{equation*}
		1 \ll h_- < h < h_+,
	\end{equation*}
	whose explicit choices will be determined later. Define the associated regularization by:
	\begin{equation*}
		(q_-, \vv_-, \sigma_-) \coloneqq \vps_{\le h_-} (q_0, \vv_0, \sigma_0) \qand (q_+, \vv_+, \sigma_+) \coloneqq \vps_{\le h_+}(q_0, \vv_0, \sigma_0).
	\end{equation*}
	Consider a correction function $\chi_\varepsilon$ defined on $\R_+$, so that
	\begin{equation}
		\begin{cases*}
			\chi_\varepsilon (t) \le 1, \\
			1-\chi_\varepsilon(t) \lesssim \varepsilon^2 t, \\
			\chi_\varepsilon (t) \lesssim (\varepsilon t)^{-1}, \\
			(\varepsilon t) \chi_\varepsilon \cdot (1-\chi_\varepsilon) \gtrsim (1-\chi_\varepsilon)^2 + (\epsilon^2 t)^2 \chi_\varepsilon^2.
		\end{cases*}
	\end{equation}
	These requirements may seem strange at a first glance, but they are quite natural in view of \eqref{L1 diff trans}-\eqref{req chi ep 2}. One simple candidate satisfying the above requirements is
	\begin{equation}\label{def chi epsilon}
		\chi_\varepsilon(t) \coloneqq \begin{cases*}
			1-\varepsilon^2 t, &if $0 \le t \le \varepsilon^{-1}(\varepsilon^{-1}-1)$; \\
			(1+\varepsilon t)^{-1}, &if $t > \varepsilon^{-1}(\varepsilon^{-1}-1)$.
		\end{cases*}
	\end{equation}
	Then, define the intermediate regularization states by:
	\begin{equation}\label{def tilde quantities}
		\left\{
		\begin{split}
			\wt{q} &\coloneqq q_- + \chi_\varepsilon(\cL_1^-)(q_+ - q_-)_{\restriction_{\Omega_-}}, \\
			\wt{\vv} &\coloneqq \vv_- + (\sigma_-)^{\frac{1}{2}} \cdot \chi_\varepsilon(\cL_2^- + \cL_3^-)\qty[(\sigma_-)^{-\frac{1}{2}}(\vv_+ - \vv_-)]_{\restriction_{\Omega_-}}, \\
			\wt{\sigma} &\coloneqq \sigma_- + \chi_\varepsilon(\cL_5^-)(\sigma_+ - \sigma_-)_{\restriction_{\Omega_-}},
		\end{split}
		\right.
	\end{equation}
	where $\Omega_- \coloneqq \qty{q_- > 0}$, $\cL_j^- \coloneqq \cL_j[q_-]$  $(0 \le j \le 5)$, and the operators $\chi_\varepsilon(\cL_1^-), \chi_\varepsilon(\cL_2^- + \cL_3^-)$, and $\chi_\varepsilon(\cL_5^-) $ are interpreted in the sense of functional calculus for the self-adjoint operators on Hilbert spaces. One may note that the Sobolev embeddings and Proposition \ref{prop reg op} ensure that
	\begin{equation}\label{q pm - q0 L infty}
		\abs{q_\pm - q_0}_{L^\infty} \lesssim 2^{-2h_\pm(k+1-\varkappa_0)}.
	\end{equation}
	Thus, so long as
	\begin{equation}\label{rel h- h+}
		h_-(k+1-\varkappa_0) > h_+ + c
	\end{equation}
	for some generic constant $c$, there holds
	\begin{equation*}
		\Omega_- \subset \check{\Omega}^{[h_+]},
	\end{equation*}
	which implies that the quantities $(q_+, \vv_+, \sigma_+)$ are well-defined in $\Omega_-$. Namely, the definitions in \eqref{def tilde quantities} are legitimate. The quantities $(\wt{q}, \wt{\vv}, \wt{\sigma})$ can be regarded as the regularization of $(q_0, \vv_0, \sigma_0)$ at frequency scale $2^{h_-}$ with proper high-frequency corrections. In view of the fact that $\wt{q}$ may not vanish on $\pd\Omega_-$, one should further refine the regularization. Note that $k \gg 1$, so \eqref{q pm - q0 L infty} yields that
	\begin{equation*}
		\abs{\wt{q}} \ll 2^{-2Ch} \qq{on} \pd{\Omega_-}
	\end{equation*}
	for some large $C$. Hence, one could define
	\begin{equation}\label{def hat data}
		(\hat{q}, \hat{\vv}, \hat{\sigma}) \coloneqq \qty(\wt{q}-2^{-2Ch}, \wt{\vv}, \wt{\sigma}) \qand \widehat{\Omega}\coloneqq\qty{\hat{q}>0}.
	\end{equation}
	This would be the final regularizations. One can also refer to \cite[pp. 474-477]{Ifrim-Tataru2024} for the original heuristics.
	
	The $C^2$-approximation property \eqref{reg pointwise bound} follows from the constructions, Proposition \ref{prop reg op}, and the Sobolev embeddings. Thus, it only remains to show the energy increment estimate \eqref{energy incre reg} and the higher order bound \eqref{higher order bounds hat}.
	
	For the simplicity of notations, we will denote by:
	\begin{subequations}
		\begin{equation}
			\norm{u}_{H^{0, \lambda}_{q_-}} \coloneqq \norm{u}_{H^{0, \lambda}_{q_-} (\Omega_-)} \qfor \lambda>-\frac{1}{2},
		\end{equation}
		and
		\begin{equation}
			\norm{(\wt{q}, \wt{\vv}, \wt{\sigma})}_{\fkH^{2k}_{q_-}}^2 \coloneqq \norm{\wt{q}}_{H^{2k, k+\ala}_{q_-}}^2 + \norm{\wt{\vv}}_{H^{2k, k+\alb}_{q_-}}^2 + \norm{\wt{\sigma}}_{H^{2k, k+\alb}_{q_-}}^2.
		\end{equation}
	\end{subequations}

	\subsubsection{Bounds for the regularized data}
	It follows from \eqref{def chi epsilon}-\eqref{def tilde quantities} that
	\begin{equation*}
		\norm{(\cL_1^-)^{k+1}(\wt{q}-q_-)}_{H^{0, \frac{\alpha-1}{2}}_{q_-}(\Omega_-)} \lesssim \varepsilon^{-1}\norm{(\cL_1^-)^k (q_+ - q_-)}_{H^{0, \ala}_{q_-}(\Omega_-)},
	\end{equation*}
	which, together with Proposition \ref{prop reg op}, yield that
	\begin{equation*}
		\norm{(\cL_1^-)^{k+1} \wt{q}}_{H^{0, \ala}_{q_-}(\Omega_-)} \lesssim_M \varepsilon^{-1}.
	\end{equation*}
	Similarly, there holds
	\begin{equation*}
		\norm{(\cL_2^- + \cL_3^-)^{k+1}\qty[(\sigma_-)^{-\frac{1}{2}}\wt{\vv}]}_{H^{0, \alb}_{q_-}} + \norm{(\cL_5^-)^{k+1}\wt{\sigma}}_{H^{0, \alb}_{q_-}} \lesssim_M \varepsilon^{-1}.
	\end{equation*}
	Moreover, Proposition \ref{prop reg op} implies that
	\begin{equation*}
		\norm{(\wt{q}-q_0, \wt{\vv}-\vv_0, \wt{\sigma}-\sigma_0)}_{\fkH^0_{q_0}(\Omega_0)} \lesssim_M 2^{-2kh_-} \lesssim_M \varepsilon^2,
	\end{equation*}
	and
	\begin{equation*}
		\norm{(\wt{q}, \wt{\vv}, \wt{\sigma})}_{\fkH^{0}_{q_-}(\Omega_- \setminus\Omega_0)} \lesssim_M \varepsilon^2.
	\end{equation*}
	Therefore, it can be derived from the ellipticity of operators $\cL_1^-, (\cL_2^- +\cL_3^-)$, and $\cL_5^-$ that
	\begin{equation*}
		\norm{\qty(\wt{q}, (\sigma_-)^{-\frac{1}{2}}\wt{\vv}, \wt{\sigma})}_{\fkH^{2k+2}_{q_-}(\Omega_-)} \lesssim_M \varepsilon^{-1}.
	\end{equation*}
	which, together with \eqref{equiv fkE 2k new norm}-\eqref{equiv two norms}, yield 
	\begin{equation}\label{bound wt}
		\norm{\qty(\wt{q}, \wt{\vv}, \wt{\sigma})}_{\fkH^{2k+2}_{q_-}(\Omega_-)} \lesssim_M \varepsilon^{-1}.
	\end{equation}
	Furthermore, since $\chi_\varepsilon \le 1$ on $\R_+$, it follows from Proposition \ref{prop reg op} that
	\begin{equation}
		\norm{(\wt{q}, \wt{\vv}, \wt{\sigma})}_{\fkH^{2k}_{q_-}(\Omega_-)} \lesssim_M 1.
	\end{equation}
	
	On the other hand, it holds that
	\begin{equation*}
		q_+ - \wt{q} =\qty\big[( 1 - \chi_\varepsilon)(\cL_1^-)](q_+ - q_-)_{\restriction_{\Omega_-}},
	\end{equation*}
	and
	\begin{equation*}
		\norm{(\cL_1^-)^{k-1}(\wt{q}-q_+)}_{\Hala_{q_-}} \lesssim \varepsilon^2 \norm{(\cL_1^-)^k(q_+ - q_-)}_{\Hala_{q_-}} \lesssim_M \varepsilon^2.
	\end{equation*}
	In a similar manner, one can establish the corresponding bounds for $(\wt{\vv}-\vv_+, \wt{\sigma}-\sigma_+)$ and finally arrives at the estimate:
	\begin{equation}\label{est q+ - q wt}
		\norm{(\wt{q}-q_+, \wt{\vv}-\vv_+, \wt{\sigma}-\sigma_+)}_{\fkH^{2k-2}_{q_-}(\Omega_-)} \lesssim_M \varepsilon^2.
	\end{equation}
	Indeed, \eqref{bound wt}-\eqref{est q+ - q wt} formally implies the regularized data $(\wt{q}, \wt{\vv}, \wt{\sigma})$ have the benefits from both the high-scale and on-scale regularization.
	
	\subsubsection{Energy increments}
	Now, we turn to the proof of \eqref{energy incre reg}, which is the main bulk of the work. As a necessary intermediate step, one may first compare the energy between $\fkE_{2k, \text{high}}^- (\wt{q}, \wt{\vv}, \wt{\sigma})$ and $\fkE_{2k, \text{high}}(q_0, \vv_0, \sigma_0)$, where $\fkE_{2k, \text{high}}^- (\wt{q}, \wt{\vv}, \wt{\sigma})$ is a shorthand notation of
	\begin{equation*}
		\int_{\Omega_-} q_-^{\alpha-1} \qty[\abs{\wt{\fks}_{2k}}^2 + \beta q_- \sigma_-^{-1}\abs{\wt{\fkw}_{2k}}^2 + \beta q_- \sigma_-^{-1}\abs{\wt{\vom}_{2k}}^2 + q_-\abs{\wt{\sigma}_{2k}}^2] \dd{x}.
	\end{equation*}
	To compare these two energies, there are two main challenges. One is that, as usual, they are defined in different domains. The other is that the good unknowns are multilinear forms involving derivatives, the estimates of their difference would be rather subtle. Let us now handle the first obstacle.
	
	\paragraph*{Step 1: Removing a boundary layer}
	Recall that \eqref{q pm - q0 L infty} implies that
	\begin{equation}
		\Omin \coloneqq	\Omega_0^{[<h_-(k+1-\varkappa_0) -h]} \subset \Omega_-,
	\end{equation}
	and that
	\begin{equation}\label{dif q om in}
			\abs{q_0 - q_-}(x) \lesssim_M 2^{-2h_-(k+1-\varkappa_0)} \lesssim_M 2^{-2h} 2^{-2[h_-(\kk)- h]} \lesssim_M \varepsilon q_0 (x) \qfor x \in \Omin.
	\end{equation}

	In order to compute the energies in $(\Omega_- \setminus \Omin)$, one can first observe that, in the $h_1$-boundary-layer of $\Omega_-$, there holds
	\begin{equation*}
		\begin{split}
			\norm{\wt{q}}_{H^{2k, k+\ala}_{q_-}(\Omega_-^{[h_1]})} &\lesssim \norm{\wt{q}}_{H^{2k+2, k+2+\ala}_{q_-}(\Omega_-^{[h_1]})} \\
			&\lesssim 2^{-2h_1} \norm{\wt{q}}_{H^{2k+2, k+1+\ala}_{q_-}(\Omega_-^{[h_1]})} \\
			&\lesssim_M 2^{-2(h_1 - h)},
		\end{split}
	\end{equation*}
	whose last inequality follows from the bound \eqref{bound wt}. Similar arguments lead to
	\begin{equation*}
		\norm{(\wt{\vv}, \wt{\sigma})}_{H^{2k, k+\ala}_{q_-}(\Omega_-^{[h_1]})} \lesssim_M 2^{-2(h_1 - h)}.
	\end{equation*}
	Thus, it follows from the arguments in \S\ref{sec bounds of good unknown} that
	\begin{equation}\label{est Om1 bdry layer}
		\begin{split}
			&\int_{\Omega_-^{[h_1]}} q_-^{\alpha-1} \qty[\abs{\wt{\fks}_{2k}}^2 + \beta q_- \sigma_-^{-1}\abs{\wt{\fkw}_{2k}}^2 + \beta q_- \sigma_-^{-1}\abs{\wt{\vom}_{2k}}^2 + q_-\abs{\wt{\sigma}_{2k}}^2] \dd{x} \\
			&\quad\lesssim_M \norm{(\wt{q}, \wt{\vv}, \wt{\sigma})}^2_{\fkH^{2k}_{q_-}(\Omega_-^{[h_1]})} \lesssim_M 2^{-4(h_1-h)}.
		\end{split}
	\end{equation}
	On the other hand, in order that $(\Omega_0^{[>h_-(\kk)- h]} \cap \Omega_-) \subset \Omega_-^{[>2h]}$, it suffices to have
	\begin{equation*}
		2^{-2h_-(\kk)+2 h} + 2^{-2h_-(\kk)} \lesssim 2^{-4h},
	\end{equation*}
	which automatically holds whenever
	\begin{equation}\label{h- kk 3h}
		h_-(\kk) > 3h.
	\end{equation}
	Hence, \eqref{est Om1 bdry layer}-\eqref{h- kk 3h} imply that
	\begin{equation}
		\int_{\Omega_- \setminus \Omin} q_-^{\alpha-1} \qty[\abs{\wt{\fks}_{2k}}^2 + \beta q_- \sigma_-^{-1}\abs{\wt{\fkw}_{2k}}^2 + \beta q_- \sigma_-^{-1}\abs{\wt{\vom}_{2k}}^2 + q_-\abs{\wt{\sigma}_{2k}}^2] \dd{x} \lesssim_M \varepsilon^2.
	\end{equation}
	
	In other words, it suffices to compare the restrictions of $\fkE_{2k, \text{high}}^- (\wt{q}, \wt{\vv}, \wt{\sigma})$ and $\fkE_{2k}(q_0, \vv_0, \sigma_0)$ to $\Omin$. Here one may notice that \eqref{h- kk 3h} and \eqref{q pm - q0 L infty} also imply
	\begin{equation}\label{dif q0 - q- (2)}
		\abs{q_0 - q_-} \lesssim_M \varepsilon q_- \qin \Omin.
	\end{equation}
	Moreover, it follows from Proposition \ref{prop reg op} and \eqref{h- kk 3h} that
	\begin{equation}\label{dif sigma om in}
		\abs{\sigma_- - \sigma_0} \lesssim_M \varepsilon^{2} \sigma_- \qand \abs{\sigma_- - \sigma_0} \lesssim_M \varepsilon^{2} \sigma_0 \qin \Omin.
	\end{equation}
	Thus, when comparing the two energies in the interior region $\Omin$, one will not be bothered by the different weights, as interchanging $q_0 \leftrightarrow q_-$ and $\sigma_0 \leftrightarrow \sigma_-$ would merely produce $\Order{M}{\varepsilon}$-errors. Indeed, in $\Omin$, it holds that
	\begin{equation*}
		q_- \approx (1 \pm \varepsilon) q_0 \implies q_-^{\alpha-1} \approx [1 \pm(\alpha-1)\varepsilon]q_0^{\alpha-1} \approx \qty[1+\order{\varepsilon}]q_0^{\alpha-1},
	\end{equation*}
	and vice versa.
	
	\paragraph*{Step 2: Differences of good unknowns}
	First note that the original good unknown $s_{2k}$ defined in \eqref{def good unknown}, which does not include negative powers of $\sigma$, is a multilinear form of $(\scrJ(q), \scrJ(\pd\vv), \scrJ(\sigma))$. Thus, algebraic calculations yield that the difference $(\wt{s}_{2k} - s_{2k}^0)$ is a multilinear form of $(\wt{q}, \wt{\vv}, \wt{\sigma})$, $(\wt{q}-q_0, \wt{\vv}-\vv_0, \wt{\sigma}-\sigma_0)$, and their derivatives. More explicitly, one has the following relation:
	\begin{equation*}
		s_{2k}^0 - \wt{s}_{2k} =(\var{s_{2k}})_{\restriction_{(\wt{q}, \wt{\vv}, \wt{\sigma})}}(q_0 - \wt{q}, \vv_0 - \wt{\vv}, \sigma_0 - \wt{\sigma}) + F_{2k},
	\end{equation*}
	here $(\var{s_{2k}})$ is the variational derivative of a functional, and $F_{2k}$ is composed of the above stated multilinear forms and at least bilinear in $\scrJ(\wt{q}-q_0, \wt{\vv}-\vv_0, \wt{\sigma}-\sigma_0)$. For the first term on the right hand side, one can further compute that
	\begin{equation*}
		\begin{split}
			&(\var{s_{2k}})_{\restriction_{(\wt{q}, \wt{\vv}, \wt{\sigma})}}(q_0 - \wt{q}, \vv_0 - \wt{\vv}, \sigma_0 - \wt{\sigma}) \\
			&\quad = \wt{\sigma}^k (\cL_1[\wt{q}])^k (q_0 - \wt{q}) + \underbrace{\Lambda_{2k}[\wt{q}, \wt{\vv}, \wt{\sigma}](q_0 - \wt{q}, \vv_0 - \wt{\vv}, \sigma_0 - \wt{\sigma})}_{\eqqcolon F_{2k}'},
		\end{split}
	\end{equation*}
	here the coefficient $ \Lambda_{2k}[\wt{q}, \wt{\vv}, \wt{\sigma}]$ consists in multilinear forms containing at least one factor having strictly positive order. Indeed, for $k \ge 2$, the above relation can be derived from Lemma \ref{lem recur} with routine computations. For $k = 1$, one may recall from \eqref{eqn Dt2 q} and \eqref{def good unknown} that
	\begin{equation*}
		s_2 = \beta q\sigma \laplace{q} + \frac{\sigma}{2}\abs{\grad{q}}^2 + \beta q \grad{q}\vdot\grad{\sigma} + \beta q \qty(\beta(\divergence\vv)^2 + \tr\qty[(\grad\vv)^2]),
	\end{equation*}
	which implies that
	\begin{equation*}
		(\var{s_2})_{\restriction_{(q, \vv, \sigma)}} u = \sigma \cL_1[q] u + \beta\sigma\laplace q \cdot u + \order{\pd\vv\pd\vv} u.
	\end{equation*}	
	Thus, it holds that $\Lambda_{2}[q, \vv, \sigma]$ also contains at least one factor having strictly positive order.
	
	Note that the regularized data $(\wt{q}, \wt{\vv}, \wt{\sigma})$ can be formally regarded as the low-frequency component of $(q_0, \vv_0, \sigma_0)$, and thus $(q_0 - \wt{q}, \vv_0 - \wt{\vv}, \sigma_0 - \wt{\sigma})$ can be considered the high-frequency part. Specifically, the remainder $F_{2k}$ can be viewed as ``high-high'' interactions, which are usually not problematic. The remainder term $F_{2k}'$ can be seen as ``low-high balanced'' terms, which are more challenging than the ``high-high'' terms but still manageable. Heuristically, the main difficulty arises from the highest order term $\wt{\sigma}^k(\cL_1[\wt{q}])^k (q_0 - \wt{q})$. Note that the $\wt{\sigma}^k$ coefficient disrupts the symmetry of the differential operator, so we aim to avoid this by introducing the modified good unknown $\wt{\fks}_{2k} \coloneqq \wt{\sigma}^{-k}\wt{s}_{2k}$.
	
	Since the original good unknown $s_{2k}^0 $ certainly contains at least one factor of positive order, the term $(\sigma_0^{-k} - \wt{\sigma}^{-k})s_{2k}^0$ can be treated as ``low-high balanced'' or ``high-high'' terms. Namely, for the modified good unknowns $\fks_{2k}^0$ and $\wt{\fks}_{2k}$ defined by \eqref{def mod good unknown}, it holds that
	\begin{subequations}\label{diff formula gu 2k}
		\begin{equation}\label{diff formula s2k}
			\fks_{2k}^0 - \wt{\fks}_{2k} = \qty\big(\wt{\cL}_1)^k (q_0 - \wt{q}) + \sigma_0^{-k} \wt{\sigma}^{-k} \qty(\fkF_{2k}' + \fkF_{2k}),
		\end{equation}
		where $\wt{\cL}_1 \coloneqq \cL_1[\wt{q}]$, $\fkF_{2k}'$ and $\fkF_{2k}$ are multilinear forms of $(\wt{q}, \wt{\vv}, \wt{\sigma})$, $(\wt{q}-q_0, \wt{\vv}-\vv_0, \wt{\sigma}-\sigma_0)$, and their derivatives. Here $\fkF_{2k}$ is at least bilinear in $\scrJ(\wt{q}-q_0, \wt{\vv}-\vv_0, \wt{\sigma}-\sigma_0)$, and $\fkF_{2k}'$ is linear in $\scrJ(\wt{q}-q_0, \wt{\vv}-\vv_0, \wt{\sigma}-\sigma_0)$ and contains at least one factor of $\scrJ(\pd^2\wt{q}, \pd\wt{\vv}, \pd\wt{\sigma})$.
		
		Similarly, one can calculate that
		\begin{equation}
			\fkw_{2k}^0 - \wt{\fkw}_{2k} = \qty\big(\wt{\cL}_2)^k (\vv_0 - \wt{\vv}) + \sigma_0^{-k} \wt{\sigma}^{-k}\qty(\fkG_{2k}' + \fkG_{2k}),
		\end{equation}
		\begin{equation}
			\vom_{2k}^0 - \wt{\vom}_{2k} = \qty\big(\wt{\cL}_3)^k (\vv_0 - \wt{\vv}) +  \sigma_0^{-k} \wt{\sigma}^{-k}\qty(\fkW_{2k}' + \fkW_{2k}),
		\end{equation}
		and
		\begin{equation}\label{diff formula sigma 2k}
			\sigma_{2k}^0 - \wt{\sigma}_{2k} = \qty\big(\wt{\cL}_5)^k (\sigma_0 - \wt{\sigma})  + \sigma_0^{-k} \wt{\sigma}^{-k}\qty(\fkS_{2k}' + \fkS_{2k}),
		\end{equation}
	\end{subequations}
	where $(\fkG_{2k}, \fkW_{2k}, \fkS_{2k})$ and $(\fkG_{2k}', \fkW_{2k}', \fkS_{2k}')$ are respectively the ``high-high'' and ``low-high balanced'' terms discussed previously.
	
	\paragraph*{Step 3: Controlling the remainders}
	As in \eqref{def A*}-\eqref{def B}, define the control parameters for the difference data $(q_0 - \wt{q}, \vv_0 - \wt{\vv}, \sigma_0 - \wt{\sigma})$ restricted in $\Omin$ by
	\begin{equation}
		A_*^\Delta \coloneqq \norm{q_0 - \wt{q}}_{C^{1+\varepsilon_*}(\Omin)} + \norm{(\vv_0 - \wt{\vv}, \sigma_0 - \wt{\sigma})}_{C^{\frac{1}{2}+\varepsilon_*}(\Omin)} + \norm*{\sigma_0^{-1} - \wt{\sigma}^{-1}}_{L^\infty(\Omin)},
	\end{equation}
	and
	\begin{equation}
		B^\Delta \coloneqq \norm{\grad(q_0 - \wt{q})}_{\wtC^{0, \frac{1}{2}}(\Omin)} + \norm{\grad(\vv_0 - \wt{\vv})}_{L^\infty(\Omin)} + \norm{\grad(\sigma_0 - \wt{\sigma})}_{L^\infty(\Omin)}.
	\end{equation}
	Then, Proposition \ref{prop reg op} and the Sobolev embeddings yield that
	\begin{equation}\label{est A B dif}
		A_*^\Delta + B^\Delta \lesssim_M \varepsilon^2.
	\end{equation}
	Since the ``high-high'' terms are at least bilinear in the difference $\scrJ(q_0 - \wt{q}, \vv_0 - \wt{\vv}, \sigma_0 - \wt{\sigma})$ and do not involve the top order terms, one can utilize Lemmas \ref{lem err est A}-\ref{lem err est B} and \eqref{est A B dif} to conclude that
	\begin{equation*}
		\norm{\fkF_{2k}}_{H^{0, \ala}_{q_-}(\Omin)} + \norm{(\fkG_{2k}, \fkW_{2k}, \fkS_{2k})}_{H^{0, \alb}_{q_-}(\Omin)} \lesssim_M \varepsilon^2.
	\end{equation*} 

	The estimates for the ``low-high balanced'' terms can be obtained from similar interpolation and scaling arguments. More precisely, since they are linear in $\scrJ(\wt{q}-q_0, \wt{\vv}-\vv_0, \wt{\sigma}-\sigma_0)$ and contains at least one factor of $\scrJ(\pd^2\wt{q}, \pd\wt{\vv}, \pd\wt{\sigma})$, one has
	\begin{equation*}
		\begin{split}
			&\norm{\fkF_{2k}'}_{H^{0, \ala}_{q_-}(\Omin)} + \norm{(\fkG_{2k}', \fkW_{2k}', \fkS_{2k}')}_{H^{0, \alb}_{q_-}(\Omin)} \\
			&\quad \lesssim_{\wt{A}_*, A_*^\Delta} (B^\Delta + A_*^\Delta \wt{B}) \norm{(\wt{q}, \wt{\vv}, \wt{\sigma})}_{\fkH^{2k-1}_{q_-}(\Omin)} + \wt{B} \norm{(q_0 - \wt{q}, \vv_0 - \wt{\vv}, \sigma_0 - \wt{\sigma})}_{\fkH^{2k-1}_{q_-}(\Omin)},
		\end{split}
	\end{equation*}
	where $\wt{A}_*$ and $\wt{B}$ are the control parameters for $(\wt{q}, \wt{\vv}, \wt{\sigma})$ defined by \eqref{def A*} and \eqref{def B} respectively. Note that $\wt{B}$ is merely bounded, so one can conclude that
	\begin{equation*}
		\begin{split}
			&\norm{\fkF_{2k}'}_{H^{0, \ala}_{q_-}(\Omin)} + \norm{(\fkG_{2k}', \fkW_{2k}', \fkS_{2k}')}_{H^{0, \alb}_{q_-}(\Omin)} \\
			&\quad \lesssim_M \varepsilon^2 + \norm{(q_0 - \wt{q}, \vv_0 - \wt{\vv}, \sigma_0 - \wt{\sigma})}_{\fkH^{2k-1}_{q_-}(\Omin)}.
		\end{split}
	\end{equation*}
	On the other hand, Propositions \ref{prop interp} and \ref{prop reg op} imply that
	\begin{equation*}
		\norm{(q_0 - q_+, \vv_0 - q_+, \sigma_0 - \sigma_+)}_{\fkH^{2k-1}_{q_-}(\Omin)} \lesssim_M 2^{-h_+}.
	\end{equation*}
	Hence, as long as
	\begin{equation}\label{h+ > 4h}
		h_+ > 4h,
	\end{equation}
	it can be derived that
	\begin{equation}
		\begin{split}
			&\norm{\fkF_{2k}'}_{H^{0, \ala}_{q_-}(\Omin)} + \norm{(\fkG_{2k}', \fkW_{2k}', \fkS_{2k}')}_{H^{0, \alb}_{q_-}(\Omin)} \\
		&\quad \lesssim_M \varepsilon^2 + \norm{(q_+ - \wt{q}, \vv_+ - \wt{\vv}, \sigma_+ - \wt{\sigma})}_{\fkH^{2k-1}_{q_-}(\Omin)}.
		\end{split}
	\end{equation}
		
	\paragraph*{Step 4: Comparison of the energies}
	With all these preparations, one can compare the restrictions of $\fkE_{2k, \text{high}}^- (\wt{q}, \wt{\vv}, \wt{\sigma})$ and $\fkE_{2k, \text{high}}(q_0, \vv_0, \sigma_0)$ in $\Omin$. First, the elementary inequality $a^2 \le b^2 - 2a(b-a) $, together with the $L^\infty$-estimates \eqref{dif q om in}-\eqref{dif sigma om in}, imply that
	\begin{equation}\label{comparasion 1}
		\begin{split}
			&\hspace*{-2em} \fkE_{2k, \text{high}}^- (\wt{q}, \wt{\vv}, \wt{\sigma})_{\restriction_{\Omin}} \\
			\le &\, \fkE_{2k, \text{high}}^-(q_0, \vv_0, \sigma_0)_{\restriction_{\Omin}}  \\
			&-2 \qty(\bmqty{\wt{\fks}_{2k} \\ \wt{\fkw}_{2k} \\ \wt{\vom}_{2k} \\ \wt{\sigma}_{2k}},\bmqty{(\wt{\cL}_1)^k (q_0 - \wt{q}) + \sigma_0^{-k} \wt{\sigma}^{-k} \qty(\fkF_{2k}' + \fkF_{2k}) \\ (\wt{\cL}_2)^k (\vv_0 - \wt{\vv}) + \sigma_0^{-k} \wt{\sigma}^{-k}\qty(\fkG_{2k}' + \fkG_{2k}) \\ (\wt{\cL}_3)^k (\vv_0 - \wt{\vv}) +  \sigma_0^{-k} \wt{\sigma}^{-k}\qty(\fkW_{2k}' + \fkW_{2k}) \\ (\wt{\cL}_5)^k (\sigma_0 - \wt{\sigma})  + \sigma_0^{-k} \wt{\sigma}^{-k}\qty(\fkS_{2k}' + \fkS_{2k}) })_{\fkH_{-}(\Omin)}\\
			\le&\, \fkE_{2k, \text{high}}(q_0, \vv_0, \sigma_0) + \Order{M}{1} \norm{(q_+ - \wt{q}, \vv_+ - \wt{\vv}, \sigma_+ - \wt{\sigma})}_{\fkH^{2k-1}_{q_-}(\Omin)} + \Order{M}{\varepsilon}\\
			&-2 \qty({\bmqty{\wt{\fks}_{2k} \\ \wt{\fkw}_{2k} \\ \wt{\vom}_{2k} \\ \wt{\sigma}_{2k}}},{\bmqty{(\wt{\cL}_1)^k (q_0 - \wt{q})  \\ (\wt{\cL}_2)^k (\vv_0 - \wt{\vv})  \\ (\wt{\cL}_3)^k (\vv_0 - \wt{\vv})  \\ (\wt{\cL}_5)^k (\sigma_0 - \wt{\sigma}) }})_{\fkH_{-}(\Omin)},
		\end{split}
	\end{equation}
	here the inner product is defined as
	\begin{equation*}
		\qty(\bmqty{f_1 \\ f_2 \\ f_3 \\ f_4}, \bmqty{g_1 \\ g_3 \\ g_3 \\ g_4})_{\fkH_{-}(\Omin)} \hspace*{-3em}\coloneqq \int_{\Omin} q_-^{\alpha-1} \qty[\abs{f_1 g_1}^2 + \beta q_- \sigma_-^{-1}\abs{f_2 g_2}^2 + \beta q_- \sigma_-^{-1}\abs{f_3 g_3}^2 + q_-\abs{f_4 g_4}^2].
	\end{equation*}
	Indeed, one is able to simplify the inner products. First note that the difference $(\wt{\cL}_1)^k - (\cL_1^-)^k$ induces at least one $(\wt{q}-q_-)$ factor. Furthermore, the Sobolev embeddings, Proposition \ref{prop reg op}, and \eqref{h- kk 3h} yield that
	\begin{equation*}
		\norm{\wt{q}-q_-}_{C^1(\Omin)} \lesssim_M \varepsilon^2.
	\end{equation*}
	Thus, H{\"o}lder's inequality and interpolations imply the following estimate
	\begin{equation*}
		\norm{\qty[(\wt{\cL}_1)^k - (\cL_1^-)^k](q_0 - \wt{q})}_{H^{0, \ala}_{q_-}(\Omin)} \lesssim_M \varepsilon.
	\end{equation*}
	Notice that $q_0 - \wt{q} = (q_0 - q_+) + (q_+ - \wt{q})$. In order to estimate the inner product $\ev{(\cL_1^-)^k(q_0 - q_+), \wt{s}_{2k}}_{H^{0, \ala}_{q_-}(\Omin)}$, one natural way is to use the integration by parts in $\Omega_-$. Indeed, one can first insert a cutoff function $\eta$ with $\spt(\eta) \subset \Omega_0 \cap \Omega_-$ and $\eta \equiv 1 $ in $\Omin$. Observe that, whenever $\spt(\eta) \subset \Omega_-^{[h_2]}$, the arguments in Step 1 yield
	\begin{equation*}
		\ev{(\cL_1^-)^k(q_0 - q_+), \wt{s}_{2k}}_{H^{0, \ala}_{q_-}(\Omin)} = \ev{(\cL_1^-)^k[\eta(q_0 - q_+)],\wt{s}_{2k}}_{H^{0, \ala}_{q_-}(\Omega_-)} + \Order{M}{\varepsilon},
	\end{equation*}
	provided that $h_2 \le h_+ + c$. Indeed, when $h_2 \gg h_+$, the derivatives of $\eta$ will become the main contributions, which is to be avoided. Recall that $\Omega_- \setminus \Omin \supset \Omega_-^{[>h_-(k+1-\varkappa_0)-h-c]}$, whenever
	\begin{equation}\label{h_- h_+ rel}
		h_- (k+1-\varkappa_0) \le h_+ + h + c,
	\end{equation}
	one can derive that	
	\begin{equation*}
		\begin{split}
			&\hspace*{-1em}\ev{(\cL_1^-)^k(q_0 - q_+), \wt{s}_{2k}}_{H^{0, \ala}_{q_-}(\Omin)} \\
			&= \ev{(\cL_1^-)^k[\eta(q_0 - q_+)], \wt{s}_{2k}}_{H^{0, \ala}_{q_-}(\Omega_-)} + \Order{M}{\varepsilon} \\
			&\lesssim \abs{\ev{(\cL_1^-)^{(k-1)}[\eta(q_0 - q_+)], \cL_1^- \wt{s}_{2k}}_{H^{0, \ala}_{q_-}(\Omega_-)}} + \Order{M}{\varepsilon} \\
			&\lesssim \norm{q_0 - q_+}_{H^{2k-2, k-1+\ala}_{q_-}} \cdot \norm{(\wt{q}, \wt{\vv}, \wt{\sigma})}_{\fkH^{2k+2}_{q_-}} + \Order{M}{\varepsilon} \\
			&\lesssim_M \varepsilon + 2^{2h-2h_+}.
		\end{split}
	\end{equation*}
	Thus, \eqref{h+ > 4h} yields that
	\begin{equation*}
		\abs{\ev{(\cL_1^-)^k(q_0 - q_+), \wt{s}_{2k}}_{H^{0, \ala}_{q_-}(\Omin)}} \lesssim_M \varepsilon.
	\end{equation*}
	Similar arguments applied to the other three inner products imply that \eqref{comparasion 1} can be further simplified to 
	\begin{equation*}
		\begin{split}
			&\hspace*{-2em} \fkE_{2k, \text{high}}^- (\wt{q}, \wt{\vv}, \wt{\sigma})_{\restriction_{\Omin}} \\
			\le&\, \fkE_{2k, \text{high}}(q_0, \vv_0, \sigma_0) + \order{M} \norm{(q_+ - \wt{q}, \vv_+ - \wt{\vv}, \sigma_+ - \wt{\sigma})}_{\fkH^{2k-1}_{q_-}(\Omin)} + \Order{M}{\varepsilon}\\
			&-2 \qty({\bmqty{\wt{\fks}_{2k} \\ \wt{\fkw}_{2k} \\ \wt{\vom}_{2k} \\ \wt{\sigma}_{2k}}},{\bmqty{(\cL_1^-)^k (q_+ - \wt{q})  \\ (\cL_2^-)^k (\vv_+ - \wt{\vv})  \\ (\cL_3^-)^k (\vv_+ - \wt{\vv})  \\ (\cL_5^-)^k (\sigma_+ - \wt{\sigma}) }})_{\fkH_{-}(\Omin)}.
		\end{split}
	\end{equation*}
	Next, observe the following relations akin to \eqref{diff formula gu 2k}:
	\begin{subequations}
		\begin{align}
			\wt{\fks}_{2k} &= {\fks}_{2k}^- + \qty(\cL_1^-)^k (\wt{q}-q_-) + (\sigma_-)^{-k}\wt{\sigma}^{-k}\qty(\fkF_{2k}'' + \fkF_{2k}'''), \\
			\wt{\fkw}_{2k} &= \fkw_{2k}^- + \qty(\cL_2^-)^k (\wt{\vv}-\vv_-) + (\sigma_-)^{-k}\wt{\sigma}^{-k}\qty(\fkG_{2k}'' + \fkG_{2k}'''), \\
			\wt{\vom}_{2k} &= {\vom}_{2k}^- + \qty(\cL_3^-)^k (\wt{\vv}-\vv_-) + (\sigma_-)^{-k}\wt{\sigma}^{-k} \qty(\fkW_{2k}'' + \fkW_{2k}'''), \\
			\wt{\sigma}_{2k} &= \sigma_{2k}^- + \qty(\cL_5^-)(\wt{\sigma} - \sigma_-) + (\sigma_-)^{-k}\wt{\sigma}^{-k}\qty(\fkS_{2k}'' + \fkS_{2k}'''),
		\end{align}
	\end{subequations}
	where $(\fkF_{2k}'', \fkG_{2k}'', \fkW_{2k}'', \fkS_{2k}'')$ (the ``high-high'' terms) and $(\fkF_{2k}''', \fkG_{2k}''', \fkW_{2k}''', \fkS_{2k}''')$ (the ``low-high balanced'' terms) are all multilinear forms of $\scrJ(q_-, \vv_-, \sigma_-)$ and $\scrJ(\wt{q}-q_-, \wt{\vv}-\vv_-, \wt{\sigma}-\sigma_-)$. Thus, following the arguments in Step 3, one can bound the ``high-high'' terms by
	\begin{equation*}
		\norm{\fkF_{2k}''}_{H^{0, \ala}_{q_-}(\Omin)} + \norm{(\fkG_{2k}'', \fkW_{2k}'', \fkS_{2k}'')}_{H^{0, \alb}_{q_-}(\Omin)} \lesssim_M \varepsilon^2.
	\end{equation*}
	For the ``low-high balanced'' terms, one can calculate that
	\begin{equation*}
		\begin{split}
			&\abs{\qty({\bmqty{\fkF_{2k}''' \\ \fkG_{2k}''' \\ \fkW_{2k}''' \\ \fkS_{2k}'''}},{\bmqty{(\cL_1^-)^k (q_+ - \wt{q})  \\ (\cL_2^-)^k (\vv_+ - \wt{\vv})  \\ (\cL_3^-)^k (\vv_+ - \wt{\vv})  \\ (\cL_5^-)^k (\sigma_+ - \wt{\sigma}) }})_{\fkH_{-}(\Omin)}} \\
			&\quad \lesssim \abs{\qty({\bmqty{\cL_1^-\fkF_{2k}''' \\ (\cL_2^-)^*\fkG_{2k}''' \\ (\cL_3^-)^*\fkW_{2k}''' \\ \cL_5^-\fkS_{2k}'''}},{\bmqty{(\cL_1^-)^{(k-1)} (q_+ - \wt{q})  \\ (\cL_2^-)^{(k-1)} (\vv_+ - \wt{\vv})  \\ (\cL_3^-)^{(k-1)} (\vv_+ - \wt{\vv})  \\ (\cL_5^-)^{(k-1)} (\sigma_+ - \wt{\sigma}) }})_{\fkH_{-}(\Omega_-)}} + \Order{M}{\varepsilon} \\
			&\quad \lesssim_M \varepsilon + \norm{(q_+ - \wt{q}, \vv_+ - \wt{\vv}, \sigma_+ - \wt{\sigma})}_{\fkH^{2k-2}_{q-}(\Omega_-)} \cdot \norm{(\wt{q}-q_-, \wt{\vv}-\vv_-, \wt{\sigma}-\sigma_-)}_{\fkH^{2k+1}_{q_-}(\Omega_-)} \\
			&\quad \lesssim_M \varepsilon + \varepsilon^2\norm{(\wt{q}-q_-, \wt{\vv}-\vv_-, \wt{\sigma}-\sigma_-)}_{\fkH^{2k+1}_{q_-}(\Omega_-)},
		\end{split}
	\end{equation*}
	where the last inequality follows from \eqref{est q+ - q wt}. Similar arguments yield that
	\begin{equation*}
		\begin{split}
			&\abs{\qty({\bmqty{\fks_{2k}^- \\ \fkw_{2k}^- \\ {\vom}_{2k}^- \\ \sigma_{2k}^-}},{\bmqty{(\cL_1^-)^k (q_+ - \wt{q})  \\ (\cL_2^-)^k (\vv_+ - \wt{\vv})  \\ (\cL_3^-)^k (\vv_+ - \wt{\vv})  \\ (\cL_5^-)^k (\sigma_+ - \wt{\sigma}) }})_{\fkH_{-}(\Omin)}} \\
			&\quad \lesssim \abs{\qty({\bmqty{\cL_1^-\fks_{2k}^- \\ (\cL_2^-)^*\fkw_{2k}^- \\ (\cL_3^-)^*{\vom}_{2k}^- \\ \cL_5^-\sigma_{2k}^-}},{\bmqty{(\cL_1^-)^{(k-1)} (q_+ - \wt{q})  \\ (\cL_2^-)^{(k-1)} (\vv_+ - \wt{\vv})  \\ (\cL_3^-)^{(k-1)} (\vv_+ - \wt{\vv})  \\ (\cL_5^-)^{(k-1)} (\sigma_+ - \wt{\sigma}) }})_{\fkH_{-}(\Omega_-)}} + \Order{M}{\varepsilon} \\
			&\quad \lesssim_M \varepsilon + \norm{(q_+ - \wt{q}, \vv_+ - \wt{\vv}, \sigma_+ - \wt{\sigma})}_{\fkH^{2k-2}_{q-}(\Omega_-)} \cdot \norm{(q_-, \vv_-, \sigma_-)}_{\fkH^{2k+2}_{q_-}(\Omega_-)} \\
			&\quad \lesssim_M \varepsilon.
		\end{split}
	\end{equation*}
	In summary, by noting \eqref{comm L2 L3}-\eqref{L3 L0}, one can conclude from the above arguments and interpolations that
	\begin{equation*}
		\begin{split}
			&\hspace*{-2em} \fkE_{2k, \text{high}}^- (\wt{q}, \wt{\vv}, \wt{\sigma})_{\restriction_{\Omin}} \\
			\le&\, \fkE_{2k, \text{high}}(q_0, \vv_0, \sigma_0) + \Order{M}{1} \norm{(q_+ - \wt{q}, \vv_+ - \wt{\vv}, \sigma_+ - \wt{\sigma})}_{\fkH^{2k-1}_{q_-}(\Omin)} \\
			&+ \Order{M}{\varepsilon^2} \norm{(\wt{q}-q_-, \wt{\vv}-\vv_-, \wt{\sigma}-\sigma_-)}_{\fkH^{2k+1}_{q_-}(\Omin)} + \Order{M}{\varepsilon}\\
			&- 2\qty(\bmqty{(\cL_1^-)^k(q_+ - \wt{q}) \\ (\sigma_-)^{-\frac{1}{2}}(\cL_2^- + \cL_3^-)^k (\vv_+ - \wt{\vv}) \\ (\cL_5^-)^k (\sigma_+ - \wt{\sigma})}, \bmqty{(\cL_1^-)^k (\wt{q}-q_-) \\ (\sigma_-)^{-\frac{1}{2}}(\cL_2^- + \cL_3^-)^k (\wt{\vv}-\vv_-) \\ (\cL_5^-)^k (\wt{\sigma}-\sigma_-)})_{\fkH_{q_-}^0(\Omega_-)},
		\end{split}
	\end{equation*}
	where the inner product is defined as
	\begin{equation*}
		\qty(\bmqty{f_1 \\ f_2 \\ f_3}, \bmqty{g_1 \\ g_2 \\ g_3})_{\fkH^0_{q_-}(\Omega_-)} \hspace*{-3em}\coloneqq \int_{\Omega_-} q_-^{\alpha-1}\qty[f_1 g_1 + \beta q_- (f_2 g_2) + q_- (f_3 g_3)].
	\end{equation*}
	Notice that both commutators
	\begin{equation*}
		\comm{(\sigma_-)^{-\frac{1}{2}}}{(\cL_2^- + \cL_3^-)^k}(\vv_+ - \wt{\vv}) \qand \comm{(\sigma_-)^{-\frac{1}{2}}}{(\cL_2^- + \cL_3^-)^k}(\wt{\vv} - \vv_-)
	\end{equation*}
	are composed of the ``low-high balanced'' terms (possibly multiplying some negative powers of $\sigma_-$, which are harmless in the estimates anyway), so one can further refine that
	\begin{equation}
		\begin{split}
			&\hspace*{-1em} \fkE_{2k, \text{high}}^- (\wt{q}, \wt{\vv}, \wt{\sigma})_{\restriction_{\Omin}} \\
			\le&\, \fkE_{2k, \text{high}}(q_0, \vv_0, \sigma_0) + \Order{M}{1} \norm{(q_+ - \wt{q}, \vv_+ - \wt{\vv}, \sigma_+ - \wt{\sigma})}_{\fkH^{2k-1}_{q_-}(\Omin)} \\
			&+ \Order{M}{\varepsilon^2} \norm{(\wt{q}-q_-, \wt{\vv}-\vv_-, \wt{\sigma}-\sigma_-)}_{\fkH^{2k+1}_{q_-}(\Omin)} + \Order{M}{\varepsilon}\\
			&- 2\qty(\bmqty{(\cL_1^-)^k(q_+ - \wt{q}) \\ (\cL_2^- + \cL_3^-)^k \qty\big[(\sigma_-)^{-\frac{1}{2}}(\vv_+ - \wt{\vv})] \\ (\cL_5^-)^k (\sigma_+ - \wt{\sigma})}, \bmqty{(\cL_1^-)^k (\wt{q}-q_-) \\ (\cL_2^- + \cL_3^-)^k \qty\big[(\sigma_-)^{-\frac{1}{2}}(\wt{\vv}-\vv_-)] \\ (\cL_5^-)^k (\wt{\sigma}-\sigma_-)})_{\fkH^0_{q_-}(\Omega_-)}.
		\end{split}
	\end{equation}
	It follows from the Cauchy-Schwartz inequality that
	\begin{equation*}
		\begin{split}
			&C(M)\norm{(q_+ - \wt{q}, \vv_+ - \wt{\vv}, \sigma_+ - \wt{\sigma})}_{\fkH^{2k-1}_{q_-}(\Omin)} \\
			&\quad \le \frac{\abs{C(M)}^2}{4\epsilon'} 2^{-2h} + \epsilon' 2^{2h} \norm{(q_+ - \wt{q}, \vv_+ - \wt{\vv}, \sigma_+ - \wt{\sigma})}_{\fkH^{2k-1}_{q_-}(\Omin)}^2 \\
			&\quad \lesssim_M \frac{\varepsilon}{\epsilon'} + \epsilon' \varepsilon^{-1} \norm{\qty(q_+ - \wt{q}, (\sigma_-)^{-\frac{1}{2}}(\vv_+ - \wt{\vv}), \sigma_+ - \wt{\sigma})}_{\fkH^{2k-1}_{q_-}(\Omin)}^2
		\end{split}
	\end{equation*}
	here $\epsilon'$ is a small parameter. Similarly, it holds that 
	\begin{equation*}
		\begin{split}
			&C(M)\varepsilon^2 \norm{(\wt{q}-q_-, \wt{\vv}-\vv_-, \wt{\sigma}-\sigma_-)}_{\fkH^{2k+1}_{q_-}(\Omin)} \\
			&\quad \le \frac{\abs{C(M)}^2\varepsilon}{4\epsilon''} + \epsilon'' \varepsilon^3 \norm{(\wt{q}-q_-, \wt{\vv}-\vv_-, \wt{\sigma}-\sigma_-)}_{\fkH^{2k+1}_{q_-}(\Omin)}^2 \\
			&\quad \lesssim_M \frac{\varepsilon}{\epsilon''} + \epsilon'' \varepsilon^3 \norm{\qty(\wt{q}-q_-, (\sigma_-)^{-\frac{1}{2}}(\wt{\vv}-\vv_-), \wt{\sigma}-\sigma_-)}_{\fkH^{2k+1}_{q_-}(\Omin)}^2,
		\end{split}
	\end{equation*}
	where $\epsilon''$ is another small parameter. In view of the above relations, one can control the remainders by utilizing the inner product. More precisely, it follows from \eqref{def chi epsilon}-\eqref{def tilde quantities} that
	\begin{equation}\label{L1 diff trans}
	\begin{split}
			&\ev{(\cL_1^-)^k (q_+ - \wt{q}),\, (\cL_1^-)^k (\wt{q}- q_-)}_{H^{0, \ala}_{q_-}}  \\
			 &\quad=\ev{(\cL_1^-)^k(\textup{Id}-\chi_\varepsilon(\cL_1^-))(q_+ - q_-),\, (\cL_1^-)^k \chi_\varepsilon(\cL_1^-)(q_+ - q_-)}_{H^{0, \ala}_{q_-}} \\
			 &\quad = \ev{\chi_\varepsilon(\cL_1^-)(\textup{Id}-\chi_\varepsilon(\cL_1^-))(\cL_1^-)^k(q_+ - q_-),\, (\cL_1^-)^k (q_+ - q_-)}_{H^{0, \ala}_{q_-}}.
	\end{split}
	\end{equation}
	Note that the choice of $\chi_\varepsilon$ ensures that
	\begin{equation}\label{req chi ep 1}
		t\chi_\varepsilon(t)\cdot[1-\chi_\varepsilon(t)] \gtrsim \varepsilon^{-1}[1-\chi_\varepsilon(t)]^2 + \varepsilon^3 t^2 \abs{\chi_\varepsilon(t)}^2,
	\end{equation}
	which yields
	\begin{equation}\label{req chi ep 2}
		\begin{split}
			&\hspace*{-2em}\ev{(\cL_1^-)^k (q_+ - \wt{q}),\, (\cL_1^-)^k (\wt{q}- q_-)}_{H^{0, \ala}_{q_-}}\\
			\gtrsim &\,\varepsilon^{-1}\norm{q_+ - \wt{q}}_{H^{2k-1, k-\frac{1}{2}+\ala}_{q_-}(\Omega_-)}^2 + \varepsilon^3 \norm{\wt{q}-q_-}_{H^{2k+1, k+\frac{1}{2}+\ala}_{q_-}(\Omega_-)}^2 \\
			 &- C(M)\norm{q_+ - q_-}_{H^{0, \ala}_{q_-}}^2.
		\end{split}
	\end{equation}
	Similar arguments applied to the second and the third part of the inner product lead to the estimate
	\begin{equation}
		\begin{split}
			\fkE_{2k, \text{high}}^- (\wt{q}, \wt{\vv}, \wt{\sigma})_{\restriction_{\Omin}} \le &\,\fkE_{2k, \text{high}}(q_0, \vv_0, \sigma_0) \\
			 &+ C\norm{(q_+ - q_-, \vv_+ - \vv_-, \sigma_+ - \sigma_-)}_{H^{0, \ala}_{q_-}}^2 + \Order{M}{\varepsilon}.
		\end{split}
	\end{equation}
	Furthermore, Proposition \ref{prop reg op} implies that the weighted $L^2$-norm of the difference data is $\Order{M}{\varepsilon}$-small. Namely, counting into the contributions in the removed boundary layer, it holds that
	\begin{equation*}
		\fkE_{2k, \text{high}}^- (\wt{q}, \wt{\vv}, \wt{\sigma}) \le \qty[1+\Order{M}{\varepsilon}]\fkE_{2k}(q_0, \vv_0, \sigma_0).
	\end{equation*}
	The control of $\fkE_{2k, \text{low}}^- (\wt{q}, \wt{\vv}, \wt{\sigma})$ follows routinely from Proposition \ref{prop reg op}. In other words, one obtains
	\begin{equation}\label{fkE2k wt est}
		\fkE_{2k}^- (\wt{q}, \wt{\vv}, \wt{\sigma}) \le \qty[1+\Order{M}{\varepsilon}]\fkE_{2k}(q_0, \vv_0, \sigma_0).
	\end{equation}
	
	\paragraph*{Step 5: Final reductions}
	Here we collect the range of $h_\pm$ from \eqref{rel h- h+}, \eqref{h- kk 3h}, \eqref{h+ > 4h}, and \eqref{h_- h_+ rel} for the sake of convenience:
	\begin{gather*}
		1 \ll h_- < h < h_+, \\
		h_{-}(k+1-\varkappa_0) > h_+ + c_1, \\
		h_{-}(k+1-\varkappa_0) > 3h, \\
		h_+ > 4h,		\\
		h_- (k+1-\varkappa_0) \le h_+ + h + c_2,
	\end{gather*}
	here $c_1$ and $c_2$ are generic constants determined by the primitive regularization operator $\vps$ and the Sobolev embeddings, in particular, they are independent of $h$. In other words, whenever $k$ is large enough, one can simply take
	\begin{equation}
		h_- \approx \frac{1}{2} h \qand h_+ \approx \frac{1}{2}(k-\varkappa_0) h.
	\end{equation}
	Thus, one can rewrite \eqref{def hat data} in a more explicit way:
	\begin{equation}
		(\hat{q}, \hat{\vv}, \hat{\sigma}) = (\wt{q}-2^{-(k-1-\varkappa_0)h}, \wt{\vv}, \wt{\sigma}){\restriction_{\qty{\wt{q} > 2^{-(k-1-\varkappa_0)h}}}}.
	\end{equation}
	On the other hand, the Sobolev embeddings and Proposition \ref{prop reg op} yield that
	\begin{equation*}
		\abs{\wt{q}-q_-}_{L^\infty} \lesssim_M 2^{-2h_-(k+1-\varkappa_0)} \lesssim_M \varepsilon 2^{-h(k-1-\varkappa_0)},
	\end{equation*}
	which implies
	\begin{equation}
		\widehat{\Omega} \coloneqq \qty{\hat{q} > 0} \subset \Omega_-.
	\end{equation}
	Therefore, it follows from the $L^\infty$-estimates for the differences, \eqref{bound wt}, and \eqref{fkE2k wt est} that Proposition \ref{prop reg state} holds, which concludes the proof of Theorem \ref{thm existence}. \\ \qed

	\section{Rough Solutions and Continuation Criteria}\label{sec rough sol}
	This section concerns the existence and continuation criteria in the fractional state spaces, which are achievable by interpolations in the notion of frequency envelopes introduced by Tao \cite{Tao2001} (cf. \cite[\S7]{Ifrim-Tataru2024}). One can also refer to \cite[\S 5]{Ifrim-Tataru2023} for a more general exposition.
		
	\subsection{Frequency Envelopes}\label{sec freq env}
	Let $\Omega$ be a bounded domain with a $C^{1+}$ boundary and $r$ its non-degenerate defining function. Recall that for a parameter $\lambda > 0$ and an integer $j \in \mathbb{N}$, the weighted Sobolev space $H^{j, \frac{\lambda-1}{2}}_r(\Omega)$ is defined by
	\begin{equation*}
		H^{j, \frac{\lambda-1}{2}}_r(\Omega) \coloneqq \Set*{f \in \mathcal{D}'(\Omega) \given \sum_{0 \le \abs{\nu} \le j} \norm\big{r^{\frac{\lambda-1}{2}}\pd^{\nu}f}_{L^2(\Omega)}^2 < \infty },
	\end{equation*}
	and that one can define such weighted Sobolev spaces for non-integral indices by using complex interpolations. Here, we focus on the spaces $H^{2\varkappa, \varkappa+\frac{\lambda-1}{2}}(\Omega)$ for $\varkappa>0$. Thanks to Proposition \ref{prop interp}, it holds that
	\begin{equation*}
		\qty[H^{0, \frac{\lambda-1}{2}}(\Omega), H^{4, 2+\frac{\lambda-1}{2}}(\Omega)]_{\frac{1}{2}} \hookrightarrow H^{2, 1+\frac{\lambda-1}{2}}(\Omega),
	\end{equation*}
	whose left hand side is the complex interpolation space. On the other hand, layer-wise interpolations  and limiting arguments yield that
	\begin{equation*}
		H^{2, 1+\frac{\lambda-1}{2}}(\Omega) \hookrightarrow \qty(H^{0, \frac{\lambda-1}{2}}(\Omega), H^{4, 2+\frac{\lambda-1}{2}}(\Omega))_{\frac{1}{2}, 2; K},
	\end{equation*}
	whose right hand side is the real interpolation space obtained by using the K-method. Since $H^{0, \frac{\lambda-1}{2}}$ and $H^{4, 2+\frac{\lambda-1}{2}}$ are both Hilbert spaces, with $H^{4, 2+\frac{\lambda-1}{2}} \xhookrightarrow[]{\text{dense}} H^{0, \frac{\lambda-1}{2}} $, all interpolation methods will induce the same space (cf. \cite[Corollary 4.37]{Lunardi2018}). Thus, one obtains
	\begin{equation*}
		H^{2, 1+\frac{\lambda-1}{2}}(\Omega) = \qty(H^{0, \frac{\lambda-1}{2}}(\Omega), H^{4, 2+\frac{\lambda-1}{2}}(\Omega))_{\frac{1}{2}, 2; J},
	\end{equation*}
	whose right hand side is the interpolation space using the J-method. In a similar manner, it follows that the spaces $H^{2\varkappa, \varkappa+\frac{\lambda-1}{2}}$ for integer-valued $\varkappa$ given by J-method of interpolations coincide with those defined explicitly.  In particular, there holds the following result (\cite[Lemma 2.5]{Ifrim-Tataru2024}), which is akin to the Littlewood-Paley decomposition.
	\begin{lemma}\label{lem J interp decomp}
		Let $N \ge 1$ be an integer and $0 < \varkappa < N$. Then the space $H^{2\varkappa, \varkappa + \frac{\lambda-1}{2}}(\Omega)$ can be defined as the collection of distributions in $\Omega$ admitting representations
		\begin{equation*}
			u = \sum_{\ell \ge 0} u_\ell,
		\end{equation*}
		for which the sequence $\qty{u_\ell}$ satisfies
		\begin{equation}\label{seq norm}
			\abs{\qty{u_{\ell}}}^2 \coloneqq \sum_{\ell \ge 0} 2^{4\varkappa\ell}\norm{u_\ell}_{H^{0, \frac{\lambda-1}{2}}(\Omega)}^2 + 2^{-4(N-\varkappa)\ell}\norm{u_\ell}_{H^{2N, N+\frac{\lambda-1}{2}}(\Omega)}^2 < \infty.
		\end{equation}
		Moreover, the norm of $H^{2\varkappa, \varkappa+\frac{\lambda-1}{2}}$ can be given equivalently as
		\begin{equation*}
			\norm{u}_{H^{2\varkappa, \varkappa+\frac{\lambda-1}{2}}} \coloneqq \inf \abs{\qty{u_{\ell}}},
		\end{equation*}
		where the infimum is taken from all possible sequential representatives.
	\end{lemma}
	Given the above lemma, norms of the intermediate spaces $H^{2\varkappa, \varkappa+\frac{\lambda-1}{2}}$ can be defined in an explicit way.
	Next, we consider a refinement of Proposition \ref{prop reg op}. For a rough state $(q, \vv, \sigma) \in \bbH^{2\varkappa}$ (here $\varkappa$ is not necessarily being an integer), define the regularizations 
	\begin{equation*}
		(q^h, \vv^h, \sigma^h) \coloneqq \vps_{\le h}(q, \vv, \sigma),
	\end{equation*}
	where $\vps_{\le h}$ is the regularization operator introduced in \S\ref{sec reg op}. In further discussions, the regularized states $(q^h, \vv^h, \sigma^h)$ are restricted to the domain $\Omega_h \coloneqq \qty{q^h > 0}$. In particular, the following proposition holds (see \cite[Proposition 2.11]{Ifrim-Tataru2024}):
	\begin{prop}\label{prop freq env}
		Suppose that $\varkappa > \varkappa_0$, and that $(q, \vv, \sigma) \in \bbH^{2\varkappa}$ is a given state. Then, there exist a family of regularized data $(q^h, \vv^h, \sigma^h) \in \bbH^{2\varkappa}$ and a corresponding slowly varying frequency envelope $\qty{a_h}_{h \ge 0} \in l^2$ satisfying (here $0<\delta\ll 1$ is a generic constant)
		\begin{equation*}
			\frac{a_j}{a_m} \lesssim 2^{\delta\abs{j-m}} \qfor j, m \in \mathbb{N}
		\qand
		 \norm{a_h}_{l^2} \simeq \norm{(q, \vv, \sigma)}_{\bbH^{2\varkappa}}.
		\end{equation*}
		Moreover, there hold the following properties:
		\begin{enumerate}
			\item \textbf{Good Approximations:}
			\begin{equation*}
				(q^h, \vv^h, \sigma^h) \xrightarrow{h\to\infty} (q, \vv, \sigma) \qin C^1_x \times C^{\frac{1}{2}}_x \times C^{\frac{1}{2}}_x,
			\end{equation*}
			and
			\begin{equation}\label{diff est qh q L infty}
				\norm{q^h - q}_{L^\infty(\Omega)} \lesssim 2^{-2(\varkappa+1-\varkappa_0)h}.
			\end{equation}
			\item \textbf{Uniform bounds:}
			\begin{equation}
				\norm{(q^h, \vv^h, \sigma^h)}_{\fkH^{2\varkappa}_{q^h}(\Omega_h)} \lesssim \norm{(q, \vv, \sigma)}_{\fkH^{2\varkappa}_q(\Omega)}.
			\end{equation}
			\item \textbf{Higher order regularities:}
			\begin{equation}
				\norm{(q^h, \vv^h, \sigma^h)}_{\fkH^{2\varkappa + 2\gamma}_{q^h}(\Omega_h)} \lesssim 2^{2\gamma h} a_h \qfor \gamma > 0, \gamma \in \R.
			\end{equation}
			\item \textbf{Low-frequency difference bounds:}
			\begin{equation}\label{diff est qh}
				\norm{(q^{h+1}, \vv^{h+1}, \sigma^{h+1})-(q^h, \vv^h, \sigma^h)}_{\fkH_{\wt{q}}^{0}(\wt{\Omega})} \lesssim 2^{-2\varkappa h}a_h,
			\end{equation}
			for any (compactly supported) defining function $\wt{q}$ satisfying $\abs{\wt{q} - q}_{L^\infty} \ll 2^{-2h}$.
		\end{enumerate}
	\end{prop}
	\begin{remark}
		Denote by $\qty{(s_\ell, \vw_\ell, \zeta_\ell)}_{\ell \ge 0}$ an appropriate decomposition of $(q, \vv, \sigma)$ in the sense of Lemma \ref{lem J interp decomp} (i.e., the $\bbH^{2\varkappa}$ norm of $(q, \vv, \sigma)$ and the norm of $\qty{(s_\ell, \vw_\ell, \zeta_\ell)}_{\ell \ge 0}$ given by \eqref{seq norm} are equivalent), one candidate for the frequency envelope is
		\begin{equation*}
			\begin{split}
				a_j \coloneqq &\max_{m\in\mathbb{N}} 2^{-\delta\abs{j-m}}\qty(2^{2\varkappa m}\norm{(s_m, \vw_m, \zeta_{m})}_{\fkH^0_{q}} + 2^{-2(N-\varkappa)m}\norm{(s_m, \vw_m, \zeta_m)}_{\fkH^{2N}_{q}}) \\
				&+2^{-\delta j} \norm{(q, \vv, \sigma)}_{\fkH^{2\varkappa}_{q}}.
			\end{split}
		\end{equation*}
	\end{remark}
	
	\subsection{Uniform Bounds and Lifespans of Regular Solutions}	
	
	Given a rough initial data $(q_0, \vv_0, \sigma_0) \in \bbH^{2\varkappa}\, (\varkappa > \frac{1}{2} + \varkappa_0)$, the first step is to use the regularization procedures to obtain a series of smooth initial data $(q^h_0, \vv^h_0, \sigma^h_0)$ satisfying Proposition \ref{prop freq env}. Note that each $(q^h_0, \vv^h_0, \sigma^h_0) \in \bbH^{2N}$ for sufficiently large $N \in \mathbb{N}$, so there exist solutions to the compressible Euler equations with initial data $(q^h_0, \vv^h_0, \sigma^h_0)$. An obvious obstacle one encounters at present is that the $\bbH^{2N}$-norms of $ (q^h_0, \vv^h_0, \sigma^h_0)$ are not uniformly bounded for all $h \ge 1$. Thus, one needs to show that there exists a constant $T$ determined by $\norm{(q_0, \vv_0, \sigma_0)}_{\fkH^{2\varkappa}_{q_0}}$ (and, of course, the non-degeneracy constant $c_0$) such that the regular solutions $(q^h, \vv^h, \sigma^h)$ are defined for all $t \in [0, T]$. It suffices to show that
	\begin{equation}\label{uniform bounds h data T}
		\sup_{t \in [0, T]} \norm{(q^h, \vv^h, \sigma^h)}_{\fkH^{2\varkappa}_{q^h}} \lesssim M_0 \coloneqq \norm{(q_0, \vv_0, \sigma_0)}_{\fkH^{2\varkappa}_{q_0}}. 
	\end{equation}
	
	Define the control parameters $(A_{*0}, B_0)$ and $(A^h_*, B^h) $ for the initial data $(q_0, \vv_0, \sigma_0)$ and regular solutions $(q^h, \vv^h, \sigma^h)$ by \eqref{def A*}-\eqref{def B}, respectively. Here one can take $0 < \varepsilon_* < (\varkappa - \varkappa_0 - \frac{1}{2})$ in \eqref{def A*}. Furthermore, denote the non-degeneracy parameters for $q^h$ by
	\begin{equation*}
		c^h(t) \coloneqq \inf_{\overline{x} \in \Gmt^h} \abs{\grad{q^h}}(t, \overline{x}).
	\end{equation*}
	To show the uniform bounds, one may first make the following bootstrap assumptions:
	\begin{equation}
		A^h_* (t) \le 2A_{*0} \qc B^h (t) \le 2B_0 \qc c^h(t) \ge \frac{c_0}{2} \qfor t \in [0, T] \qand h_0 \le h \le h_1,
	\end{equation}
	here $h_0$ is a constant depending only on $M_0$, $h_1$ is finite but arbitrarily large, and $T$ is sufficiently small. The role of $h_0$ is to reduce the errors for the control parameters during the regularization procedures (i.e., to ensure that the bootstrap assumptions hold for the regularized initial data), and that of $h_1$ is to guarantee that one merely needs to handle finitely many quantities at once in the bootstrap arguments. The present goal is to show that one can improve the bootstrap assumptions so long as $T \le T_0$ for some $T_0$ independent of $h_1$. A straightforward candidate for $T_0$ is
	\begin{equation*}
		T_0 \le \frac{1}{C(B_0, c_0)},
	\end{equation*}
	where $C(B_0, c_0)$ is a large constant determined by the control parameter $B_0$ and the non-degeneracy constant $c_0$ for the rough initial data. 
	
	Thanks to Theorem \ref{thm energy est}, the previous bootstrap assumptions imply that the $\bbH^{2l}$ ($l \in \mathbb{N}$) norms of $(q^h, \vv^h, \sigma^h)$ can be controlled by those of the regularized initial data. The difficulty one faces here is that, a priori, the $\bbH^{2\varkappa}$ norms may not propagate for non-integral $\varkappa$. To overcome this shortcoming, one can first derive the estimate by applying Theorem \ref{thm energy est} to the integral indices and Theorem \ref{thm uniqueness} to the difference bounds. 
	
	More precisely, Proposition \ref{prop freq env} and Theorem \ref{thm energy est} yield that
	\begin{equation}\label{bounds h data t}
		\sup_{0 \le t \le T_0}\norm{(q^h, \vv^h, \sigma^h)}_{\fkH^{2(\varkappa + \gamma)}_{q^h}} \lesssim 2^{2\gamma h} a_h \qfor \gamma > 0, \gamma + \varkappa \in \mathbb{N}.
	\end{equation}
	Here one may note that this can be viewed as a high-frequency control, which formally bounds the parts of $(q^h, \vv^h, \sigma^h)$ with dyadic frequency $2^{h}$ or higher. On the other hand, \eqref{diff est qh} yields that
	\begin{equation*}
		\scD \qty{(q_0^{h+1}, \vv_0^{h+1}, \sigma_0^{h+1}); (q_0^h, \vv_0^h, \sigma_0^h)} \lesssim 2^{-4\varkappa h}\abs{a_h}^2,
	\end{equation*}
	where the difference functional $\scD$ is defined by \eqref{diff functional 1}.  Theorem \ref{thm uniqueness} implies that
	\begin{equation}\label{diff functional est qh}
		\sup_{0 \le t \le T_0} \scD \qty{(q^{h+1}, \vv^{h+1}, \sigma^{h+1}); (q^h, \vv^h, \sigma^h)} \lesssim 2^{-4\varkappa h}\abs{a_h}^2.
	\end{equation}
	Therefore, one can apply the high-frequency and the $L^2$-difference bounds to control the $\bbH^{2\varkappa}$ norms of the regular solutions. To do so, one can first consider a telescopic decomposition:
	\begin{equation*}
		(q^h, \vv^h, \sigma^h) = (q^{h_0}, \vv^{h_0}, \sigma^{h_0}) + \sum_{h_0 \le \ell \le h-1} (q^{\ell+1} -q^{\ell}, \vv^{\ell+1}-\vv^\ell, \sigma^{\ell+1} - \sigma^\ell).
	\end{equation*}
	However, one still faces the obstacle that these solutions are defined in different domains. Recall that the regularized quantities are well-defined in an enlarged domain, so one can refine the previous decomposition into	
	\begin{equation}\label{telescope decomp}
		\begin{split}
			(q^h, \vv^h, \sigma^h) = &\,\vps_{\le h} (q^h, \vv^h, \sigma^h) + (\textup{Id} - \vps_{\le h}) (q^h, \vv^h, \sigma^h) \\
			= &\,\vps_{\le h_0} (q^{h_0}, \vv^{h_0}, \sigma^{h_0}) +  (\textup{Id} - \vps_{\le h}) (q^h, \vv^h, \sigma^h) + \\
			&+ \sum_{h_0 \le \ell \le h-1} \qty\Big{\vps_{\le \ell+1}(q^{\ell+1}, \vv^{\ell+1}, \sigma^{\ell + 1}) - \vps_{\le \ell}(q^\ell, \vv^\ell, \sigma^\ell)}.
		\end{split}
	\end{equation}
	The estimates for the second term on the right hand side follows from Proposition \ref{prop reg op}:
	\begin{equation*}
		\norm{ (\textup{Id} - \vps_{\le h}) (q^h, \vv^h, \sigma^h)}_{\fkH_{q^h}^{2N}(\Omega_h)} \lesssim \norm{(q^h, \vv^h, \sigma^h)}_{\fkH^{2N}_{q^h}},
	\end{equation*}
	and
	\begin{equation*}
		\norm{ (\textup{Id} - \vps_{\le h}) (q^h, \vv^h, \sigma^h)}_{\fkH_{q^h}^0(\Omega_h)} \lesssim 2^{-2Nh} \norm{(q^h, \vv^h, \sigma^h)}_{\fkH^{2N}_{q^h}}.
	\end{equation*}
	To estimate the telescopic terms, one can further decompose them into
	\begin{equation}\label{tele decomp 2}
		\begin{split}
			&\vps_{\le \ell+1}(q^{\ell+1}, \vv^{\ell+1}, \sigma^{\ell + 1}) - \vps_{\le \ell}(q^\ell, \vv^\ell, \sigma^\ell) \\
			&\quad= (\vps_{\le\ell+1}-\vps_{\le\ell})(q^{\ell+1}, \vv^{\ell+1}, \sigma^{\ell+1}) + \qty[\vps_{\le\ell}(q^{\ell+1}, \vv^{\ell+1}, \sigma^{\ell + 1}) -\vps_{\le \ell}(q^\ell, \vv^\ell, \sigma^\ell)].
		\end{split}
	\end{equation}
	The estimates for the first term can also be derived from Proposition \ref{prop reg op} that
	\begin{equation*}
		\norm{(\vps_{\le\ell+1}-\vps_{\le\ell})(q^{\ell+1}, \vv^{\ell+1}, \sigma^{\ell+1})}_{\fkH^{2N}_{q^h}(\Omega_h)} \lesssim \norm{(q^{\ell+1}, \vv^{\ell+1}, \sigma^{\ell+1})}_{\fkH^{2N}_{q^{\ell+1}}} \lesssim 2^{2(N-\varkappa)\ell}a_\ell,
	\end{equation*}
	and
	\begin{equation*}
		\norm{(\vps_{\le\ell+1}-\vps_{\le\ell})(q^{\ell+1}, \vv^{\ell+1}, \sigma^{\ell+1})}_{\fkH^{0}_{q^h}(\Omega_h)} \lesssim 2^{-2N\ell}\norm{(q^{\ell+1}, \vv^{\ell+1}, \sigma^{\ell+1})}_{\fkH^{2N}_{q^{\ell+1}}} \lesssim 2^{-2\varkappa\ell}a_\ell.
	\end{equation*}
	To estimate the second term on the right of \eqref{tele decomp 2}, one need first check that the domains of $(q^{\ell+1}, \vv^{\ell+1}, \sigma^{\ell+1})$ and $(q^\ell, \vv^\ell, \sigma^\ell)$ are close enough in the scale $2^{-2\ell}$. Actually, for some parameter $\delta' > 0$, there holds
	\begin{equation}\label{dist Gml Gml+1}
		\dist\qty(\Gamma_\ell, \Gamma_{\ell+1}) \lesssim 2^{-2(1+\delta')\ell}.
	\end{equation}
	Indeed, suppose that $r \coloneqq \dist\qty(\Gamma_\ell, \Gamma_{\ell+1})$. Due to the $L^\infty$ bounds \eqref{diff est qh q L infty}, one can take a small parameter $\epsilon > 0$ independent of $\ell$, so that there exists a ball $B_{\epsilon r} \subset \Omt^\ell \cap \Omt^{\ell+1}$ satisfying
	\begin{equation*}
		q^\ell, q^{\ell +1}, \abs{q^\ell - q^{\ell+1}} \simeq r \qin B_{\epsilon r}.
	\end{equation*}
	Thus, the difference estimate \eqref{diff functional est qh} implies that
	\begin{equation*}
		r^{1+\alpha+d} \lesssim \scD\qty{(q^\ell, \vv^\ell, \sigma^\ell); (q^{\ell+1}, \vv^{\ell+1}, \sigma^{\ell +1})} \lesssim 2^{-4\varkappa\ell}\abs{a_\ell}^2.
	\end{equation*}
	Recall that $2\varkappa_0 = 1 + d + \alpha$ and $\varkappa >\frac{1}{2} + \varkappa_0 $, so \eqref{dist Gml Gml+1} holds. Therefore, the second term on the right of \eqref{tele decomp 2} can be estimated by:
	\begin{equation*}
		\begin{split}
			&\norm{\vps_{\le\ell}(q^{\ell+1}, \vv^{\ell+1}, \sigma^{\ell + 1}) -\vps_{\le \ell}(q^\ell, \vv^\ell, \sigma^\ell)}_{\fkH^{2N}_{q^h}(\Omega_h)} \\
			&\quad \lesssim \norm{(q^{\ell+1}, \vv^{\ell+1}, \sigma^{\ell+1})}_{\fkH^{2N}_{q^{\ell+1}}} + \norm{(q^\ell, \vv^{\ell}, \sigma^\ell)}_{\fkH^{2N}_{q^\ell}} \lesssim 2^{2(N-\varkappa)\ell}a_{\ell}.
		\end{split}
	\end{equation*}
	The $L^2$-difference bounds can be derived from Proposition \ref{prop reg op} and the coercivity arguments for Lemma \ref{lem equiv dist functional}:
	\begin{equation*}
		\begin{split}
			&\norm{\vps_{\le\ell}(q^{\ell+1}, \vv^{\ell+1}, \sigma^{\ell + 1}) -\vps_{\le \ell}(q^\ell, \vv^\ell, \sigma^\ell)}_{\fkH^{0}_{q^h}(\Omega_h)}^2 \\
			&\quad \lesssim  \scD\qty{(q^\ell, \vv^\ell, \sigma^\ell); (q^{\ell+1}, \vv^{\ell+1}, \sigma^{\ell +1})} \lesssim 2^{-4\varkappa\ell}\abs{a_\ell}^2.
		\end{split}
	\end{equation*}
	
	Hence, by plugging the telescopic decomposition \eqref{telescope decomp} into Lemma \ref{lem J interp decomp}, one can conclude from the above bounds and \eqref{bounds h data t} that
	\begin{equation}\label{uniform bound reg sol}
		\sup_{t \in [0, T_0]} \norm{(q^h, \vv^h, \sigma^h)}_{\fkH^{2\varkappa}_{q^h}} \lesssim_{A_{*0}} M_0.
	\end{equation}
	In particular, if $B_0$ is chosen so that $B_0 \gg A_{*0}$ and $B_0 \gg M_0$, the above estimates and the Sobolev embeddings will improve the bootstrap assumptions for $B^h$. 
	
	Next, we turn to the improvements of $A_*$. Since one needs to handle the estimates for pointwise bounds, it would be more convenient to consider the Lagrangian formulation of the compressible Euler equations. Define the flow map $\vb{X}$ by:
	\begin{equation*}
		{\pdv{\vb{X}}{t}} (t, y) = \vv\qty\big(t, \vb{X}(t, y)).
	\end{equation*}
	Denote by $J (t, y) \coloneqq \det(\grad\vb{X})(t, y)$ the Jacobian determinant of the flow map. Then, direct computations yield
	\begin{equation*}
		{\pdv{J}{t}} (t, y)= J(t, y) \cdot (\div\vv)\circ \vb{X}(t, y).
	\end{equation*}
	Therefore, the compressible Euler equations can be rewritten as:
	\begin{equation*}
		\begin{cases*}
			\pd_t \vb{X} = \vv\circ\vb{X}, \\
			\pd_t (q\circ\vb{X}) \cdot J + \beta (q\circ\vb{X}) \cdot \pd_t J = 0, \\
			\pd_t (\vv\circ\vb{X}) + (\sigma\grad{q})\circ\vb{X} = \vb*{0}, \\
			\pd_t (\sigma\circ\vb{X}) = 0.
		\end{cases*}
	\end{equation*}
	The improvements of bounds for $\norm*{q^h}_{C^{\varepsilon_*}}$ and $\norm*{(\vv^h, \sigma^h)}_{C^{\frac{1}{2}+\varepsilon_*}}$ can be deduced from the above evolution equations and the smallness of $T_0$. For the $C^{\varepsilon_*}$-bound of $\grad q^h$, one may first note its evolution equation
	\begin{equation*}
		\Dt \grad{q} + \cL_2[q]\vv = \vb*{0},
	\end{equation*}
	where the operator $\cL_2[q]$ is given by \eqref{def L2}. Then, Lemmas \ref{lem err est A}, \ref{lem J interp decomp}, and the estimate \eqref{uniform bound reg sol} yield that
	\begin{equation*}
		\sup_{0 \le t \le T_0} \norm{\cL_2[q^h]\vv^h}_{H^{2(\varkappa-1), (\varkappa-1)+\alb}_{q^h}(\Omega_h)} \lesssim_{A_{*0}} \sup_{0 \le t \le T_0} \norm{(q^h, \vv^h, \sigma^h)}_{\fkH^{2\varkappa}_{q^h}} \lesssim_{A_{*0}} M_0.
	\end{equation*}
	Note that $\varkappa > (1+\frac{d}{2}+\frac{\alpha}{2})$, so the Sobolev embedding yields that
	\begin{equation*}
		\sup_{0 \le t \le T_0} \norm{\cL_2[q^h]\vv^h}_{C^{\varepsilon_*}} \lesssim_{A_{*0}} M_0.
	\end{equation*}
	Whence, the $C^{\varepsilon_*}$-norm of $\grad{q^h}$ can be improved, as long as $T_0$ is small enough.	Similarly, the improvements of the non-degeneracy parameters $c^h$ can also be established, provided that $T_0$ is sufficiently tiny.
	
	In summary, there exists a constant $T$ depending only on the size of the rough initial data $(q_0, \vv_0, \sigma_0) \in \bbH^{2\varkappa}$, so that the bound \eqref{uniform bounds h data T} holds uniformly for $h \ge 0$. This demonstrates the uniform bounds and lifespans for the family of regular solutions $(q^h, \vv^h, \sigma^h) \in \bbH^{2\varkappa}$.
	
	\subsection{Limiting Solutions: Existence and Continuous Dependence}
	\subsubsection{Existence}
	To show the convergence of $(q^h, \vv^h, \sigma^h)$, as in the previous subsection, one can instead consider the limit
	\begin{equation}\label{conv sol}
		(q, \vv, \sigma) \coloneqq	\lim_{h \to \infty} \vps_{\le h}(q^h, \vv^h, \sigma^h).
	\end{equation}
	Note that the arguments after \eqref{tele decomp 2} also imply that, for any non-degenerate defining function $\wt{q}$ with $\norm{\wt{q}-q^{\ell}}_{L^\infty} \ll 2^{-2\ell}$, the following estimates hold
	\begin{equation}\label{diff general est ql}
		\norm{\vps_{\le \ell+1}(q^{\ell+1}, \vv^{\ell+1}, \sigma^{\ell + 1}) - \vps_{\le \ell}(q^\ell, \vv^\ell, \sigma^\ell)}_{\fkH_{\wt{q}}} \lesssim 2^{-2\varkappa\ell}a_\ell,
	\end{equation}
	and
	\begin{equation}
		\norm{\vps_{\le \ell+1}(q^{\ell+1}, \vv^{\ell+1}, \sigma^{\ell + 1}) - \vps_{\le \ell}(q^\ell, \vv^\ell, \sigma^\ell)}_{\fkH^{2N}_{\wt{q}}} \lesssim 2^{2(N-\varkappa)\ell}a_\ell.
	\end{equation}
	The Sobolev embeddings and Lemma \ref{lem J interp decomp} yield that the limit \eqref{conv sol} exists in the Lipschitz topology. In particular, one can observe further from \eqref{dist Gml Gml+1} that the limit
	\begin{equation*}
			\Omega \coloneqq \lim_{h\to\infty}\Omega_h
	\end{equation*}
	exists, and it has a Lipschitz boundary $\Gamma$. Moreover, there holds
	\begin{equation}\label{est dist gamma gamma h}
		\dist\qty(\Gamma, \Gamma_h) \lesssim 2^{-2(1+\delta')h}.
	\end{equation}
	It follows from the interpolations that the convergence \eqref{conv sol} actually holds in $\fkH^{2\varkappa-}_{q}(\Omega)$. In order to show the convergence in $\fkH^{2\varkappa}_q$, one can also exploit the following telescopic decomposition:
	\begin{equation*}
		(q, \vv, \sigma) = \vps_{\le h_0}(q^{h_0}, \vv^{h_0}, \sigma^{h_0}) + \sum_{\ell \ge h_0} \vps_{\le \ell+1}(q^{\ell+1}, \vv^{\ell+1}, \sigma^{\ell + 1}) - \vps_{\le \ell}(q^\ell, \vv^\ell, \sigma^\ell).
	\end{equation*}
	Then, \eqref{diff general est ql}-\eqref{est dist gamma gamma h} yield that
	\begin{equation*}
		\norm{(q, \vv, \sigma)}_{\fkH^{2\varkappa}_q(\Omega)} \lesssim \norm{a_h}_{l^2},
	\end{equation*}
	and
	\begin{equation}\label{conv sol freq env}
		\norm{(q, \vv, \sigma) - \vps_{\le \ell}(q^\ell, \vv^\ell, \sigma^\ell)}_{\fkH^{2\varkappa}_q(\Omega)} \lesssim \norm{a_{\ge \ell}}_{l^2} \xrightarrow{\ell \to \infty} 0.
	\end{equation}
	To show that the convergence \eqref{conv sol} holds in $\bbH^{2\varkappa}$, it suffices to compare them with the constant sequences $\qty{\vps_{\le m}(q^m, \vv^m, \sigma^m)}$. Observe that the following estimate holds for $\ell \ge m$:
	\begin{equation*}
		\norm{\vps_{\le \ell}(q^\ell, \vv^\ell, \sigma^\ell) -\vps_{\le m}(q^m, \vv^m, \sigma^m)}_{\fkH^{2\varkappa}_q(\Omega)} \lesssim \norm{a_{\ge m}}_{l^2} \xrightarrow{m \to \infty} 0,
	\end{equation*}
	which implies the strong convergence in $\bbH^{2\varkappa}$. Similar arguments yield that the limiting states $(q, \vv, \sigma) \in C\qty([0, T]; \bbH^{2\varkappa})$. Finally, it is routine to check that $(q, \vv, \sigma)$ actually solves the compressible Euler equations.
	
	\subsubsection{Continuous dependence}
	
	Assume that $(q^{(n)}_0, \vv^{(n)}_0, \sigma^{(n)}_0)$ is a sequence of initial data converging to $(q_0, \vv_0, \sigma_0)$ in $\bbH^{2\varkappa}$. To show the convergence of the corresponding solutions, one can first observe that $(q^{(n)}_0, \vv^{(n)}_0, \sigma^{(n)}_0)$ are uniformly bounded in $\bbH^{2\varkappa}$, which yields a uniform lifespan $T$ and a uniform $\bbH^{2\varkappa}$-bound $M$ for the series of solutions.
	
	One strategy to prove the convergence in $C([0, T]; \bbH^{2\varkappa})$ is to exploit the regularizations as in the constructions of rough solutions. More precisely, denote by \begin{equation*}
		(q_0^{(n),h}, \vv_0^{(n),h}, \sigma_0^{(n),h}) \qand (q^h_0, \vv^h_0, \sigma^h_0)
	\end{equation*} 
	the regularized initial data. Then, the definitions and Proposition \ref{prop freq env} yield that
	\begin{equation*}
		(q_0^{(n),h}, \vv_0^{(n),h}, \sigma_0^{(n),h}) \xrightarrow{n\to\infty} (q^h_0, \vv^h_0, \sigma^h_0) \qin C^\infty(\R^d).
	\end{equation*}
	Furthermore, it is clear that $(q_0^{(n),h}, \vv_0^{(n),h}, \sigma_0^{(n),h}) $ and $(q^h_0, \vv^h_0, \sigma^h_0)$ are uniformly bounded in $\bbH^{2\varkappa}$. Thus, the corresponding solutions also have a uniform lifespan and $\bbH^{2\varkappa}$ bound. 
	
	Next, observe that Theorem \ref{thm uniqueness} implies
	\begin{equation*}
		\sup_{0 \le t \le T}\scD\qty{(q^{(n),h}, \vv^{(n),h}, \sigma^{(n),h}); (q^h, \vv^h, \sigma^h)} \xrightarrow{n \to \infty} 0 \qq{uniformly in} h,
	\end{equation*}
	which yields the uniform domain convergence $\Omega_{(n),h} \xrightarrow{n\to\infty}\Omega_{h}$. Moreover, the above estimates also imply the $L^2$-convergence away from a boundary layer of thickness $\simeq 2^{-2h}$. Thus, for each fixed $h$, the uniform-in-$n$ boundedness of $(q^{(n),h}, \vv^{(n),h}, \sigma^{(n),h}) $ in $\bbH^{2j}$, the $L^2$-convergence, the interpolations, and the Sobolev embeddings yield that
	\begin{equation}\label{C infty convergence}
		\vps_{\le h}(q^{(n),h}, \vv^{(n),h}, \sigma^{(n),h}) \xrightarrow{n\to\infty} \vps_{\le h}(q^{h}, \vv^{h}, \sigma^{h}) \qin C^\infty(\R^d)
	\end{equation}
	for every fixed $h$.	
	
	Denote by $\{a_h^{(n)}\}$ and $\qty{a_h}$ the corresponding frequency envelopes for the initial data. Then, it follows from \eqref{conv sol freq env} that
	\begin{equation}\label{H2k approx}
		\norm{(q^{(n)}, \vv^{(n)}, \sigma^{(n)}) - \vps_{\le h}(q^{(n),h}, \vv^{(n),h}, \sigma^{(n),h}) }_{\fkH^{2\varkappa}_{q^{(n)}}} \lesssim \norm{a_{\ge h}^{(n)}}_{l^2}.
	\end{equation}
	In order to conclude the proof, it suffices, for every given $\epsilon > 0$, to choose frequency envelopes so that
	\begin{equation}\label{conv freq enve}
		\limsup_{h \to \infty} \sup_{n} \norm{a_{\ge h}^{(n)}}_{l^2} < \epsilon.
	\end{equation}
	Indeed, one can observe that, for each $\epsilon > 0$, there exists a large integer $L$ so that
	\begin{equation*}
		\norm{(q_0^{(n)}, \vv_0^{(n)}, \sigma_0^{(n)}) - \vps_{\le L}(q_0^{(n)}, \vv_0^{(n)}, \sigma_0^{(n)}) }_{\fkH^{2\varkappa}_{q_0^{(n)}}} \ll \epsilon,
	\end{equation*}
	which can be achieved due to the uniform boundedness of $(q_0^{(n)}, \vv_0^{(n)}, \sigma_0^{(n)})$ in $\bbH^{2\varkappa}$. Thus, one can simply let $\{a_h^{(n)}\}$ be the sum of the frequency envelopes of $\vps_{\le L}(q_0^{(n)}, \vv_0^{(n)}, \sigma_0^{(n)})$ and those of $\qty[\vps_{\le L}(q_0^{(n)}, \vv_0^{(n)}, \sigma_0^{(n)}) - (q_0^{(n)}, \vv_0^{(n)}, \sigma_0^{(n)})]$. The smoothness and uniform boundedness of $\vps_{\le L}(q_0^{(n)}, \vv_0^{(n)}, \sigma_0^{(n)})$ yields \eqref{conv freq enve}.
	
	In summary, the continuous dependence on the initial data follows from \eqref{C infty convergence}-\eqref{conv freq enve}.
	
	\subsection{Lifespans of Rough Solutions}
	Now, we turn to the continuation criteria for solutions in $\bbH^{2\varkappa}$. Given an initial data $(q_0, \vv_0, \sigma_0) \in  \bbH^{2\varkappa}$, let $(q, \vv, \sigma)$ be the corresponding solutions in a time interval $[0, T)$. Assume that
	\begin{equation*}
		\Theta \coloneqq \sup_{0 \le t < T} A_*(t) + \int_0^T B(t) \dd{t} < \infty; \qand c(t) \ge c_0 > 0 \qfor t \in [0, T),
	\end{equation*}
	where the control parameters $A_*$ and $B$ are defined by \eqref{def A*} and \eqref{def B} respectively, and $c(t) \coloneqq \inf_{\Gmt} \abs{\grad{q}}$ is the non-degeneracy constant. Thanks to the local well-posedness result, it suffices to show that
	\begin{equation*}
		\sup_{t\in[0, T)}\norm{(q, \vv, \sigma)}_{\fkH^{2\varkappa}_{q}} < \infty.
	\end{equation*}
	Akin to the previous arguments, one can consider the regularized initial data $(q^h_0, \vv^h_0, \sigma^h_0) \in \bbH^{2\varkappa}$ and the corresponding solutions $(q^h, \vv^h, \sigma^h)$ defined on time intervals $[0, T^h)$. The continuous dependence result yields that
	\begin{equation*}
		\liminf T^h \ge T \qand (q^h, \vv^h, \sigma^h) \to (q, \vv, \sigma) \text{ on } [0, T).
	\end{equation*}
	However, there is no uniform bounds for the control parameters $B^h$. One can still utilize bootstrap arguments. Denote by 
	\begin{equation*}
		\mathscr{A} \coloneqq \sup_{t\in [0, T)} A_*(t).
	\end{equation*}
	One can make the following bootstrap assumptions on a time interval $[0, T_0]$ with $T_0 < T$:
	\begin{equation}\label{bot ass}
		\int_0^{T_0} B^h (t) \dd{t} < 4C_1(\mathscr{A})\Theta \qc A^h_*(t) \le 4P(\Theta), \qand c^h(t) \ge \frac{c_0}{2} \qfor h \ge h_0,
	\end{equation} 
	where $P$ is a large polynomial and $h_0$ is a large constant.
	
	In order to improve the above assumptions, first observe that the bound \eqref{est dist gamma gamma h} still holds for the present case. Therefore, $\vps_{\le h}(q, \vv, \sigma)$ and $\vps_{\le h}(q^h, \vv^h, \sigma^h)$ are both well-defined in the union $\Omega\cup\Omega_h$. In particular, the definition of $\vps_{\le h}$ implies that
	\begin{equation}\label{point est psi h q}
		\norm{\grad\vps_{\le h}(q, \vv, \sigma)}_{\wtC^{0, \frac{1}{2}}\times L^\infty \times L^\infty} \le C_1 B.
	\end{equation}
	Next, one can compare the $B$-parameters for $\vps_{\le h}(q, \vv, \sigma)$ and $\vps_{\le h}(q^h, \vv^h, \sigma^h)$ by invoking a telescopic decomposition:
	\begin{equation*}
		\vps_{\le h}(q, \vv, \sigma) - \vps_{\le h}(q^h, \vv^h, \sigma^h) = \sum_{\ell \ge h}\vps_{\le h}(q^{\ell+1}, \vv^{\ell+1}, \sigma^{\ell+1}) - \vps_{\le h}(q^\ell, \vv^\ell, \sigma^\ell).
	\end{equation*}
	It can be deduced from Proposition \ref{prop freq env} and Theorem \ref{thm uniqueness} that
	\begin{equation*}
		\norm{\vps_{\le h}(q^{\ell+1}, \vv^{\ell+1}, \sigma^{\ell+1}) - \vps_{\le h}(q^\ell, \vv^\ell, \sigma^\ell)}_{\fkH^{2m}_{q^h}} \lesssim_{\Theta} 2^{-2\varkappa\ell}2^{2mh}a_{\ell} \qfor m \ge 0.
	\end{equation*}
	Therefore, direct summations lead to
	\begin{equation*}
		\norm{\vps_{\le h}(q, \vv, \sigma) - \vps_{\le h}(q^h, \vv^h, \sigma^h)}_{\fkH^{2m}_{q^h}} \lesssim_{\Theta} 2^{-2\varkappa h}2^{2mh}a_{h},
	\end{equation*}
	which, together with the Sobolev embeddings, yield that
	\begin{equation}\label{diff est psi q psi qh}
		\norm{\grad[\vps_{\le h}(q, \vv, \sigma) - \vps_{\le h}(q^h, \vv^h, \sigma^h)]}_{\wtC^{0, \frac{1}{2}}\times L^\infty \times L^\infty} \lesssim_{\Theta, M_0} 2^{-\delta_* h},
	\end{equation}
	where $\delta_*> 0$ is a generic constant determined by $\varkappa$ and $\varkappa_0$.
	On the other hand, it follows from Theorem \ref{thm energy est} and Proposition \ref{prop freq env} that
	\begin{equation*}
		\norm{(q^h, \vv^h, \sigma^h) - \vps_{\le h}(q^h, \vv^h, \sigma^h)}_{\fkH^{2m}_{q^h}} \lesssim a_h 2^{-2\varkappa h}2^{2mh} \qfor m \ge 0.
	\end{equation*}
	In particular, the Sobolev embeddings imply that
	\begin{equation}\label{diff est psi qh qh}
		\norm{\grad[(q^h, \vv^h, \sigma^h) - \vps_{\le h}(q^h, \vv^h, \sigma^h)]}_{\wtC^{0, \frac{1}{2}}\times L^\infty \times L^\infty} \lesssim_{\Theta, M_0} 2^{-\delta_* h}.
	\end{equation}
	The combination of \eqref{point est psi h q}-\eqref{diff est psi qh qh} yields that
	\begin{equation*}
		B^h \le C_1 B + C_2(\Theta, M_0) 2^{-\delta_* h}.
	\end{equation*}
	Thus, by taking $h_0 \gg 1$, one can derive that
	\begin{equation*}
		\int_0^{T_0} B^h(t) \dd{t} < 2C_1\Theta \qfor h \ge h_0,
	\end{equation*}
	which improves the first assumption in \eqref{bot ass}. Indeed, similar arguments combined with different Sobolev's embeddings imply that
	\begin{equation*}
		A_*^h \lesssim P(\Theta)A + C_2 2^{-\delta_* h} \qand c^h(t) \ge c(t) - C_2 2^{-\delta_* h}.
	\end{equation*}
	The largeness of $h_0$ concludes the improvements of \eqref{bot ass}.
	
	The bootstrap arguments yield that the solutions $(q^h, \vv^h, \sigma^h)$ can be extended to some $T_1 > T$ for all $h \ge h_0$. The continuous dependence implies that the same extension property holds for $(q, \vv, \sigma)$, which demonstrates the continuation criterion.
	\\	\qed
	
	\section*{Acknowledgments}
	Liu's research is supported by the UM Postdoctoral Fellow scheme under the UM Talent Programme at the University of Macau. Liu extends gratitude to Prof. Changfeng Gui at the University of Macau for his encouragements and supports.
	
	Luo's research is supported by a grant from the Research Grants Council of the Hong Kong Special Administrative Region, China (Project No. 11310023).
	
	
	\subsection*{Data availability}
	This manuscript has no associated data.
	
	\subsection*{Conflict of interest}
	The authors have no conflict of interest to disclose.

	\fancyhead[RO,LE]{\sc{References}}
	
	\bibliographystyle{amsplain0}
	{\small \bibliography{ref}}
\end{document}